\documentclass{article}

\usepackage{amsmath,amsfonts,stmaryrd,amssymb} 

\usepackage{enumerate} 

\usepackage[ruled]{algorithm2e} 

\usepackage[framemethod=tikz]{mdframed} 

\usepackage{listings} 
\usepackage{geometry} 

\usepackage[utf8]{inputenc} 
\usepackage[T1]{fontenc} 

\usepackage{XCharter} 

\mdfdefinestyle{commandline}{
	leftmargin=10pt,
	rightmargin=10pt,
	innerleftmargin=15pt,
	middlelinecolor=black!50!white,
	middlelinewidth=3pt,
	frametitlerule=false,
	backgroundcolor=black!5!white,
	frametitle={Command Line},
	frametitlefont={\normalfont\sffamily\color{white}\hspace{-1em}},
	frametitlebackgroundcolor=black!50!white,
	nobreak,
}

\mdfdefinestyle{file}{
	innertopmargin=4.0\baselineskip,
	innerbottommargin=3.5\baselineskip,
	topline=false, bottomline=false,
	leftline=false, rightline=false,
	leftmargin=2cm,
	rightmargin=2cm,
	singleextra={%
		\draw[fill=black!10!white](P)++(0,-1.2em)rectangle(P-|O);
		\node[anchor=north west]
		at(P-|O){\ttfamily\mdfilename};
		\def\l{3em}
		\draw(O-|P)++(-\l,0)--++(\l,\l)--(P)--(P-|O)--(O)--cycle;
		\draw(O-|P)++(-\l,0)--++(0,\l)--++(\l,0);
	},
	nobreak,
}

\mdfdefinestyle{question}{
	innertopmargin=3.5\baselineskip,
	innerbottommargin=3.0\baselineskip,
	roundcorner=5pt,
	nobreak,
	singleextra={%
		\draw(P-|O)node[xshift=1em,anchor=west,fill=white,draw,rounded corners=5pt]{%
		Question \theQuestion\questionTitle};
	},
}

\newcounter{Question} 

\mdfdefinestyle{warning}{
	topline=false, bottomline=false,
	leftline=false, rightline=false,
	nobreak,
	singleextra={%
		\draw(P-|O)++(-0.5em,0)node(tmp1){};
		\draw(P-|O)++(0.5em,0)node(tmp2){};
		\fill[black,rotate around={45:(P-|O)}](tmp1)rectangle(tmp2);
		\node at(P-|O){\color{white}\scriptsize\bf !};
		\draw[very thick](P-|O)++(0,-1em)--(O);
	}
}

\mdfdefinestyle{info}{%
	topline=false, bottomline=false,
	leftline=false, rightline=false,
	nobreak,
	singleextra={%
		\fill[black](P-|O)circle[radius=0.4em];
		\node at(P-|O){\color{white}\scriptsize\bf i};
		\draw[very thick](P-|O)++(0,-0.8em)--(O);
	}
}

\usepackage{indentfirst}
\usepackage{amsmath}
\usepackage[utf8]{inputenc}
\usepackage[english]{babel}
\usepackage{amsthm}
\usepackage{bbm}
\usepackage{bm}
\usepackage{dsfont}
\usepackage{blindtext}
\usepackage{amssymb}
\usepackage{extarrows}
\usepackage{authblk}
\newtheorem{theorem}{Theorem}[section]
\newtheorem{corollary}{Corollary}[theorem]
\newtheorem{proposition}{Proposition}[theorem]
\newtheorem{lemma}[theorem]{Lemma}
\theoremstyle{remark}
\newtheorem*{remark}{Remark}
\theoremstyle{definition}
\newtheorem{definition}{Definition}[section]

\title{Approximation schemes for McKean-Vlasov and Boltzmann type equations (error analysis in total variation distance)}
\author[*]{Yifeng Qin}

\affil[*]{Universit\'e Gustave Eiffel, LAMA (UMR CNRS, UPEMLV, UPEC), MathRisk INRIA, F-77454 Marne-la-Vall\'ee, France.
Email address: gzevonqin@gmail.com}
\date{Jan. 2023}

\begin{document}

\maketitle

\textbf{Abstract}
We deal with  Mckean-Vlasov and Boltzmann type  jump equations. This means that the coefficients of the stochastic equation depend on the law of the solution, and the equation is driven by a Poisson point measure with intensity measure which depends on the law of the solution as well. In $\cite{ref1}$, Alfonsi and Bally have proved that under some suitable conditions, the solution $X_t$ of such equation exists and is unique. One also proves that $X_t$ is the probabilistic interpretation of an analytical weak equation. Moreover, the Euler scheme $X_t^{\mathcal{P}}$ of this equation  converges to $X_t$ in Wasserstein distance. In this paper, under more restricted assumptions, we show that the Euler scheme $X_t^{\mathcal{P}}$ converges to $X_t$ in total variation distance and $X_t$ has a smooth density (which is a function solution of the analytical weak equation). On the other hand, in view of simulation, we use a truncated Euler scheme $X^{\mathcal{P},M}_t$ which has a finite numbers of jumps in any compact interval. We prove that $X^{\mathcal{P},M}_{t}$ also converges to $X_t$ in total variation distance. Finally, we give an algorithm based on a particle system associated to $X^{\mathcal{P},M}_t$ in order to approximate the density of the law of $X_t$. Complete estimates of the error are obtained.

\textbf{Key words:}
Mckean-Vlasov equation, Boltzmann equation, Malliavin calculus, Total variation distance, Wasserstein distance, Particle system

\tableofcontents

\section{Introduction}
In this paper, we consider a $d-$dimensional Mckean-Vlasov and Boltzmann type jump equation as follows. \begin{eqnarray}
X_{t}&=&X_0+\int_{0}^{t}b(r,X_{r},\rho_{r})dr +\int_{0}^{t}\int_{\mathbb{R}^d\times\mathbb{R}^d}c(r,v,z,X_{r-},\rho_{r-})N_{\rho_{r-}}(dv,dz,dr), \label{Intro1}
\end{eqnarray}
where $\rho_{t}(dv)=\mathbb{P}(X_{t}\in dv)$ is the law of $X_{t}$, $t\in[0,T]$, $N_{\rho_{t}}$ is a Poisson point measure on the state space $\mathbb{R}^d\times\mathbb{R}^d$ with intensity measure $\rho_{t}(dv)\mu(dz)dr$, $\mu$ is a positive $\sigma$-finite measure on $\mathbb{R}^d$, $X_0$ is the initial random variable independent of the Poisson point measure $N_{\rho_{t}}$,    and $b,c$ are functions which verify some regularity and ellipticity conditions (see \textbf{Hypotheses 2.1$\sim$2.4} in Section 2.2  for precise statements). In particular, we assume that for every multi-indices $\beta_1,\beta_2$, there exists a non-negative  function $\bar{c}:\mathbb{R}^d\rightarrow\mathbb{R}_{+}$ such that 
\[
\vert {c}(r,v,z,x,\rho)\vert+\vert \partial_z^{\beta_2}\partial_x^{\beta_1} {c}(r,v,z,x,\rho)\vert\leq\bar{c}(z),
\]with $\int_{\mathbb{R}^d}\vert \bar{c}(z)\vert^p\mu(dz)<\infty,\ \forall p\geq1.$ We  also assume that  there exists a non-negative function $\underline{c}:\mathbb{R}^d\rightarrow\mathbb{R}_{+}$ such that for every $ \zeta\in\mathbb{R}^d$,
\[
\sum_{j=1}^d\langle\partial_{z_j}{c}(r,v,z,x,\rho),\zeta\rangle^{2}\geq \underline{c}%
(z)\vert\zeta\vert^2.
\] We remark that we use the notations from $\cite{ref8}$ and we refer to $\cite{ref51}$, $\cite{ref6}$,  $\cite{ref8}$,  $\cite{ref12}$, $\cite{ref13}$, $\cite{ref36}$ and $\cite{ref37}$ for the basic theory of the classical jump equations.  We stress that our equation is a more general kind of jump equation (than the classical one) in the following sense. The coefficients $b$ and $c$ depend on the law of the solution, so our equation is of Mckean-Vlasov type. One can see for example $\cite{ref7}$ for a mathematical approach to  this kind of equation and see $\cite{ref44}$, $\cite{ref43}$, $\cite{ref38}$, $\cite{ref42}$, $\cite{ref41}$, $\cite{ref39}$ and $\cite{ref40}$  for the approximation schemes of a Mckean-Vlasov equation. Moreover, the intensity of the Poisson point measure $N_{\rho_{t}}$ depends on the law of the solution as well, so our equation is also of Boltzmann type. The probabilistic approach to the Boltzmann equation is initiated by Tanaka in $\cite{ref20}$, $\cite{ref21}$, and followed by many others in $\cite{ref23}$, $\cite{ref26}$, $\cite{ref29}$, $\cite{ref30}$, $\cite{ref14}$,   $\cite{ref31}$ and $\cite{ref32}$ for example. One can also see $ \cite{ref27}$ and $\cite{ref28}$ for the analytical Boltzmann equation and $\cite{ref25}$ for the physical background. Recently, there is also some work on inhomogeneous Boltzmann equations (see for instance  $\cite{ref33}$, $\cite{ref34}$ and $\cite{ref35}$).  We have to mention however that our equation (\ref{Intro1}) does not cover the general physical Boltzmann equation for the following reason. In that equation, the intensity of the jumps $\mu(dz)$ is replaced by $\gamma(r,v,z,x,\rho)\mu(dz)$ which depends on the position $x=X_{r-}$ of the solution of the equation. At least at this time, we are not able to include this case in our study. The simplified model that we treat in our paper corresponds to Maxwell molecules (see  $\cite{ref29}$ for example).

Now we construct the Euler scheme. For any partition $\mathcal{P}=\{0=r_0<r_1<\cdots<r_{n-1}<r_n=T\}$ of the time interval $[0,T]$, we define $\tau(r)=r_k$ when $r\in[r_k,r_{k+1})$, and we consider the equation
\begin{eqnarray}
X^{\mathcal{P}}_{t}&=&X_0+\int_{0}^{t}b(\tau(r),X^{\mathcal{P}}_{\tau(r)},\rho^{\mathcal{P}}_{\tau(r)})dr  \nonumber\\
&+&\int_{0}^{t}\int_{\mathbb{R}^d\times\mathbb{R}^d}{c}(\tau(r),v,z,X^{\mathcal{P}}_{\tau(r)-},\rho^{\mathcal{P}}_{\tau(r)-})N_{\rho^{\mathcal{P}}_{\tau(r)-}}(dv,dz,dr), \label{IntroEul}
\end{eqnarray}
where $\rho^{\mathcal{P}}_{t}$ is the law of $X^{\mathcal{P}}_{t}$,  and $N_{\rho^{\mathcal{P}}_{t}}(dv,dz,dr)$ is a Poisson point measure with intensity  $\rho^{\mathcal{P}}_{t}(dv)\mu(dz)dr$, independent of $X_0$. We remark that for the classical jump equations (the coefficients and the Poisson point measures do not depend on the law of the solution), there is a huge amount of work on the convergence of the Euler scheme. One can see for example $\cite{ref50}$, $\cite{ref48}$, $\cite{ref45}$, $\cite{ref47}$, $\cite{ref46}$,  $\cite{ref49}$,   $\cite{ref16}$ and the references therein. For the equation (\ref{Intro1}), $\cite{ref1}$ has proved recently that under some regularity conditions on the coefficients $b$ and $c$, the solution of the equation (\ref{Intro1}) exists and is unique, and $X_t$ is the probabilistic interpretation of the following analytical weak equation. \begin{eqnarray}
&&\forall\phi\in C_b^1(\mathbb{R}^d) (the\ space\ of\ differentiable\ and\ bounded\ functions\ with\ bounded\ derivatives),\nonumber\\ &&\int_{\mathbb{R}^d}\phi(x)\rho_{t}(dx)=\int_{\mathbb{R}^d}\phi(x)\rho_{0}(dx)+\int_{0}^{t}\int_{\mathbb{R}^d}\langle b(r,x,\rho_{r}),\nabla\phi(x)\rangle\rho_{r}(dx)dr  \nonumber\\
&&+\int_{0}^{t}\int_{\mathbb{R}^d\times\mathbb{R}^d}\rho_{r}(dx)\rho_{r}(dv)\int_{\mathbb{R}^d}(\phi(x+c(r,v,z,x,\rho_{r}))-\phi(x))\mu(dz)dr. \label{Intro2}
\end{eqnarray} Moreover, $\cite{ref1}$ has proved that the Euler scheme $X^{\mathcal{P}}_{t}$ converges to $X_t$ in Wasserstein distance (of order 1) $W_1$.  In our paper, under supplementary hypotheses, we prove a stronger result. We prove (see \textbf{Theorem 2.1}) that  the Euler scheme $X^{\mathcal{P}}_{t}$ converges to $X_t$ in total variation distance: for any $\varepsilon>0$, there exists a constant $C$  such that
\begin{eqnarray}d_{TV}(X^{\mathcal{P}}_{t},X_t)\leq C\vert\mathcal{P}\vert^{1-\varepsilon}\rightarrow0, \label{IntroTV1}\end{eqnarray}as $\vert\mathcal{P}\vert\rightarrow0$, with $\vert\mathcal{P}\vert:=\max\limits_{k\in\{0,\cdots,n-1\}}(r_{k+1}-r_k)$. We also show that the law of $X_t$ has a smooth density $p_t(x)$, which is a function solution of the analytical weak equation (\ref{Intro2}).

Since we have infinite numbers of jumps (due to \textbf{Hypothesis 2.4} in Section 2.2), we have $\mu(\mathbb{R}^d)=\infty$. In view of simulation, we need to work with a truncated Poisson point measure, which has a finite number of jumps in any compact time interval. For $M\in\mathbb{N}$, we denote $B_M=\{z\in\mathbb{R}^d:\vert z\vert\leq M\}$, $c_M(r,v,z,x,\rho):=c(r,v,z,x,\rho)\mathbbm{1}_{B_M}(z)$ and $a^M_{T}:=\sqrt{T\int_{\{\vert z\vert>M\}}\underline{c}(z)  \mu(dz)}$. Now we cancel the jumps of size $\vert z\vert>M$ and we replace them by a Gaussian random variable. 
\begin{eqnarray}
X^{\mathcal{P},M}_{t}&=&X_0+a^M_{T}\Delta+\int_{0}^{t}b(\tau(r),X^{\mathcal{P},M}_{\tau(r)},\rho^{\mathcal{P},M}_{\tau(r)})dr \nonumber \\
&+&\int_{0}^{t}\int_{\mathbb{R}^d\times \mathbb{R}^d}{c}_M(\tau(r),v,z,X^{\mathcal{P},M}_{\tau(r)-},\rho^{\mathcal{P},M}_{\tau(r)-})N_{\rho^{\mathcal{P},M}_{\tau(r)-}}(dv,dz,dr), \label{Intro3}
\end{eqnarray}where $\rho^{\mathcal{P},M}_{t}$ is the law of $X^{\mathcal{P},M}_{t}$, $N_{\rho^{\mathcal{P},M}_{t}}(dv,dz,dr)$ is a Poisson point measure independent of $X_0$ with intensity  $\rho^{\mathcal{P},M}_{t}(dv)\mu(dz)dr$, $\Delta$ is a $d-$dimensional standard Gaussian random variable independent of $X_0$ and of $N_{\rho^{\mathcal{P},M}_{t}}$.  We prove (see \textbf{Theorem 2.2}) that  $X^{\mathcal{P},M}_{t}$ converges to $X_t$ in total variation distance: for any $\varepsilon>0$, there exists a constant $C$ such that 
\begin{eqnarray}d_{TV}(X^{\mathcal{P},M}_{t},X_t)\leq C(\sqrt{\varepsilon_M}+\vert\mathcal{P}\vert)^{1-\varepsilon}\rightarrow0,\label{IntroTV2}\end{eqnarray}as $\vert\mathcal{P}\vert\rightarrow0$ and $M\rightarrow\infty$, with $\varepsilon_M:=\int_{\{\vert z\vert>M\}}\vert\bar{c}(z)\vert^2\mu(dz)+\vert\int_{\{\vert z\vert>M\}}\bar{c}(z)\mu(dz)\vert^2$. Moreover, the law of $X^{\mathcal{P},M}_{t}$ has a smooth density. 

In order to construct an approximation scheme which is appropriate for simulation, we need to compute $\rho^{\mathcal{P},M}_{t}$ as well, so we use the following particle system.
We take an initial vector $(X^1_0,\cdots,X^N_0)$ with components which are independent and identically distributed with common law $\rho_0$ (which is the law of $X_0$), and  $(\Delta^1,\cdots,\Delta^N)$ which is a $N\times d-$dimensional standard Gaussian random variable independent of $(X^1_0,\cdots,X^N_0)$. Then we construct the particle system $\overrightarrow{\bm{X}}^{\mathcal{P},M}_{t}=(X^{\mathcal{P},M,1}_{t},\cdots,X^{\mathcal{P},M,N}_{t})$:
\begin{eqnarray}
X^{\mathcal{P},M,i}_{t}&=&X_0^i+a^M_{T}\Delta^i+\int_{0}^{t}b(\tau(r),X^{\mathcal{P},M,i}_{\tau(r)},\widehat{\rho}(\overrightarrow{\bm{X}}^{\mathcal{P},M}_{\tau(r)}))dr  \nonumber\\
&+&\int_{0}^{t}\int_{\mathbb{R}^d\times \mathbb{R}^d}{c}_M(\tau(r),v,z,X^{\mathcal{P},M,i}_{\tau(r)-},\widehat{\rho}(\overrightarrow{\bm{X}}^{\mathcal{P},M}_{\tau(r)-}))N^i_{\widehat{\rho}(\overrightarrow{\bm{X}}^{\mathcal{P},M}_{\tau(r)-})}(dv,dz,dr),\ i=1,\cdots,N,\label{Intropart}
\end{eqnarray}where \begin{eqnarray*}
\widehat{\rho}(\overrightarrow{\bm{X}}^{\mathcal{P},M}_{t})(dv)=\frac{1}{N}\sum_{i=1}^N\delta_{X^{\mathcal{P},M,i}_{t}}(dv)
\end{eqnarray*}is the empirical measure of $\overrightarrow{\bm{X}}^{\mathcal{P},M}_{t}$ (with $\delta_x(dv)$ the Dirac measure), $N^i_{\widehat{\rho}(\overrightarrow{\bm{X}}^{\mathcal{P},M}_{t})}(dv,dz,dr),\ i=1,\cdots,N$ are Poisson point measures that are independent each other conditionally to $\overrightarrow{\bm{X}}^{\mathcal{P},M}_{t}$  and independent of $(X^1_0,\cdots,X^N_0,\Delta^1,\cdots,\Delta^N)$  with intensity $\widehat{\rho}(\overrightarrow{\bm{X}}^{\mathcal{P},M}_{t})(dv)\mu(dz)dr$. It is clear that $\overrightarrow{\bm{X}}^{\mathcal{P},M}_{t}$ may be simulated in an explicit way (see (\ref{fictiveshock*})).

We denote \begin{eqnarray*}
V_N:=\mathbbm{1}_{d=1}N^{-\frac{1}{2}}+\mathbbm{1}_{d=2}N^{-\frac{1}{2}}\log (1+N)+\mathbbm{1}_{d\geq3}N^{-\frac{1}{d}}, 
\end{eqnarray*}
and we  consider the following $d-$dimensional regularization kernels 
 \begin{eqnarray*}
\varphi (x)=\frac{1}{(2\pi )^{d/2}}e^{-\frac{\left\vert x\right\vert ^{2}}{2}%
},\quad \varphi _{\delta }(x)=\frac{1}{\delta ^{d}}\varphi (\frac{x}{\delta }),\quad0<\delta\leq1.
\end{eqnarray*}We have proved in \textbf{Theorem 2.1} that the law of $X_t$ has a density function $p_t(x)$.
Now we obtain in \textbf{Theorem 2.3} the following results concerning the approximation of  the density  $p_t(x)$. We take \[\delta=(\vert\mathcal{P}\vert+\sqrt{\varepsilon_M})^{\frac{1}{d+3}},\quad \text{and take}\ N\ \text{such that}\quad V_N\leq\vert\mathcal{P}\vert+\sqrt{\varepsilon_M}.\] Then we have
\begin{eqnarray}p_t(x)=\frac{1}{N}\sum_{i=1}^N\mathbb{E}\varphi_\delta(X^{\mathcal{P},M,i}_t-x)+\mathcal{O}((\vert\mathcal{P}\vert+\sqrt{\varepsilon_M})^{\frac{2}{d+3}}),\label{Introold}\end{eqnarray} where  $\mathcal{O}(\bullet)$ is the big $\mathcal{O}$ notation (i.e. for a strictly positive function $g$ defined on $\mathbb{R}_{+}$, $\exists  C>0$, s.t. $\vert\mathcal{O}(g(y))\vert\leq Cg(y)$). If we take  \[\delta=(\vert\mathcal{P}\vert+\sqrt{\varepsilon_M})^{\frac{1}{d+5}},\quad \text{and take}\ N\ \text{such that}\quad V_N\leq\vert\mathcal{P}\vert+\sqrt{\varepsilon_M},\] then we get moreover by Romberg method that
\begin{eqnarray}p_t(x)=\frac{2}{N}\sum_{i=1}^N\mathbb{E}\varphi_{\delta/\sqrt{2}}(X^{\mathcal{P},M,i}_t-x)-\frac{1}{N}\sum_{i=1}^N\mathbb{E}\varphi_\delta(X^{\mathcal{P},M,i}_t-x)+\mathcal{O}((\vert\mathcal{P}\vert+\sqrt{\varepsilon_M})^{\frac{4}{d+5}}). \label{Introold*} \end{eqnarray}It is clear that the approximation scheme based on Romberg method gives a better accuracy: we have the power $\frac{4}{d+5}>\frac{2}{d+3}$. So we are able to simulate the density function of $X_t$ in an explicit way, with error $\mathcal{O}((\vert\mathcal{P}\vert+\sqrt{\varepsilon_M})^{\frac{4}{d+5}})$. We notice however, that the speed of convergence of the error depends on the dimension $d$, so it converges slowly when $d$ is large. In \textbf{Theorem 2.4}, we  prove an alternative approximation result. We give up the approximation of the density, and we focus on the approximation in total variation distance. We take supplementally $\widetilde{\Delta}$ a $d-$dimensional standard Gaussian random variable independent of $\overrightarrow{\bm{X}}_t^{\mathcal{P},M}$. For any $\varepsilon>0$, we take \begin{eqnarray*}\delta=(\vert\mathcal{P}\vert+\varepsilon_M)^{\frac{1}{2}(1-\varepsilon^\prime)}\quad \text{and take}\ N\ \text{such that}\quad V_N\leq(\vert\mathcal{P}\vert+\varepsilon_M)^{\frac{d+3}{2}(1-\varepsilon^{\prime\prime})},\end{eqnarray*}with $\varepsilon^\prime=\frac{\varepsilon}{2-\varepsilon}$ and $\varepsilon^{\prime\prime}=\frac{(d+5)\varepsilon-2\varepsilon^2}{(d+3)(2-\varepsilon)}$. For every measurable and bounded function $f$, we prove that \begin{eqnarray}
\int_{\mathbb{R}^d}f(x)p_t(x)dx
&=&\frac{1}{N}\sum_{i=1}^N\mathbb{E}f(X_t^{\mathcal{P},M,i}+\delta\widetilde{\Delta})+\Vert f\Vert_\infty\times \mathcal{O}((\vert\mathcal{P}\vert+\sqrt{\varepsilon_M})^{1-\varepsilon}).\label{Intronew}
\end{eqnarray}We notice that the speed of convergence in (\ref{Intronew}) no longer depends on the dimension $d$, so it still behaves well for large dimension. We also stress that the speed of convergence in (\ref{Intronew}) is  the same as in (\ref{IntroTV2}) for the truncated Euler scheme. Moreover, for any $\varepsilon>0$,   we take \begin{eqnarray*}\delta=(\vert\mathcal{P}\vert+\varepsilon_M)^{\frac{1}{4}(1-\varepsilon_\prime)}\quad \text{and take}\ N\ \text{such that}\quad V_N\leq(\vert\mathcal{P}\vert+\varepsilon_M)^{\frac{d+5}{4}(1-\varepsilon_{\prime\prime})},\end{eqnarray*}with $\varepsilon_\prime=\frac{\varepsilon^2}{2-\varepsilon}$ and $\varepsilon_{\prime\prime}=\frac{8\varepsilon+(d-3)\varepsilon^2}{(d+5)(2-\varepsilon)}$. 
Then for every measurable and bounded function $f$, we get  by Romberg method that
\begin{eqnarray}
\int_{\mathbb{R}^d}f(x)p_t(x)dx
&=&\frac{2}{N}\sum_{i=1}^N\mathbb{E}f(X^{\mathcal{P},M,i}_t+\frac{\delta}{\sqrt{2}}\widetilde{\Delta})-\frac{1}{N}\sum_{i=1}^N\mathbb{E}f(X^{\mathcal{P},M,i}_t+\delta\widetilde{\Delta})+\Vert f\Vert_\infty\times \mathcal{O}((\vert\mathcal{P}\vert+\sqrt{\varepsilon_M})^{1-\varepsilon}).\quad\quad\label{Intronew*}
\end{eqnarray}
We remark that (\ref{Intronew*}) is even a better simulation scheme than (\ref{Intronew}) in the sense that the numbers of particles $N$  is smaller than the one in (\ref{Intronew}) and $\delta$ is larger than the one in (\ref{Intronew}). 

\bigskip

We give now a general view on the strategy used in the paper. Notice that
the Poisson process which appears in the equation (\ref{Intro1}) has intensity $\mu (dz)$
which is an infinite measure. As we mentioned before, it is convenient, both from the point of view
of Malliavin calculus and for simulation, to introduce an intermediary
equation driven by a Poison point measure with intensity $\mathbbm{1}_{\{\left\vert
z\right\vert \leq M\}}\mu (dz)$ which is a finite measure. We denote by $%
X_{t}^{M}$ the solution of this equation (which is a truncated version of (\ref{Intro1}), see (\ref{truncationM}) for precise expression). Since $X_{t}^{M}$ depends only on
a finite number of jumps in any compact time interval, this will be a
"simple functional" in the Malliavin calculus with respect to the amplitudes
of the jumps. We also replace the jumps larger then $M$ (which have been
canceled) by a Gaussian noise - this is necessary in order to obtain the non
degeneracy for $X_{t}^{M}.$ Moreover, in order to be able to establish
integration by parts formulas, we assume  (see \textbf{Hypothesis 2.4} $b)$) that the measure $\mu$ is absolutely continuous with respect to the Lebesgue measure: $\mu(dz)=h(z)dz$, where $h$ is infinitely differentiable and $\ln h$ has bounded derivatives of any order. Using the convergence $X_{t}^{M}\rightarrow X_{t}$, we are
able to prove that $X_{t}$ is smooth in the sense of Malliavin calculus for
jump processes. We use this calculus in order to prove that the law of $%
X_{t}$ is absolutely continuous with respect to the Lebesgue measure, with
smooth density $p_{{t}}(dx).$ 

Moreover, we construct an explicit algorithm which allows us to use Monte
Carlo simulation in order to approximate $X_{t}$ and $p_{{t}}.$ To do
it, we consider the Euler scheme $X_{t}^{\mathcal{P}}$ and the truncated
Euler scheme $X_{t}^{\mathcal{P},M}$ (see (\ref{IntroEul}) and (\ref{Intro3})). Now we focus on three equations 
with solutions $X_{t},X_{t}^{M}$ and $X_{t}^{\mathcal{P},M}.$
There is a supplementary difficulty which appears here: the Poisson point measures which
govern these equations have an intensity which depends on the law of the
solution of each of these equations. It is convenient to use similar
equations  driven by the same Poisson point measure. This is obtained by a
coupling procedure: we construct $x_{t},x_{t}^{M}$ and $x_{t}^{%
\mathcal{P},M}$ which have the same law as $X_{t},X_{t}^{M}$ and $%
X_{t}^{\mathcal{P},M}$ but are defined on the same probability space and
verify equations driven by the same Poisson point measure (this is done in Section 2.7). This allows us to
compare them by using an $L^{1}$ calculus. This is why all our computations
will concern these last equations. 

In $\cite{ref1}$, one obtains estimates of the Wasserstein distance between
these processes. In order to estimate the total variation distance between
them, we will use Malliavin integration by parts techniques (which are
presented in Section 3) together with some results from $\cite{ref2}$ which
allows us to pass from estimates in Wassestein distance to estimates in total
variation distance. Consequently a large part of the technical effort in the
paper will concern estimates of the Malliavin-Sobolev norms of $x_{t},x_{t}^{M}$ and $x_{t}^{\mathcal{P},M}$ as well as the proof of the non-degeneracy of these random variables (see Section 4). 

Our paper is organized in the following way. In Section 2, we   state our problems and give the hypotheses. We define the main equation $X_t$, the Euler scheme $X^{\mathcal{P}}_t$, the truncated Euler scheme  $X^{\mathcal{P},M}_{t}$ and the particle system $X^{\mathcal{P},M,i}_{t}$, $i=1,\cdots,N$. Then we state our main results: \textbf{Theorem 2.1, 2.2} (see (\ref{IntroTV1}) and (\ref{IntroTV2})) and \textbf{Theorem 2.3, 2.4} (see (\ref{Introold}), (\ref{Introold*}), (\ref{Intronew}) and (\ref{Intronew*})).  We also give some typical examples to apply our main results. At the end of this section, we  make a coupling argument to construct $x^M_t$, $x^{\mathcal{P},M}_{t}$ and $x_t$. In Section 3, we  give an abstract integration by parts framework (of Malliavin type) and then apply these abstract results to the solutions of our equations. There are two types of results that we have to prove in order to make the integration by parts machinery works. First, we prove that the Malliavin-Sobolev norms of $x^M_t$,  $x^{\mathcal{P},M}_{t}$ and $x_t$ are bounded, uniformly with respect to $\mathcal{P}$ and $M$ (see \textbf{Lemma 3.7}). Moreover we have to check the non-degeneracy condition for the Malliavin covariance matrix. This is done in \textbf{Lemma 3.8}. Both these two lemmas are rather technical so we leave the proofs for Section 4. Once these lemmas are proved, \textbf{Proposition 3.6.1} allows us to conclude that $X^{\mathcal{P},M}_{t}\rightarrow X_t$ in total variation distance. We also prove that the Euler scheme $X^{\mathcal{P}}_{t}\rightarrow X_t$ in total variation distance in a similar way. Furthermore,  we  obtain an algorithm based on the particle system $X^{\mathcal{P},M,i}_{t}$, $i=1,\cdots,N$ in order to compute the density function $p_t(x)$ of the law of $X_t$, and we estimate the error.

\section{Main results}
\subsection{Basic notations and the main equation}
We give a time horizon $T>0$ and let $0<t\leq T$. To begin, we introduce some notations which will be used through our paper. For a multi-index $\beta$, we denote $\vert\beta\vert$ to be the length of $\beta$. We denote $C_b^{l}(\mathbb{R}^d)$ the space of $l-$times differential and bounded functions on $\mathbb{R}^d$ with bounded derivatives up to order $l$, and $\left\Vert f \right\Vert _{l,\infty }:=\sum\limits_{ \vert\beta\vert\leq l}\left\Vert  \partial^\beta f\right\Vert _{\infty }$ for a function $f\in C_b^{l}(\mathbb{R}^d)$. We  also denote $\mathcal{P}_l(\mathbb{R}^d)$ the space of all probability measures on $\mathbb{R}^d$ with finite $l-$moment. For $\rho_1,\rho_2\in\mathcal{P}_1(\mathbb{R}^d)$, we define the Wasserstein distance $W_1$ by 
\begin{eqnarray}
W_1(\rho_1,\rho_2)=\sup\limits_{Lip(f)\leq1}\big\vert\int_{\mathbb{R}^d}f(x)\rho_1(dx)-\int_{\mathbb{R}^d}f(x)\rho_2(dx)\big\vert,\label{W1}
\end{eqnarray}with $Lip(f):=\sup\limits_{x\neq y}\frac{\vert f(x)-f(y)\vert}{\vert x-y\vert}$ the Lipschitz constant of $f$, and we define the total variation distance $d_{TV}$ by 
\begin{eqnarray}d_{TV}(\rho_1,\rho_2)=\sup\limits_{\Vert f\Vert_{\infty}\leq1}\big\vert\int_{\mathbb{R}^d}f(x)\rho_1(dx)-\int_{\mathbb{R}^d}f(x)\rho_2(dx)\big\vert.\label{dTV}\end{eqnarray} For $F,G\in L^1(\Omega)$, we also denote $W_1(F,G)=W_1(\mathcal{L}(F),\mathcal{L}(G))$ and $d_{TV}(F,G)=d_{TV}(\mathcal{L}(F),\mathcal{L}(G))$, with $\mathcal{L}(F)$(respectively $\mathcal{L}(G)$) the law of the random variable $F$(respectively $G$). In addition, along the paper, $C$ will be a constant which may change from a line to another. It may depend on some parameters and sometimes the dependence is precised in the notation (ex. $C_l$ is a constant depending on $l$).

In this paper, we consider the $d-$dimensional stochastic differential equation with jumps
\begin{eqnarray}
X_{t}&=&X_0+\int_{0}^{t}b(r,X_{r},\rho_{r})dr +\int_{0}^{t}\int_{\mathbb{R}^d\times\mathbb{R}^d}c(r,v,z,X_{r-},\rho_{r-})N_{\rho_{r-}}(dv,dz,dr),  \label{1.1}
\end{eqnarray}
where $\rho_{t}(dv)=\mathbb{P}(X_{t}\in dv)$ is the law of $X_{t}$, $N_{\rho_{t}}$ is a Poisson point measure on the state space $\mathbb{R}^d\times\mathbb{R}^d$ with intensity measure $\widehat{N}_{\rho_{t}}(dv,dz,dr)=\rho_{t}(dv)\mu(dz)dr$, $X_0$ is the initial random variable with law $\rho_0$ independent of the Poisson point measure $N_{\rho_{t}}$,  $\mu$ is a positive $\sigma$-finite measure on $\mathbb{R}^d$, and $b:[0,T]\times\mathbb{R}^d\times\mathcal{P}_1(\mathbb{R}^d)\rightarrow\mathbb{R}^d$,  $c:[0,T]\times\mathbb{R}^d\times\mathbb{R}^d\times\mathbb{R}^d\times\mathcal{P}_1(\mathbb{R}^d)\rightarrow\mathbb{R}^d$.
\begin{remark}
We remark that we will assume in the following that $\int_{\mathbb{R}^d}\sup\limits_{r\in[0,T]}\sup\limits_{v,x\in\mathbb{R}^d}\sup\limits_{\rho\in\mathcal{P}_1(\mathbb{R}^d)}\vert {c}(r,v,z,x,\rho)\vert\mu(dz)<\infty$, so we are in the finite variation case. The integral with respect to the Poisson point measure is not compensated.
\end{remark}

\subsection{Hypotheses}
Here we give our hypotheses.  

\textbf{Hypothesis 2.1} (\textbf{Regularity})
We assume that  the function $x\mapsto b(r,x,\rho)$   is infinitely differentiable with bounded derivatives of any orders, and that  $\rho_0\in\bigcap_{p=1}^\infty\mathcal{P}_p(\mathbb{R}^d)$. 
We also assume that the function $(z,x)\mapsto c(r,v,z,x,\rho)$ is infinitely differentiable and for every multi-indices $\beta_1,\beta_2$, there exists a non-negative function $\bar{c}:\mathbb{R}^d\rightarrow\mathbb{R}_{+}$ depending on $\beta_1,\beta_2$ such that  we have
\[
\sup_{r\in[0,T]}\sup_{v,x\in\mathbb{R}^d}\sup_{\rho\in\mathcal{P}_1(\mathbb{R}^d)}(\vert {c}(r,v,z,x,\rho)\vert+\vert \partial_z^{\beta_2}\partial_x^{\beta_1} {c}(r,v,z,x,\rho)\vert)\leq\bar{c}(z), \quad\forall z\in\mathbb{R}^d,
\]
\begin{eqnarray}
with\quad \int_{\mathbb{R}^d}\vert \bar{c}(z)\vert^p\mu(dz):=\bar{c}_p<\infty,\quad\forall p\geq1. \label{cbar}
\end{eqnarray}

Moreover, there exists a constant $L_b>0$ such that for any $r_1,r_2\in[0,T], v_1,v_2,x\in\mathbb{R}^d, z\in\mathbb{R}^d,  \rho_1,\rho_2\in\mathcal{P}_1(\mathbb{R}^d)$,
\begin{eqnarray*}
&&\vert b(r_1,x,\rho_1)-b(r_2,x,\rho_2)\vert \leq L_b(\vert r_1-r_2\vert+W_1(\rho_1,\rho_2)),
\end{eqnarray*}
\begin{eqnarray*}and\quad
&&\vert {c}(r_1,v_1,z,x,\rho_1)-{c}(r_2,v_2,z,x,\rho_2)\vert\\ 
&&+\vert \nabla_{z}{c}(r_1,v_1,z,x,\rho_1)-\nabla_{z}{c}(r_2,v_2,z,x,\rho_2)\vert+\vert \nabla_{x}{c}(r_1,v_1,z,x,\rho_1)-\nabla_{x}{c}(r_2,v_2,z,x,\rho_2)\vert\\
&&\leq \bar{c}(z)(\vert r_1-r_2\vert+\vert v_1-v_2\vert+W_1(\rho_1,\rho_2)).
\end{eqnarray*}

\begin{remark}
We will use several times the following consequence of (\ref{cbar}) and of Burkholder inequality (see for example the Theorem 2.11 in $\cite{ref12}$, see also in $\cite{ref13}$): Let ${\Phi}(r,v,z,\omega,\rho):[0,T]\times\mathbb{R}^d\times\mathbb{R}^d\times\Omega\times\mathcal{P}_1(\mathbb{R}^d)\rightarrow\mathbb{R}_{+}$ and  ${\varphi}(r,v,\omega,\rho):[0,T]\times\mathbb{R}^d\times\Omega\times\mathcal{P}_1(\mathbb{R}^d)\rightarrow\mathbb{R}_{+}$ be two  functions such that 
\[
\vert{\Phi}(r,v,z,\omega,\rho)\vert\leq\vert\bar{c}(z)\vert\vert{\varphi}(r,v,\omega,\rho)\vert.
\]
Then for any $p\geq2,\rho\in\mathcal{P}_1(\mathbb{R}^d)$,
\begin{eqnarray}
	\mathbb{E}\Big\vert\int_0^t\int_{\mathbb{R}^d\times\mathbb{R}^d}{\Phi}(r,v,z,\omega,\rho)N_{\rho}(dv,dz,dr)\Big\vert^p
	\leq C \mathbb{E}\int_0^t\int_{\mathbb{R}^d}\vert{\varphi}(r,v,\omega,\rho)\vert^p \rho(dv)dr, \label{BurkN}
\end{eqnarray}
where $C$ is a constant depending on $p$, $\bar{c}_1$, $\bar{c}_2$, $\bar{c}_p$ and $T$.
\end{remark}

\begin{proof}
By compensating $N_\rho$ and using Burkholder inequality and (\ref{cbar}), we have
\begin{eqnarray}
&&\mathbb{E}\vert\int_0^t\int_{\mathbb{R}^d\times\mathbb{R}^d}{\Phi}(r,v,z,\omega,\rho)N_{\rho}(dv,dz,dr)\vert^p \nonumber\\
&&\leq C[\mathbb{E}(\int_0^t\int_{\mathbb{R}^d\times\mathbb{R}^d}\vert{\Phi}(r,v,z,\omega,\rho)\vert^2 \rho(dv)\mu(dz)dr)^{\frac{p}{2}} +\mathbb{E}\int_0^t\int_{\mathbb{R}^d\times\mathbb{R}^d}\vert{\Phi}(r,v,z,\omega,\rho)\vert^p \rho(dv)\mu(dz)dr \nonumber\\
&&+\mathbb{E}\vert\int_0^t\int_{\mathbb{R}^d\times\mathbb{R}^d}\vert{\Phi}(r,v,z,\omega,\rho)\vert \rho(dv)\mu(dz)dr\vert^{p}] \label{Burk} \\
&&\leq C\mathbb{E}\int_0^t\int_{\mathbb{R}^d}\vert{\varphi}(r,v,\omega,\rho)\vert^p\rho(dv)dr. \nonumber
\end{eqnarray}
\end{proof}
For the sake of simplicity of notations, in the following, for a  constant $C$, we do not precise the dependence  on the regularity constants of the function $b$ and $c$ (such as $\Vert\nabla_xb\Vert_{\infty}$, $L_b$ and $\bar{c}_p$).

\textbf{Hypothesis 2.2}
We assume that there exists a non-negative function $\breve{c}:\mathbb{R}^d\rightarrow\mathbb{R}_{+}$ such that $\int_{\mathbb{R}^d}\vert \breve{c}(z)\vert^p\mu(dz):=\breve{c}_p<\infty,\ \forall p\geq1$, and
\[
\left\Vert\nabla_{x}{c}(r,v,z,x,\rho)(I_d+\nabla_{x}{c}(r,v,z,x,\rho))^{-1}\right\Vert\leq\breve{c}(z),\quad\forall r\in[0,T], v,x\in\mathbb{R}^d, z\in\mathbb{R}^d, \rho\in\mathcal{P}_1(\mathbb{R}^d),
\]with $I_d$ the $d-$dimensional identity matrix. To avoid overburdening notation, since both hypotheses 2.1 and 2.2 apply, we take $\breve{c}(z)=\bar{c}(z)$ and $\breve{c}_p=\bar{c}_p$.
\begin{remark}
We need this hypothesis to prove the regularity of the inverse tangent flow (see Section 4.2 (\ref{cM2})).
\end{remark}

\textbf{Hypothesis 2.3} (\textbf{Ellipticity})
There exists a non-negative function $\underline{c}:\mathbb{R}^d\rightarrow\mathbb{R}_{+}$ such that for every $r\in[0,T], v,x\in\mathbb{R}^d, z\in\mathbb{R}^d, \rho\in\mathcal{P}_1(\mathbb{R}^d), \zeta\in\mathbb{R}^d$, we have
\[
\sum_{j=1}^d\langle\partial_{z_j}{c}(r,v,z,x,\rho),\zeta\rangle^{2}\geq \underline{c}%
(z)\vert\zeta\vert^2.
\]
\begin{remark}
We notice that together with \textbf{Hypothesis 2.1}, we have $\underline{c}(z)\leq\vert\bar{c}(z)\vert^2,\ \forall z\in\mathbb{R}^d$.
\end{remark}

\textbf{Hypothesis 2.4} 

We give some supplementary hypotheses concerning the function $\underline{c}$ and the measure $\mu$.

$a)$ We assume that there exists a $\theta>0$ such that
\begin{eqnarray}
\underline{\lim}_{u\rightarrow+\infty}\frac{1}{\ln u}\nu\{\underline{c}\geq \frac{1}{u}\}:=\theta>0, \label{undergamma}
\end{eqnarray}with \[\nu(dz)=\sum_{k=1}^{\infty}\mathbbm{1}_{[k-\frac{3}{4},k-\frac{1}{4}]}(\vert z\vert)\mu(dz).\] This means that $\underline{c}$ could not be too small so that we could have enough noises to deduce the non-degeneracy of the Malliavin covariance matrix (see Section 4.2 (\ref{cm6})). 
\begin{remark}
If $\mu(\mathbb{R}^d)<\infty$, then $\theta=0$. So (\ref{undergamma}) implies that $\mu(\mathbb{R}^d)=\infty$.
\end{remark}

\bigskip

$b)$ We assume that $\mu$ is absolutely continuous with respect to the Lebesgue measure: $\mu(dz)=h(z)dz$, where $h$ is infinitely differentiable and $\ln h$ has bounded derivatives of any order. 
\begin{remark}
We need this hypothesis to construct the integration by parts framework for the jump equations.
\end{remark}

\bigskip\bigskip

\subsection{The Euler scheme}
Now we construct the Euler scheme. For any partition $\mathcal{P}=\{0=r_0<r_1<\cdots<r_{n-1}<r_n=T\}$ of the interval $[0,T]$, we define $\tau(r)=r_k$ when $r\in[r_k,r_{k+1})$, and we consider the equation
\begin{eqnarray}
X^{\mathcal{P}}_{t}&=&X_0+\int_{0}^{t}b(\tau(r),X^{\mathcal{P}}_{\tau(r)},\rho^{\mathcal{P}}_{\tau(r)})dr  \nonumber\\
&+&\int_{0}^{t}\int_{\mathbb{R}^d\times\mathbb{R}^d}{c}(\tau(r),v,z,X^{\mathcal{P}}_{\tau(r)-},\rho^{\mathcal{P}}_{\tau(r)-})N_{\rho^{\mathcal{P}}_{\tau(r)-}}(dv,dz,dr),  \label{1.3}
\end{eqnarray}
where $\rho^{\mathcal{P}}_{t}$ is the law of $X^{\mathcal{P}}_{t}$  and $N_{\rho^{\mathcal{P}}_{t}}(dv,dz,dr)$ is a Poisson point measure with intensity measure $\rho^{\mathcal{P}}_{t}(dv)\mu(dz)dr$, independent of $X_0$.

In $\cite{ref1}$(Theorem 3.5, 3.7, 3.8, Proposition 3.9), Alfonsi and Bally have proved that under some suitable regularity conditions on the coefficients $b$ and $c$ (which are some conditions weaker than the \textbf{Hypothesis 2.1} in this paper), the strong solution of the equation (\ref{1.1}) exists and is unique, and the following statements are true.

$a)$ There exists a constant $C$ depending on $T$  such that for every $0<t\leq T$ and every partition $\mathcal{P}$ of $[0,T]$, 
\begin{eqnarray}
W_1(X^{\mathcal{P}}_{t},X_{t})\leq C\vert\mathcal{P}\vert,\quad \text{with}\quad \vert\mathcal{P}\vert:=\max\limits_{k\in\{0,\cdots,n-1\}}(r_{k+1}-r_k). \label{AB}
\end{eqnarray}

$b)$ The solution of the following weak equation exists.
\begin{eqnarray}
\forall\phi\in C_b^1(\mathbb{R}^d),\ &&\int_{\mathbb{R}^d}\phi(x)\rho_{t}(dx)=\int_{\mathbb{R}^d}\phi(x)\rho_{0}(dx)+\int_{0}^{t}\int_{\mathbb{R}^d}\langle b(r,x,\rho_{r}),\nabla\phi(x)\rangle\rho_{r}(dx)dr  \nonumber\\
&&+\int_{0}^{t}\int_{\mathbb{R}^d\times\mathbb{R}^d}\rho_{r}(dx)\rho_{r}(dv)\int_{\mathbb{R}^d}(\phi(x+c(r,v,z,x,\rho_{r}))-\phi(x))\mu(dz)dr. \quad\quad  \label{1.4}
\end{eqnarray}
And the solution of the equation (\ref{1.1}) is the probabilistic interpretation of  (\ref{1.4}) in the sense that $\rho_t=\mathcal{L}(X_t)$ (the law of $X_t$) solves (\ref{1.4}).

We recall the notation  $\theta$ in \textbf{Hypothesis 2.4}.
One aim of this paper is to prove the following error estimate.
\begin{theorem}
Under the \textbf{Hypothesis 2.1, Hypothesis 2.2, Hypothesis 2.3} and \textbf{Hypothesis 2.4}, we have 

a) For any $0<t\leq T$, when $t>\frac{8d(l+d)}{\theta}$, the law of $X_{t}$ has a $l-$times differentiable density $p_{t}$: 
\begin{eqnarray}\mathbb{P}(X_{t}\in dx)=\rho_{t}(dx)=p_{t}(x)dx,\label{001}\end{eqnarray}and the density $p_{t}$ is a function solution of the equation (\ref{1.4}).

b) For any $\varepsilon>0$, there exists a constant $C$ depending on $\varepsilon,d$ and $T$ such that for every partition $\mathcal{P}$ of $[0,T]$ with $\vert\mathcal{P}\vert\leq1$,  when $t>\frac{8d}{\theta}(\frac{8}{\varepsilon}+1)$,
\begin{eqnarray}d_{TV}(X^\mathcal{P}_{t},X_{t})\leq C\vert\mathcal{P}\vert^{1-\varepsilon}.\label{002}\end{eqnarray} 
\end{theorem}
\begin{remark}
In the case $\theta=\infty$, the results in \textbf{Theorem 2.1}  hold for every $0<t\leq T$.
\end{remark}
The proof of this theorem will be given in Section 3.3.
The main methods we will use in the proofs are the Malliavin calculus techniques introduced in $\cite{ref2}$. We will discuss them in Section 3. 

\subsection{The truncated Euler scheme}
Since we have $\mu(\mathbb{R}^d)=\infty$ (which is a consequence of (\ref{undergamma})), we have infinitely many jumps. We use a truncation argument in order to have finite numbers of jumps and obtain a representation by means of a compound Poisson process. This is necessary in order to obtain a scheme which may be simulated. For $M\in\mathbb{N}$, we denote $B_M=\{z\in\mathbb{R}^d:\vert z\vert\leq M\}$, $c_M(r,v,z,x,\rho):=c(r,v,z,x,\rho)\mathbbm{1}_{B_M}(z)$, and \begin{eqnarray}a^M_{T}:=\sqrt{T\int_{\{\vert z\vert>M\}}\underline{c}(z)  \mu(dz)}.\label{aMt}\end{eqnarray} This is a deterministic sequence such that $ a^M_{T}\rightarrow0$ as $M\rightarrow\infty$. We also denote $\Delta=(\Delta_1,\cdots,\Delta_d)$ to be a $d-$dimensional standard Gaussian random variable independent of $X_0$ and $N_{\rho}$. Now we cancel the "big jumps" (the jumps of size $\vert z\vert>M$) and replace them by a Gaussian random variable $a^M_{t}\Delta$. 
\begin{eqnarray}
X^{\mathcal{P},M}_{t}&=&X_0+a^M_{T}\Delta+\int_{0}^{t}b(\tau(r),X^{\mathcal{P},M}_{\tau(r)},\rho^{\mathcal{P},M}_{\tau(r)})dr  \nonumber\\
&+&\int_{0}^{t}\int_{\mathbb{R}^d\times \mathbb{R}^d}{c}_M(\tau(r),v,z,X^{\mathcal{P},M}_{\tau(r)-},\rho^{\mathcal{P},M}_{\tau(r)-})N_{\rho^{\mathcal{P},M}_{\tau(r)-}}(dv,dz,dr),  \label{truncation}
\end{eqnarray}
where  $\rho^{\mathcal{P},M}_{t}$ is the law of $X^{\mathcal{P},M}_{t}$, and $N_{\rho^{\mathcal{P},M}_{t}}(dv,dz,dr)$ is a Poisson point measure with intensity measure $\rho^{\mathcal{P},M}_{t}(dv)\mu(dz)dr$, independent of $X_0$ and of $\Delta$.
We remark that $\Delta$ is necessary in order to obtain the non-degeneracy of the Malliavin covariance matrix, which will be discussed in detail in Section 4.2. 

The advantage of considering $X^{\mathcal{P},M}_{t}$ is that we may represent it by means of compound Poisson processes.
For $k\in\mathbb{N}$, we denote $I_1=B_1$, $I_k=B_{k}\backslash B_{k-1}$ for $k\geq2$ and take $({J}_t^{k})_{t\in[0,T]}$ a Poisson process of intensity $\mu (I_k)$. We denote by $({T}_i^{k})_{i\in\mathbb{N}}$ the jump times of $({J}_t^{k})_{t\in[0,T]}$ and we consider some sequences of independent random variables ${Z}_i^{k}\sim\mathbbm{1}_{I_k}(z)\frac{\mu(dz)}{\mu(I_k)}$,  and ${V}_{k,i}^{\mathcal{P},M}\sim\rho^{\mathcal{P},M}_{\tau({T}_i^{k})-}(dv)$, $k,i\in\mathbb{N}$.  Moreover, $(({J}_t^{k})_{\substack{t\in[0,T]\\k\in\mathbb{N}}}, ({Z}_i^{k})_{k,i\in\mathbb{N}}$, $  ({V}_{k,i}^{\mathcal{P},M})_{k,i\in\mathbb{N}}, X_0,\Delta)$ are taken to be independent. Then in order to do the simulation, we  represent the jump's parts of the equation  (\ref{truncation}) by  compound Poisson processes: 
\begin{eqnarray}
X^{\mathcal{P},M}_{t}&=&X_0+a^M_{T}\Delta+\int_{0}^{t}b(\tau(r),X^{\mathcal{P},M}_{\tau(r)},\rho^{\mathcal{P},M}_{\tau(r)})dr  +\sum_{k=1}^{M}\sum_{i=1}^{J^k_t}{c}(\tau({T}_i^{k}),{V}_{k,i}^{\mathcal{P},M},{Z}_i^{k},X^{\mathcal{P},M}_{\tau({T}_i^{k})-},\rho^{\mathcal{P},M}_{\tau({T}_i^{k})-}).  \nonumber\\ \label{fictiveshock}
\end{eqnarray}Notice that the solution of the equation (\ref{fictiveshock}) may be constructed in an explicit way (except for $\rho^{\mathcal{P},M}_{\tau(r)}$ and $\rho^{\mathcal{P},M}_{\tau({T}_i^{k})-}$ which will be discussed in detail in Section 2.5).

We denote \begin{eqnarray}
\varepsilon_M:=\int_{\{\vert z\vert>M\}}\vert\bar{c}(z)\vert^2\mu(dz)+\vert\int_{\{\vert z\vert>M\}}\bar{c}(z)\mu(dz)\vert^2, \label{epsM}
\end{eqnarray}and recall the notation  $\theta$ in \textbf{Hypothesis 2.4}.
We obtain the following error estimate for $X^{\mathcal{P},M}_{t}$.
\begin{theorem}
Under the \textbf{Hypothesis 2.1, Hypothesis 2.2, Hypothesis 2.3} and \textbf{Hypothesis 2.4}, we have 

a) For any $0<t\leq T$, the law of $X^{\mathcal{P},M}_{t}$ has a smooth density $p^{\mathcal{P},M}_{t}$.

b) For any $\varepsilon>0$, there exists a constant $C$ depending on $\varepsilon,d$ and $T$ such that for  every partition $\mathcal{P}$ of $[0,T]$ with $\vert\mathcal{P}\vert\leq1$,  every $M\in\mathbb{N}$ with $\varepsilon_M\leq 1$ and $\vert\bar{c}(z)\vert^2\mathbbm{1}_{\{\vert z\vert>M\}}\leq 1$, when $t>\frac{8d}{\theta}(\frac{8}{\varepsilon}+1)$,
\begin{eqnarray}d_{TV}(X^{\mathcal{P},M}_{t},X_{t})\leq C(\sqrt{\varepsilon_M}+\vert\mathcal{P}\vert)^{1-\varepsilon}.\label{002*}\end{eqnarray} 

\end{theorem}
\begin{remark}
In the case $\theta=\infty$, the results in  \textbf{Theorem 2.2} hold for every $0<t\leq T$.
\end{remark}

The proof of this theorem will be given in Section 3.3
by using some Malliavin integration by parts techniques. 

\subsection{The particle system}
We notice that we still cannot compute ${\rho}^{\mathcal{P},M}_{\tau(r)}$ and $\rho^{\mathcal{P},M}_{\tau({T}_i^{k})-}$ directly in (\ref{fictiveshock}), so we construct the particle system as follows in order to obtain an explicit scheme of simulation.
For a random vector $X=(X^1,\cdots, X^N),\ X^i\in\mathbb{R}^d,\ i=1,\cdots,N$ with a fixed dimension $N$, we associate the (random) empirical measure
\begin{eqnarray}
\widehat{\rho}(X)(dv)=\frac{1}{N}\sum_{i=1}^N\delta_{X^i}(dv),\label{empirical}
\end{eqnarray}where  $\delta_x(dv)$ is the Dirac measure. Now we consider an initial vector $(X^1_0,\cdots,X^N_0)$ with components which are independent and identically distributed with common law $\rho_0$ (we recall that $\rho_0$ is the law of $X_0$ in (\ref{1.1})), and we consider $(\Delta^1,\cdots,\Delta^N)$ which is a $N\times d-$dimensional standard Gaussian random variable independent of $(X^1_0,\cdots,X^N_0)$. Then we construct the particle system $\overrightarrow{\bm{X}}^{\mathcal{P},M}_t=(X^{\mathcal{P},M,1}_t,\cdots,X^{\mathcal{P},M,N}_t)$:
\begin{eqnarray}
X^{\mathcal{P},M,i}_{t}&=&X_0^i+a^M_{T}\Delta^i+\int_{0}^{t}b(\tau(r),X^{\mathcal{P},M,i}_{\tau(r)},\widehat{\rho}(\overrightarrow{\bm{X}}^{\mathcal{P},M}_{\tau(r)}))dr  \nonumber\\
&+&\int_{0}^{t}\int_{\mathbb{R}^d\times \mathbb{R}^d}{c}_M(\tau(r),v,z,X^{\mathcal{P},M,i}_{\tau(r)-},\widehat{\rho}(\overrightarrow{\bm{X}}^{\mathcal{P},M}_{\tau(r)-}))N^i_{\widehat{\rho}(\overrightarrow{\bm{X}}^{\mathcal{P},M}_{\tau(r)-})}(dv,dz,dr),\ i=1,\cdots,N,\label{particle}
\end{eqnarray}where  $N^i_{\widehat{\rho}(\overrightarrow{\bm{X}}^{\mathcal{P},M}_{t})}(dv,dz,dr),\ i=1,\cdots,N$ are Poisson point measures that are independent each other conditionally to $\overrightarrow{\bm{X}}^{\mathcal{P},M}_{t}$ and independent of $(X^1_0,\cdots,X^N_0,\Delta^1,\cdots,\Delta^N)$ with intensity $\widehat{\rho}(\overrightarrow{\bm{X}}^{\mathcal{P},M}_{t})(dv)\mu(dz)dr$. We give now the representation of the above equation in terms of compound Poisson processes. This is necessary in order to obtain an explicit simulation algorithm. We recall that we denote $I_1=B_1$, $I_k=B_{k}\backslash B_{k-1}$ for $k\geq2$. Now for  $i=1,\cdots,N$, $k\in\mathbb{N}$, we take $({J}_t^{k,i})_{t\in[0,T]}$ a Poisson process of intensity $\mu (I_k)$. We denote by $({T}_l^{k,i})_{l\in\mathbb{N}}$ the jump times of $({J}_t^{k,i})_{t\in[0,T]}$ and we consider some sequences of independent random variables ${Z}_l^{k,i}\sim\mathbbm{1}_{I_k}(z)\frac{\mu(dz)}{\mu(I_k)}$ and $U_l^{k,i}$ uniformly distributed on $\{1,\cdots,N\}$, for all $i=1,\cdots,N$, $k,l\in\mathbb{N}$.  Moreover, $({J}_t^{k,i})_{\substack{t\in[0,T]}}, {Z}_l^{k,i}$, $  {U}_l^{k,i}, \Delta^i, X^i_0,\ i=1,\cdots,N,\ k,l\in\mathbb{N}$ are taken to be independent. Then  we  represent the jump's parts of the equation  (\ref{particle}) by  compound Poisson processes to give an explicit scheme of simulation.
\begin{eqnarray}
X^{\mathcal{P},M,i}_{t}&=&X_0^i+a^M_{T}\Delta^i+\int_{0}^{t}b(\tau(r),X^{\mathcal{P},M,i}_{\tau(r)},\widehat{\rho}(\overrightarrow{\bm{X}}^{\mathcal{P},M}_{\tau(r)}))dr \nonumber\\ &+&\sum_{k=1}^{M}\sum_{l=1}^{J^{k,i}_t}{c}(\tau({T}_l^{k,i}),X^{\mathcal{P},M,U_l^{k,i}}_{\tau(T_l^{k,i})-},{Z}_l^{k,i},X^{\mathcal{P},M,i}_{\tau({T}_l^{k,i})-},\widehat{\rho}(\overrightarrow{\bm{X}}^{\mathcal{P},M}_{\tau({T}_l^{k,i})-})).   \label{fictiveshock*}
\end{eqnarray}So now the solution of the equation (\ref{fictiveshock*}) is constructed in an explicit way.

We denote \begin{eqnarray}
V_N:=\mathbbm{1}_{d=1}N^{-\frac{1}{2}}+\mathbbm{1}_{d=2}N^{-\frac{1}{2}}\log (1+N)+\mathbbm{1}_{d\geq3}N^{-\frac{1}{d}}, \label{VN}
\end{eqnarray}
and we  consider the $d-$dimensional regularization kernels 
 \begin{eqnarray}
\varphi (x)=\frac{1}{(2\pi )^{d/2}}e^{-\frac{\left\vert x\right\vert ^{2}}{2}%
},\quad \varphi _{\delta }(x)=\frac{1}{\delta ^{d}}\varphi (\frac{x}{\delta }),\quad0<\delta\leq1.\label{kernel}
\end{eqnarray}
We recall the notations $\varepsilon_M$ in (\ref{epsM}) and $\theta$ in \textbf{Hypothesis 2.4}. In \textbf{Theorem 2.1}, we proved that under appropriate hypotheses, $\mathcal{L}(X_t)(dx)=p_t(x)dx$. We give now some approximation results for $p_t(x)$.
\begin{theorem}

Under the \textbf{Hypothesis 2.1, Hypothesis 2.2, Hypothesis 2.3} and \textbf{Hypothesis 2.4},   for  every partition $\mathcal{P}$ of $[0,T]$ and  every $M\in\mathbb{N}$ with $\vert\mathcal{P}\vert+\sqrt{\varepsilon_M}\leq1$, we have the following:

$i)$ We take \[\delta=(\vert\mathcal{P}\vert+\sqrt{\varepsilon_M})^{\frac{1}{d+3}},\quad \text{and take}\ N\ \text{such that}\quad V_N\leq\vert\mathcal{P}\vert+\sqrt{\varepsilon_M}.\] When $t>\frac{8d}{\theta}(2+d)$,
\begin{eqnarray}p_t(x)=\frac{1}{N}\sum_{i=1}^N\mathbb{E}\varphi_\delta(X^{\mathcal{P},M,i}_t-x)+\mathcal{O}((\vert\mathcal{P}\vert+\sqrt{\varepsilon_M})^{\frac{2}{d+3}}),\label{003}\end{eqnarray} where $\mathcal{O}(\bullet)$ is the big $\mathcal{O}$ notation (i.e. for a strictly positive function $g$ defined on $\mathbb{R}_{+}$, $\exists C>0$, s.t. $\vert\mathcal{O}(g(y))\vert\leq Cg(y)$).  

$ii)$ \textbf{(Romberg)} We take  \[\delta=(\vert\mathcal{P}\vert+\sqrt{\varepsilon_M})^{\frac{1}{d+5}},\quad \text{and take}\ N\ \text{such that}\quad V_N\leq\vert\mathcal{P}\vert+\sqrt{\varepsilon_M}.\] When $t>\frac{8d}{\theta}(4+d)$,
\begin{eqnarray}p_t(x)=\frac{2}{N}\sum_{i=1}^N\mathbb{E}\varphi_{\delta/\sqrt{2}}(X^{\mathcal{P},M,i}_t-x)-\frac{1}{N}\sum_{i=1}^N\mathbb{E}\varphi_\delta(X^{\mathcal{P},M,i}_t-x)+\mathcal{O}((\vert\mathcal{P}\vert+\sqrt{\varepsilon_M})^{\frac{4}{d+5}}).\label{003*}\end{eqnarray} 
\end{theorem}

\begin{theorem}

We suppose \textbf{Hypothesis 2.1, Hypothesis 2.2, Hypothesis 2.3} and \textbf{Hypothesis 2.4} hold true. We take supplementally $\widetilde{\Delta}$ to be a $d-$dimensional standard Gaussian random variable independent of $\overrightarrow{\bm{X}}_t^{\mathcal{P},M}$. Let $\mathcal{P}$ be a  partition of $[0,T]$ with $\vert\mathcal{P}\vert\leq1$,  and let $M\in\mathbb{N}$ with $\varepsilon_M\leq 1$ and $\vert\bar{c}(z)\vert^2\mathbbm{1}_{\{\vert z\vert>M\}}\leq 1$. For any $\varepsilon>0$,  for every measurable and bounded function $f$, when $t>\frac{8d}{\theta}(\frac{16}{\varepsilon}+1)$,  we have the followings.

$i)$ We take \begin{eqnarray*}\delta=(\vert\mathcal{P}\vert+\varepsilon_M)^{\frac{1}{2}(1-\varepsilon^\prime)}\quad \text{and take}\ N\ \text{such that}\quad V_N\leq(\vert\mathcal{P}\vert+\varepsilon_M)^{\frac{d+3}{2}(1-\varepsilon^{\prime\prime})},\end{eqnarray*}with $\varepsilon^\prime=\frac{\varepsilon}{2-\varepsilon}$ and $\varepsilon^{\prime\prime}=\frac{(d+5)\varepsilon-2\varepsilon^2}{(d+3)(2-\varepsilon)}$.
Then
\begin{eqnarray}
\int_{\mathbb{R}^d}f(x)p_t(x)dx
&=&\frac{1}{N}\sum_{i=1}^N\mathbb{E}f(X_t^{\mathcal{P},M,i}+\delta\widetilde{\Delta})+\Vert f\Vert_\infty\times \mathcal{O}((\vert\mathcal{P}\vert+\sqrt{\varepsilon_M})^{1-\varepsilon}).\label{004}
\end{eqnarray}

$ii)$ \textbf{(Romberg)} We take \begin{eqnarray*}\delta=(\vert\mathcal{P}\vert+\varepsilon_M)^{\frac{1}{4}(1-\varepsilon_\prime)}\quad \text{and take}\ N\ \text{such that}\quad V_N\leq(\vert\mathcal{P}\vert+\varepsilon_M)^{\frac{d+5}{4}(1-\varepsilon_{\prime\prime})},\end{eqnarray*}with $\varepsilon_\prime=\frac{\varepsilon^2}{2-\varepsilon}$ and $\varepsilon_{\prime\prime}=\frac{8\varepsilon+(d-3)\varepsilon^2}{(d+5)(2-\varepsilon)}$.
Then
\begin{eqnarray}
\int_{\mathbb{R}^d}f(x)p_t(x)dx
&=&\frac{2}{N}\sum_{i=1}^N\mathbb{E}f(X^{\mathcal{P},M,i}_t+\frac{\delta}{\sqrt{2}}\widetilde{\Delta})-\frac{1}{N}\sum_{i=1}^N\mathbb{E}f(X^{\mathcal{P},M,i}_t+\delta\widetilde{\Delta})+\Vert f\Vert_\infty\times \mathcal{O}((\vert\mathcal{P}\vert+\sqrt{\varepsilon_M})^{1-\varepsilon}).\quad\quad\label{004*}
\end{eqnarray}

\end{theorem}
\begin{remark}
In the case $\theta=\infty$, the results in  \textbf{Theorem 2.3} and \textbf{Theorem 2.4} hold for every $0<t\leq T$.
\end{remark}\begin{remark}We remark that we have determined $\delta, N$, and we obtain an explicit formula to simulate the density function $p_t(x)$, which is a function solution of the analytical equation (\ref{1.4}). We also give the error of this simulation scheme explicitly. We notice that (\ref{003*}), the scheme based on Romberg method, gives a faster speed of convergence than (\ref{003}): we have the power $\frac{4}{d+5}>\frac{2}{d+3}$.
\end{remark}\begin{remark}
We mention that we obtain the results of \textbf{Theorem 2.3}  directly without using the previous estimates (\textbf{Theorem 2.2}), but  the speed of convergence  depends on the dimension $d$. So when $d$ is large, the speed of convergence is very slow. However for \textbf{Theorem 2.4},  we need to use the previous estimates \textbf{Theorem 2.2} to obtain (\ref{004}). The advantage of considering (\ref{004}) is that the speed of convergence no longer depends on the dimension $d$. So we keep the speed of convergence even for large dimension. Finally, (\ref{004*}) is a better simulation scheme in the sense that the numbers of particles $N$  we need is smaller than the one in (\ref{004}) and $\delta$ is larger than the one in (\ref{004}). We also stress that the speed of convergence in (\ref{004}) and (\ref{004*}) is $(\vert\mathcal{P}\vert+\sqrt{\varepsilon_M})^{1-\varepsilon}$, the same as in \textbf{Theorem 2.2} (\ref{002*}) for the truncated Euler scheme.
\end{remark}

The proof of this theorem will  be given in Section 3.3
by using Malliavin integration by parts techniques.

\bigskip\bigskip

\subsection{Some examples}
We give some typical examples to illustrate our main results.

\textbf{Example 1}
We take $h=1$ so the measure $\mu$ is the Lebesgue measure. We consider two types of behaviour for $c$.

\textbf{i) Exponential decay}
We assume that $\vert\bar{c}(z)\vert^2=e^{-a_1\vert z\vert^p}$ and $\underline{c}(z)=e^{-a_2\vert z\vert^p}$ with some constants $0<a_1\leq a_2$, $p>0$. We only check \textbf{Hypothesis 2.4} here. We have
\[
\nu\{\underline{c}>\frac{1}{u}\}= \nu\{\vert z\vert<(\frac{\ln u}{a_2})^{\frac{1}{p}}\}\geq\frac{r_d}{2}(\frac{\ln (u-1)}{a_2})^{\frac{d}{p}},
\]with $r_d$ the volume of the unit ball in $\mathbb{R}^d$, so that 
\[\frac{1}{\ln u}\nu\{\underline{c}>\frac{1}{u}\}\geq \frac{r_d}{2(a_2)^{\frac{d}{p}}}\frac{(\ln (u-1))^{\frac{d}{p}}}{\ln u}.\]
We notice that $\theta=0$ when $p>d$; $\theta=\infty$ when $0<p<d$; and $\theta=\frac{r_d}{2a_2}$ when $p=d$. Therefore, when $p>d$, we can say nothing in \textbf{Theorem 2.1} and \textbf{Theorem 2.2}; when $0<p<d$, all the results in \textbf{Theorem 2.1} and \textbf{Theorem 2.2} are true for every $0<t\leq T$; and when $p=d$, (\ref{001}) holds true for $t>\frac{8d(3l+2)a_2}{r_d}$, (\ref{002}) and (\ref{002*}) hold true for $t>\frac{16da_2}{r_d}(\frac{2}{\varepsilon}+1)$.

\textbf{ii) Polynomial decay}
We assume that $\vert\bar{c}(z)\vert^2=\frac{a_1}{1+\vert z\vert^p}$ and $\underline{c}(z)=\frac{a_2}{1+\vert z\vert^p}$ for some constants $0<a_2\leq a_1$ and $p>d$. Then
\[
\nu\{\underline{c}>\frac{1}{u}\}= \nu\{\vert z\vert<(a_2u-1)^{\frac{1}{p}}\}\geq\frac{r_d}{2}(a_2(u-1)-1)^{\frac{d}{p}},
\]so that 
\[\frac{1}{\ln u}\nu\{\underline{c}>\frac{1}{u}\}\geq \frac{r_d}{2}\frac{(a_2(u-1)-1)^{\frac{d}{p}}}{\ln u}.\]We notice that in this case, $\theta=\infty$. Thus, all the results in \textbf{Theorem 2.1} and \textbf{Theorem 2.2} holds for every $0<t\leq T$.

\bigskip

\textbf{Example 2}
We consider the ($1-$dimensional) truncated $\alpha-$stable  process: 
$X_t=X_0+\int_0^t\sigma(X_{r-})dU_r.$
Here $(U_t)_{t\geq0}$ is a (pure jump) L\'{e}vy process with intensity measure \[\widetilde{\mu}(dz)=\mathbbm{1}_{\{\vert z\vert\leq1\}}\frac{1}{\vert z\vert^{1+\alpha}}dz,\quad 0\leq \alpha<1.\] We assume that $\sigma\in C_b^\infty(\mathbb{R})$, $0<\underline{\sigma}\leq\sigma(x)\leq\bar{\sigma}$ and $-1<\underline{a}\leq\sigma'(x)\leq\bar{\sigma},\ \forall x\in\mathbb{R}$, for some universal constants $\bar{\sigma},\underline{\sigma},\underline{a}$, where $\sigma'$ is the differential of $\sigma$ in $x$.  Then by a change of variable $z\mapsto\frac{1}{z}$, we come back to the setting of this paper with $c(r,v,z,x,\rho)=\sigma(x)\times\frac{1}{z}$ and ${\mu}(dz)=\mathbbm{1}_{\{\vert z\vert\geq1\}}\frac{1}{\vert z\vert^{1-\alpha}}dz$.  In this case, $\underline{c}(z)=\underline{\sigma}\times\frac{1}{\vert z\vert^4}$, then
\[\frac{1}{\ln u}\nu\{\underline{c}>\frac{1}{u}\}\geq\frac{1}{\ln u}\int_1^{(\underline{\sigma}(u-1))^{\frac{1}{4}}}\frac{1}{\vert z\vert^{1-\alpha}}dz=\frac{(\underline{\sigma}(u-1))^{\frac{\alpha}{4}}-1}{\alpha\ln u},\]so that $\theta=\infty$. Thus, all the results in \textbf{Theorem 2.1} and \textbf{Theorem 2.2} hold for every $0<t\leq T$.

\bigskip\bigskip

\subsection{Preliminaries: coupling}
Before moving on to the next section, we make some preliminary computations here.
For some technical reasons, besides the truncated Euler scheme (\ref{truncation}), we also consider the truncation of the original equation (\ref{1.1})  as follows (with $a^M_T$, $\Delta$ and $c_M$ defined in Section 2.4).
\begin{eqnarray}
X^{M}_{t}&=&X_0+a^M_{T}\Delta+\int_{0}^{t}b(r,X^{M}_{r},\rho_{r})dr  \nonumber\\
&+&\int_{0}^{t}\int_{\mathbb{R}^d\times \mathbb{R}^d}{c}_M(r,v,z,X^{M}_{r-},\rho_{r-})N_{\rho_{r-}}(dv,dz,dr).  \label{truncationM}
\end{eqnarray} We notice that we keep $\rho_{r}$ (the law of $X_r$) instead of taking $\rho_{r}^M$ (the law of $X^M_r$) to simplify the calculation below, so the equation (\ref{truncationM}) is just an intermediate equation (which is not used for simulation).

We notice that the jump's parts of $X^{\mathcal{P},M}_{t}$ and  $X^{M}_{t}$ solutions of (\ref{truncation}), (\ref{truncationM})  are defined with respect to different Poisson point measures (on different probability spaces), so it is not possible to estimate the $L^2$ distance between them directly (we need to estimate the $L^2$ distance later in the proof of \textbf{Lemma 3.9}). To overcome this difficulty, we use similar
equations  driven by the same Poisson point measure. This is done by a coupling procedure. In this section, we make a coupling argument to construct $x_{t},x_{t}^{\mathcal{P}}$ and $x_{t}^{%
\mathcal{P},M}$ which have the same law as $X_{t},X_{t}^{\mathcal{P}}$ and $%
X_{t}^{\mathcal{P},M}$ but are defined on the same probability space and
verify equations driven by the same Poisson point measure.

We remark that the basic distance which appears in our framework is $W_1$ (see (\ref{W1})). However for technical reasons, we need to estimate the distance $W_{2+\varepsilon_\ast}$ (defined immediately below) for some small $\varepsilon_\ast>0$. This is because we need $L^2$ estimate in \textbf{Lemma 3.9} and we have to use the H$\ddot{o}$lder inequality with conjugates $1+\frac{\varepsilon_\ast}{2}$ and $\frac{2+\varepsilon_\ast}{\varepsilon_\ast}$. So now we take $\varepsilon_{\ast}>0$ which is small enough. For $\rho_1,\rho_2\in\mathcal{P}_{2+\varepsilon_{\ast}}(\mathbb{R}^d)$, we denote the Wasserstein distance of order $2+\varepsilon_{\ast}$ by
\[W_{2+\varepsilon_{\ast}}(\rho_1,\rho_2)=\inf_{\pi\in\Pi(\rho_1,\rho_2)}\big\{(\int_{\mathbb{R}^d\times\mathbb{R}^d}\vert x-y\vert^{2+\varepsilon_{\ast}}\pi(dx,dy))^{\frac{1}{2+\varepsilon_{\ast}}}\big\},\]where $\Pi(\rho_1,\rho_2)$ is the set of probability measures on $\mathbb{R}^d\times\mathbb{R}^d$ with marginals $\rho_1$ and $\rho_2$. Some basic properties of $W_p, p\geq1$ can be found in $\cite{ref19}$ and $\cite{ref22}$ for example, and we mention that $W_1(\rho_1,\rho_2)\leq W_{2+\varepsilon_{\ast}}(\rho_1,\rho_2)$. 

Now we make the optimal coupling in $W_{2+\varepsilon_{\ast}}$ distance between $X^{\mathcal{P},M}_{\tau(t)-}$  and $X_{t-}$. We recall that $\rho^{\mathcal{P},M}_{\tau(t)-}$ is the law of $X^{\mathcal{P},M}_{\tau(t)-}$  and $\rho_{t-}$ is the law of  $X_{t-}$. For every partition $\mathcal{P}$, $M\in\mathbb{N}$ and time $0<t\leq T$, one can easily check that $\rho^{\mathcal{P},M}_{\tau(t)-}$  and $\rho_{t-}$ both belong to $\mathcal{P}_{2+\varepsilon_{\ast}}(\mathbb{R}^d)$. This is a consequence of \textbf{Hypothesis 2.1} and of (\ref{BurkN}) with \[\Phi(r,v,z,\omega,\rho)={c}_M(\tau(r),v,z,X^{\mathcal{P},M}_{\tau(r)-},\rho^{\mathcal{P},M}_{\tau(r)-})\] and with \[\Phi(r,v,z,\omega,\rho)={c}(r,v,z,X_{r-},\rho_{r-}).\] 
Then we take   $\Pi^{\mathcal{P},M}_t(dv_1,dv_2)$ to be the optimal $W_{2+\varepsilon_{\ast}}-$coupling of $\rho^{\mathcal{P},M}_{\tau(t)-}(dv_1)$ and $\rho_{t-}(dv_2)$. So we  have \begin{eqnarray*}
(W_{2+\varepsilon_{\ast}}(\rho^{\mathcal{P},M}_{\tau(t)-},\rho_{t-}))^{2+\varepsilon_{\ast}}&=&\int_{\mathbb{R}^d\times\mathbb{R}^d}\vert v_1-v_2\vert^{2+\varepsilon_{\ast}}\Pi^{\mathcal{P},M}_t(dv_1,dv_2).\\
\end{eqnarray*}
We will need the representation of $\Pi^{\mathcal{P},M}_t(dv_1,dv_2)$ by means of the Lebesgue measure $dw$ on $[0,1]$. This will be done by using the following lemma.
\begin{lemma}
There exists a measurable map $\Phi:[0,1)\times\mathcal{P}_1(\mathbb{R}^d)\rightarrow\mathbb{R}^d$ such that for any $\rho\in\mathcal{P}_1(\mathbb{R}^d)$, any bounded and measurable function $\phi:\mathbb{R}^d\rightarrow\mathbb{R}$, we have
\[\int_0^1\phi(\Phi(w,\rho))dw=\int_{\mathbb{R}^d}\phi(x)\rho(dx).\]
\end{lemma} This result is stated in  $\cite{ref24}$ and is useful when we estimate the $L^p$ distance.
We construct $(\eta^1_t(w),\eta^2_t(w))$ which represents $\Pi^{\mathcal{P},M}_t$ in the sense of \textbf{Lemma 2.5}, this means 
\[\int_0^1\phi(\eta^1_t(w),\eta^2_t(w)) dw=\int_{\mathbb{R}^d\times\mathbb{R}^d}\phi(v_1,v_2)\Pi^{\mathcal{P},M}_t(dv_1,dv_2).\]
In particular, this gives for any measurable and bounded function $f:\mathbb{R}^d\rightarrow\mathbb{R}$,
\begin{eqnarray}
&\int_0^1 f(\eta^1_t(w)) dw=\int_{\mathbb{R}^d} f(v_1)\rho^{\mathcal{P},M}_{\tau(t)-}(dv_1),\quad \int_0^1 f(\eta^2_t(w)) dw=\int_{\mathbb{R}^d} f(v_2)\rho_{t-}(dv_2),\nonumber\\
&\int_0^1\vert \eta^1_t(w)-\eta^2_t(w)\vert^{2+\varepsilon_{\ast}} dw=\int_{\mathbb{R}^d\times\mathbb{R}^d}\vert v_1-v_2\vert^{2+\varepsilon_{\ast}}\Pi^{\mathcal{P},M}_t(dv_1,dv_2)=(W_{2+\varepsilon_{\ast}}(\rho^{\mathcal{P},M}_{\tau(t)-},\rho_{t)-}))^{2+\varepsilon_{\ast}}.\quad\label{coupling}
\end{eqnarray}
Now we construct a Poisson point measure $\mathcal{N}(dw,dz,dr)$ on the state space $[0,1]\times\mathbb{R}^d$ with intensity measure $dw\mu(dz)dr$. Then we consider the equations
\begin{eqnarray}
x_{t}&=&X_0+\int_{0}^{t}b(r,x_{r},\rho_{r})dr +\int_{0}^{t}\int_{[0,1]\times \mathbb{R}^d}{c}(r,\eta^2_r(w),z,x_{r-},\rho_{r-})\mathcal{N}(dw,dz,dr),  \label{x}
\end{eqnarray}
\begin{eqnarray}
x^{M}_{t}&=&X_0+a^M_{T}\Delta+\int_{0}^{t}b(r,x^{M}_{r},\rho_{r})dr +\int_{0}^{t}\int_{[0,1]\times \mathbb{R}^d}{c}_M(r,\eta^2_r(w),z,x^{M}_{r-},\rho_{r-})\mathcal{N}(dw,dz,dr),  \label{xM}
\end{eqnarray}
\begin{eqnarray}
x^{\mathcal{P},M}_{t}&=&X_0+a^M_{T}\Delta+\int_{0}^{t}b(\tau(r),x^{\mathcal{P},M}_{\tau(r)},\rho^{\mathcal{P},M}_{\tau(r)})dr  \nonumber\\
&+&\int_{0}^{t}\int_{[0,1]\times \mathbb{R}^d}{c}_M(\tau(r),\eta^1_r(w),z,x^{\mathcal{P},M}_{\tau(r)-},\rho^{\mathcal{P},M}_{\tau(r)-})\mathcal{N}(dw,dz,dr).  \label{xPM}
\end{eqnarray}
One can check by It$\hat{o}$ formula that $x^{\mathcal{P},M}_{t}$ has the same law as $X^{\mathcal{P},M}_{t}$ (solution of (\ref{truncation})), $x^{M}_{t}$ has the same law as $X^{M}_{t}$ (solution of (\ref{truncationM})) and $x_{t}$ has the same law as $X_{t}$ (solution of (\ref{1.3})). 
Then \begin{eqnarray}
(W_{2+\varepsilon_{\ast}}(\rho^{\mathcal{P},M}_{\tau(t)-},\rho_{t-}))^{2+\varepsilon_{\ast}}=(W_{2+\varepsilon_{\ast}}(\mathcal{L}(X^{\mathcal{P},M}_{\tau(t)-}),\mathcal{L}(X_{t-})))^{2+\varepsilon_{\ast}}&=&(W_{2+\varepsilon_{\ast}}(\mathcal{L}(x^{\mathcal{P},M}_{\tau(t)-}),\mathcal{L}(x_{t-})))^{2+\varepsilon_{\ast}}\nonumber\\&\leq& \mathbb{E}\vert x^{\mathcal{P},M}_{\tau(t)-}- x_{t-}\vert^{2+\varepsilon_{\ast}}.\label{Coupling}
\end{eqnarray}

\begin{remark}
We also have  the following consequence  of Burkholder inequality (as a variant of (\ref{BurkN}) and (\ref{Burk})): Let $\bar{\Phi}(r,w,z,\omega,\rho):[0,T]\times[0,1]\times\mathbb{R}^d\times\Omega\times\mathcal{P}_1(\mathbb{R}^d)\rightarrow\mathbb{R}_{+}$ and  $\bar{\varphi}(r,w,\omega,\rho):[0,T]\times[0,1]\times\Omega\times\mathcal{P}_1(\mathbb{R}^d)\rightarrow\mathbb{R}_{+}$ be two non-negative functions. 

$a)$ Then for any $p\geq2$,
\begin{eqnarray}
	&&\mathbb{E}\Big\vert\int_0^t\int_{[0,1]\times\mathbb{R}^d}\bar{\Phi}(r,w,z,\omega,\rho)\mathcal{N}(dw,dz,dr)\Big\vert^p\nonumber\\
	&&\leq C [\mathbb{E}\int_0^t\int_0^1\int_{\mathbb{R}^d}\vert\bar{\Phi}(r,w,z,\omega,\rho)\vert^p\mu(dz) dwdr+\mathbb{E}\int_0^t\int_0^1\vert\int_{\mathbb{R}^d}\vert\bar{\Phi}(r,w,z,\omega,\rho)\vert\mu(dz)\vert^p dwdr\nonumber\\
	&&+\mathbb{E}\int_0^t\int_0^1\vert\int_{\mathbb{R}^d}\vert\bar{\Phi}(r,w,z,\omega,\rho)\vert^2\mu(dz)\vert^{\frac{p}{2}} dwdr], \label{Burk*}
\end{eqnarray}
where $C$ is a constant depending on $p$, $T$.

$b)$ If we have 
\[
\vert\bar{\Phi}(r,w,z,\omega,\rho)\vert\leq\vert\bar{c}(z)\vert\vert\bar{\varphi}(r,w,\omega,\rho)\vert,
\]
then for any $p\geq2$,
\begin{eqnarray}
	\mathbb{E}\Big\vert\int_0^t\int_{[0,1]\times\mathbb{R}^d}\bar{\Phi}(r,w,z,\omega,\rho)\mathcal{N}(dw,dz,dr)\Big\vert^p
	\leq C \mathbb{E}\int_0^t\int_0^1\vert\bar{\varphi}(r,w,\omega,\rho)\vert^p dwdr. \label{BurkM}
\end{eqnarray}
\end{remark}

Then we obtain the following consequence. We recall by (\ref{epsM}) that
$\varepsilon_M=\int_{\{\vert z\vert>M\}}\vert\bar{c}(z)\vert^2\mu(dz)+\vert\int_{\{\vert z\vert>M\}}\bar{c}(z)\mu(dz)\vert^2).$
\begin{lemma}
Assume that the \textbf{Hypothesis 2.1} holds true.
Then there exists a  constant $C$ dependent on $T$ and $\varepsilon_{\ast}$, for every $M$ such that $\varepsilon_M\leq 1$ and $\vert\bar{c}(z)\vert^2\mathbbm{1}_{\{\vert z\vert>M\}}\leq 1$, we have \[i)\quad \mathbb{E}\vert {x}_{t}^{M}-x_{t}\vert^{2+\varepsilon_{\ast}}\leq C\varepsilon_M\rightarrow0.\]
And  for every partition $\mathcal{P}$ with $\vert\mathcal{P}\vert\leq1$, we have
\[ii)\quad\mathbb{E}\vert {x}_{t}^{\mathcal{P},M}-{x}_{t}^{M}\vert^{2+\varepsilon_{\ast}}\leq C(\vert\mathcal{P}\vert+\varepsilon_M),\]
\[iii)\quad \mathbb{E}\vert {x}_{t}^{\mathcal{P},M}-{x}_{t}\vert^{2+\varepsilon_{\ast}}\leq C(\vert\mathcal{P}\vert+\varepsilon_M),\]
\[iv)\quad W_{2+\varepsilon_{\ast}}({x}_{t}^{\mathcal{P},M},{x}_{t})\leq C(\vert\mathcal{P}\vert+\varepsilon_M)^{\frac{1}{2+\varepsilon_{\ast}}}.\]
\end{lemma}
\begin{proof}We only prove $i)$ and $iii)$, since $ii)$ is a direct consequence of $i)$ and $iii)$, and $iv)$ is an immediate consequence of $iii)$.

\textbf{Proof of $i)$:}    We write $\mathbb{E}\vert x_{t}^{M}-x_{t}\vert^{2+\varepsilon_{\ast}}\leq E_0+E_1+E_2$, where $E_0=\vert a^M_{T}\vert^{2+\varepsilon_{\ast}}\mathbb{E}\vert \Delta\vert^{2+\varepsilon_{\ast}}\leq C\varepsilon_M$, and
\begin{eqnarray*}
E_1=\mathbb{E}\vert\int_{0}^{t} (b(r,x_{r}^{M
},\rho_{r})-b(r,x_{r},\rho_{r}))dr\vert^{2+\varepsilon_{\ast}},
\end{eqnarray*}
\begin{eqnarray*}
E_2&=&\mathbb{E}\vert\int_{0}^{t}\int_{[0,1]\times\mathbb{R}^d}({c}_M(r,\eta^2_r(w),z,x^{M}_{r-},\rho_{r-})-{c}(r,\eta^2_r(w),z,x_{r-},\rho_{r-}))\mathcal{N}(dw,dz,dr)\vert^{2+\varepsilon_{\ast}}
\end{eqnarray*}
Firstly, by \textbf{Hypothesis 2.1}, \begin{eqnarray}
E_1\leq C\int_0^t\mathbb{E}\vert x_{r}^{M}-x_{r}\vert^{2+\varepsilon_{\ast}} dr. \label{E1}
\end{eqnarray}
Then by \textbf{Hypothesis 2.1}, (\ref{Burk*}) with \[\bar{\Phi}(r,w,z,\omega,\rho)=\vert{c}_M(r,\eta^2_r(w),z,x_{r-},\rho_{r-})-{c}(r,\eta^2_r(w),z,x_{r-},\rho_{r-})\vert\] and by (\ref{BurkM}) with \[\bar{\Phi}(r,w,z,\omega,\rho)=\vert{c}_M(r,\eta^2_r(w),z,x^{M}_{r-},\rho_{r-})-{c}_M(r,\eta^2_r(w),z,x_{r-},\rho_{r-})\vert,\]we have \begin{eqnarray}
E_2
&\leq& \mathbb{E}\vert\int_{0}^{t}\int_{[0,1]\times\mathbb{R}^d}({c}_M(r,\eta^2_r(w),z,x_{r-},\rho_{r-})-{c}(r,\eta^2_r(w),z,x_{r-},\rho_{r-}))\mathcal{N}(dw,dz,dr)\vert^{2+\varepsilon_{\ast}}\nonumber\\
&+&\mathbb{E}\vert\int_{0}^{t}\int_{[0,1]\times\mathbb{R}^d}({c}_M(r,\eta^2_r(w),z,x^{M}_{r-},\rho_{r-})-{c}_M(r,\eta^2_r(w),z,x_{r-},\rho_{r-}))\mathcal{N}(dw,dz,dr)\vert^{2+\varepsilon_{\ast}}\nonumber\\
&\leq& C[\int_{\{\vert z\vert>M\}}\vert\bar{c}(z)\vert^{2+\varepsilon_{\ast}}\mu(dz)+\vert\int_{\{\vert z\vert>M\}}\bar{c}(z)\mu(dz)\vert^{2+\varepsilon_{\ast}}+\vert\int_{\{\vert z\vert>M\}}\vert\bar{c}(z)\vert^{2}\mu(dz)\vert^{\frac{2+\varepsilon_{\ast}}{2}}\nonumber\\
&+& \int_0^t\mathbb{E}\vert x_{r}-x_{r}^{M}\vert^{2+\varepsilon_{\ast}} dr]\nonumber\\
&\leq&  C[\varepsilon_M+ \int_0^t\mathbb{E}\vert x_{r}-x_{r}^{M}\vert^{2+\varepsilon_{\ast}} dr]. \label{E2}
\end{eqnarray}
Combining (\ref{E1}) and (\ref{E2}), we have
\begin{eqnarray*}
\mathbb{E}\vert x_{t}^{M}-x_{t}\vert^{2+\varepsilon_{\ast}}&\leq& C[\varepsilon_M+\int_0^t\mathbb{E}\vert x^M_r-x_r\vert^{2+\varepsilon_{\ast}}dr] . 
\end{eqnarray*}
So we conclude by Gronwall lemma.

\bigskip\bigskip

\textbf{Proof of $iii)$}
We write $\mathbb{E}\vert {x}_{t}^{\mathcal{P},M}-{x}_{t}\vert^{2+\varepsilon_{\ast}}\leq C[K_0+K_1+K_2]$, with $K_0=\vert a^M_T\vert^{2+\varepsilon_{\ast}}\mathbb{E}\vert\Delta\vert^{2+\varepsilon_{\ast}}\leq C\varepsilon_M$, and
\[K_1=\mathbb{E}\vert\int_{0}^{t}b(\tau(r),x^{\mathcal{P},M}_{\tau(r)},\rho^{\mathcal{P},M}_{\tau(r)})-b(r,x_{r},\rho_{r})dr\vert^{2+\varepsilon_{\ast}},\]
\[K_2=\mathbb{E}\vert \int_{0}^{t}\int_{[0,1]\times \mathbb{R}^d}{c}_M(\tau(r),\eta^1_r(w),z,x^{\mathcal{P},M}_{\tau(r)-},\rho^{\mathcal{P},M}_{\tau(r)-})-{c}(r,\eta^2_r(w),z,x_{r-},\rho_{r-})\mathcal{N}(dw,dz,dr)\vert^{2+\varepsilon_{\ast}}.\]
Using \textbf{Hypothesis 2.1},
\begin{eqnarray}
K_1&\leq& C[\vert\mathcal{P}\vert^{2+\varepsilon_{\ast}}+\int_{0}^{t}\mathbb{E}\vert x^{\mathcal{P},M}_{\tau(r)}-x_{r}\vert^{2+\varepsilon_{\ast}} dr+\int_0^t(W_1(\rho^{\mathcal{P},M}_{\tau(r)},\rho_{r}))^{2+\varepsilon_{\ast}}dr]\nonumber\\
&\leq&C[\vert\mathcal{P}\vert^{2+\varepsilon_{\ast}}+\int_{0}^{t}\mathbb{E}\vert x^{\mathcal{P},M}_{\tau(r)}-x_{r}\vert^{2+\varepsilon_{\ast}} dr].\label{K1}
\end{eqnarray}
By  \textbf{Hypothesis 2.1}, (\ref{BurkM}) with \[\bar{\Phi}(r,w,z,\omega,\rho)=\vert{c}(\tau(r),\eta^1_r(w),z,x^{\mathcal{P},M}_{\tau(r)-},\rho^{\mathcal{P},M}_{\tau(r)-})-{c}(r,\eta^2_r(w),z,x_{r-},\rho_{r-})\vert\] and by (\ref{Burk*}) with \[\bar{\Phi}(r,w,z,\omega,\rho)=\vert{c}_M(\tau(r),\eta^1_r(w),z,x^{\mathcal{P},M}_{\tau(r)-},\rho^{\mathcal{P},M}_{\tau(r)-})-{c}(\tau(r),\eta^1_r(w),z,x^{\mathcal{P},M}_{\tau(r)-},\rho^{\mathcal{P},M}_{\tau(r)-})\vert,\]we have
\begin{eqnarray}
K_2&\leq&C[\mathbb{E}\vert \int_{0}^{t}\int_{[0,1]\times \mathbb{R}^d}({c}(\tau(r),\eta^1_r(w),z,x^{\mathcal{P},M}_{\tau(r)-},\rho^{\mathcal{P},M}_{\tau(r)-})-{c}(r,\eta^2_r(w),z,x_{r-},\rho_{r-}))\mathcal{N}(dw,dz,dr)\vert^{2+\varepsilon_{\ast}}\nonumber\\
&+&\mathbb{E}\vert \int_{0}^{t}\int_{[0,1]\times \mathbb{R}^d}({c}_M(\tau(r),\eta^1_r(w),z,x^{\mathcal{P},M}_{\tau(r)-},\rho^{\mathcal{P},M}_{\tau(r)-})\nonumber\\
&-&{c}(\tau(r),\eta^1_r(w),z,x^{\mathcal{P},M}_{\tau(r)-},\rho^{\mathcal{P},M}_{\tau(r)-}))\mathcal{N}(dw,dz,dr)\vert^{2+\varepsilon_{\ast}}]\nonumber\\
&\leq&C[\vert\mathcal{P}\vert^{2+\varepsilon_{\ast}}+\int_0^t\int_0^1\vert \eta^1_r(w)-\eta^2_r(w)\vert^{2+\varepsilon_{\ast}} dwdr+\int_{0}^{t}\mathbb{E}\vert x^{\mathcal{P},M}_{\tau(r)}-x_{r}\vert^{2+\varepsilon_{\ast}} dr+\int_0^t(W_1(\rho^{\mathcal{P},M}_{\tau(r)},\rho_{r}))^{2+\varepsilon_{\ast}}dr\nonumber\\
&+&\int_{\{\vert z\vert>M\}}\vert\bar{c}(z)\vert^{2+\varepsilon_{\ast}}\mu(dz)+\vert\int_{\{\vert z\vert>M\}}\bar{c}(z)\mu(dz)\vert^{2+\varepsilon_{\ast}}+\vert\int_{\{\vert z\vert>M\}}\vert\bar{c}(z)\vert^{2}\mu(dz)\vert^{\frac{2+\varepsilon_{\ast}}{2}}]\nonumber\\
&\leq& C[\vert\mathcal{P}\vert^{2+\varepsilon_{\ast}}+\int_{0}^{t}\mathbb{E}\vert x^{\mathcal{P},M}_{\tau(r)}-x_{r}\vert^{2+\varepsilon_{\ast}} dr+\varepsilon_M],\label{K2}
\end{eqnarray}where the last inequality is obtained by (\ref{coupling}), (\ref{Coupling}), and the fact that $W_1$ distance is upper bounded by $W_{2+\varepsilon_\ast}$ distance, and so upper bounded by the $L^{2+\varepsilon_\ast}$ distance.

We notice that by (\ref{BurkM}) with \[\bar{\Phi}(r,w,z,\omega,\rho)={c}_M(\tau(r),\eta^1_r(w),z,x^{\mathcal{P},M}_{\tau(r)-},\rho^{\mathcal{P},M}_{\tau(r)-}),\]we have
\begin{eqnarray}
\mathbb{E}\vert x^{\mathcal{P},M}_{\tau(t)}-x^{\mathcal{P},M}_{t}\vert^{2+\varepsilon_{\ast}} \leq C\vert\mathcal{P}\vert.\label{Eul}
\end{eqnarray}
Combining (\ref{K1}), (\ref{K2}) and (\ref{Eul}), \[\mathbb{E}\vert {x}_{t}^{\mathcal{P},M}-{x}_{t}\vert^{2+\varepsilon_{\ast}}\leq C[K_0+K_1+K_2]\leq C[\vert\mathcal{P}\vert+\int_{0}^{t}\mathbb{E}\vert x^{\mathcal{P},M}_{r}-x_{r}\vert^{2+\varepsilon_{\ast}} dr+\varepsilon_M].\]
So finally, we  conclude by Gronwall lemma.
\end{proof}
We remark that we may represent the jump's parts of the equations (\ref{xM}) and (\ref{xPM}) by means of compound Poisson processes. With all the random variables $(({J}_t^{k})_{\substack{t\in[0,T]\\k\in\mathbb{N}}}, ({Z}_i^{k})_{k,i\in\mathbb{N}}$, $ X_0,\Delta)$ constructed in Section 2.4, we take moreover $(W_i^k)_{k,i\in\mathbb{N}}$ a sequence of independent random variables which are uniformly distributed on $[0,1]$ and independent of $(({J}_t^{k})_{\substack{t\in[0,T]\\k\in\mathbb{N}}}, ({Z}_i^{k})_{k,i\in\mathbb{N}}$, $ X_0,\Delta)$. Then we have
\begin{eqnarray}
x^{M}_{t}&=&X_0+a^M_{T}\Delta+\int_{0}^{t}b(r,x^{M}_{r},\rho_{r})dr  +\sum_{k=1}^{M}\sum_{i=1}^{J^k_t}{c}({T}_i^{k},\eta^2_{T^k_i}({W}_i^k),{Z}_i^{k},x^{M}_{{T}_i^{k}-},\rho_{{T}_i^{k}-}),   \label{xMZ}
\end{eqnarray}
\begin{eqnarray}
x^{\mathcal{P},M}_{t}&=&X_0+a^M_{T}\Delta+\int_{0}^{t}b(\tau(r),x^{\mathcal{P},M}_{\tau(r)},\rho^{\mathcal{P},M}_{\tau(r)})dr  +\sum_{k=1}^{M}\sum_{i=1}^{J^k_t}{c}(\tau({T}_i^{k}),\eta^1_{T^k_i}({W}_i^k),{Z}_i^{k},x^{\mathcal{P},M}_{\tau({T}_i^{k})-},\rho^{\mathcal{P},M}_{\tau({T}_i^{k})-}).  \nonumber\\ \label{xPMZ}
\end{eqnarray}

We recall that the laws of $x_{t}$ and $X_{t}$ coincide, $x^{\mathcal{P},M}_{t}$ has the same law as $X^{\mathcal{P},M}_{t}$, and $x^{M}_{t}$ has the same law as $X^{M}_{t}$. The advantage of considering $x_t$,  $x^M_t$ and $x^{\mathcal{P},M}_t$ is that the jump's parts of them are all defined with respect to the same Poisson point measure, which means that we are able to overcome the problems caused by the "Boltzmann term" (the Poisson point measure depends on the law of the solution). So in the following, instead of dealing with $X_t$,  $X^M_t$ and $X^{\mathcal{P},M}_t$ solutions of (\ref{1.1}),  (\ref{truncationM}) and (\ref{truncation}), we deal with $x_t$,  $x^M_t$ and $x^{\mathcal{P},M}_t$ solutions of (\ref{x}),  (\ref{xMZ}) and (\ref{xPMZ}).

\bigskip\bigskip

\section{Malliavin calculus}
\subsection{Abstract integration by parts framework} %
Here we recall the abstract integration by parts framework  in $\cite{ref2}$. 

We denote $C_p^\infty$ to be the space of smooth functions which, together with all the derivatives, have polynomial growth. We also denote $C_p^q$ to be the space of $q-$times differentiable functions which, together with all the derivatives, have polynomial growth.

We consider a probability space ($\Omega$,$\mathcal F$,$\mathbb{P}$), and a subset $\mathcal S\subset\mathop{\bigcap}\limits_{p=1}^\infty L^p(\Omega;\mathbb{R})$ such that for every $\phi\in C_p^\infty(\mathbb{R}^d)$ and every $F\in\mathcal S^d$, we have $\phi(F)\in\mathcal S$. A typical example of $\mathcal{S}$ is the space of simple functionals, as in the standard Malliavin calculus. Another example is the space of "Malliavin smooth functionals".

Given a separable Hilbert space $\mathcal{H}$, we assume that we have a derivative operator $D: \mathcal S\rightarrow\mathop{\bigcap}\limits_{p=1}^\infty L^p(\Omega;\mathcal{H})$ which is a linear application which satisfies

$a)$ 
\begin{eqnarray}
D_hF:=\langle DF,h\rangle_\mathcal{H}\in\mathcal{S},\ for\ any\ h\in\mathcal{H}, \label{0.a}
\end{eqnarray}

$b)$ $\underline {Chain\ Rule}$: For every $\phi\in C_p^1(\mathbb{R}^d)$ and $F=
(F_1,\cdots,F_d)\in \mathcal S^d$, we have

\begin{align}
  D\phi(F)=\sum_{i=1}^d\partial_i\phi(F)DF_i , \label{0.00}
\end{align}

Since $D_hF\in\mathcal{S}$, we may define by iteration the derivative operator of higher order $D^q:\mathcal S\rightarrow\mathop{\bigcap}\limits_{p=1}^\infty L^p(\Omega;\mathcal{H}^{\otimes q})$ which verifies
$\langle D^qF,\otimes_{i=1}^q h_i\rangle_{\mathcal{H}^{\otimes q}}=D_{h_q}D_{h_{q-1}}\cdots D_{h_1}F$. We also denote $D_{h_1,\cdots,h_q}^qF:=\langle D^qF,\otimes_{i=1}^q h_i\rangle_{\mathcal{H}^{\otimes q}}$,\quad for any $h_1,\cdots,h_q\in\mathcal{H}$. Then, $D_{h_1,\cdots,h_q}^qF=$
$D_{h_q}D_{h_1,\cdots,h_{q-1}}^{q-1}F$ ($q\geq2$).

We notice that since $\mathcal{H}$ is separable, there exists a countable orthonormal base $(e_i)_{i\in\mathbb{N}}$. We denote  \[D_iF=D_{e_i}F=\langle DF,e_i\rangle_{\mathcal{H}}.\] Then \[DF=\sum_{i=1}^{\infty}D_iF\times e_i\quad \text{and}\quad D^qF=\sum_{i_1,\cdots,i_q}D_{i_1,\cdots,i_q}F\times\otimes_{j=1}^qe_j.\]

For $F=(F_1,\cdots,F_d)\in\mathcal{S}^d$, we associate the Malliavin covariance matrix
\begin{eqnarray}
\sigma_F=(\sigma_F^{i,j})_{i,j=1,\cdots,d},\quad \text{with}\quad \sigma_F^{i,j}=\langle DF_i,DF_j\rangle_\mathcal{H}. \label{Mcov}
\end{eqnarray}
And we denote 
\begin{eqnarray}
\Sigma_p(F)=\mathbb{E}(1/ \det\sigma_F)^p.  \label{sigma}
\end{eqnarray}
We say that $F$ is non-degenerated if $\Sigma_p(F)<\infty$, $\forall p\geq1$.

We also assume that we have an Ornstein-Uhlenbeck (divergence) operator $L:\mathcal S\rightarrow\mathcal S$ which is a linear operator satisfying 

$a)$ $\underline {Duality}$: For every $F,G\in\mathcal S$,

\begin{align}
\mathbb{E}\langle DF,DG\rangle_\mathcal{H}=\mathbb{E}(FLG)=\mathbb{E}(GLF),
\label{0.01}
\end{align}

$b)$ $\underline {Chain\ Rule}$: For every $\phi\in C_p^2(\mathbb{R}^d)$ and $F=
(F_1,\cdots,F_d)\in \mathcal S^d$, we have

\begin{align*}
  L\phi(F)=\sum_{i=1}^d\partial_i\phi(F)LF_i-\sum_{i=1}^d\sum_{j=1}^d\partial_i\partial_j\phi(F)\langle DF_i,DF_j\rangle_\mathcal{H}. 
\end{align*}
As an immediate consequence of the duality formula, we know that $L: \mathcal{S}\subset L^2(\Omega)\rightarrow L^2(\Omega)$ is closable.

\begin{definition}
If $D^q: \mathcal{S}\subset L^2(\Omega)\rightarrow L^2(\Omega;\mathcal{H}^{\otimes q})$, $\forall q\geq1$, are closable, then the triplet $(\mathcal{S},D,L)$ is called an IbP (Integration by Parts) framework.
\end{definition}

Now, we introduce the Sobolev norms. For any $l\geq1$, $F\in\mathcal{S}$,
\begin{eqnarray}
\left\vert F\right\vert_{1,l} &=&\sum_{q=1}^{l}\left\vert D^{q}F\right\vert_{\mathcal{H}^{\otimes q}},\quad \left\vert F\right\vert_{l}=\left\vert F\right\vert+\left\vert F\right\vert_{1,l},\label{norm}
\end{eqnarray}

We put $\vert F\vert_{0}=\vert F\vert$, $\vert F\vert_{l}=0$ for $l<0$, and $\vert F\vert_{1,l}=0$ for $l\leq0$. For $F=(F_1,\cdots,F_d)\in\mathcal{S}^d$, we set 
\begin{eqnarray}
\left\vert F\right\vert_{1,l} &=&\sum_{i=1}^{d}\left\vert F_i\right\vert_{1,l},\quad \left\vert F\right\vert_{l}=\sum_{i=1}^{d}\left\vert F_i\right\vert_l,\nonumber
\end{eqnarray}

Moreover, we associate the following norms. For any $l\geq0, p\geq1$,
\begin{eqnarray}
\left\Vert F \right\Vert _{l,p}&=&(\mathbb{E}\left\vert F \right\vert_{l}^{p})^{1/p}, \quad \left\Vert F \right\Vert _{p}=(\mathbb{E}\left\vert F \right\vert^{p})^{1/p} ,\nonumber \\
\left\Vert F\right\Vert _{L,l,p}&=&\left\Vert F\right\Vert _{l,p}+\left\Vert
LF\right\Vert _{l-2,p}. \label{sobnorm}
\end{eqnarray}

With these notations, we have the following lemma from $\cite{ref3}$ (lemma\ 8 and lemma\ 10), which is a consequence of the chain rule.
\begin{lemma}
Let $F\in\mathcal{S}^d$. For every $l\in\mathbb{N},$ if $\phi: \mathbb{R}^d\rightarrow\mathbb{R}$ is a $C^{l}(\mathbb{R}^d)$ function ($l-$times differentiable function), then there is a constant $C_l$ dependent on $l$ such that 
\[
a)\quad \vert\phi(F)\vert_{1,l}\leq\vert\nabla\phi(F)\vert\vert F\vert_{1,l}+C_l\sup_{2\leq\vert\beta\vert\leq l}\vert\partial^\beta\phi(F)\vert\vert F\vert_{1,l-1}^{l}.
\]
If $\phi\in C^{l+2}(\mathbb{R}^d)$, then
\[
b)\quad \vert L\phi(F)\vert_{l}\leq\vert\nabla\phi(F)\vert\vert LF\vert_{l}+C_l\sup_{2\leq\vert\beta\vert\leq l+2}\vert\partial^\beta\phi(F)\vert(1+\vert F\vert_{l+1}^{l+2})(1+\vert LF\vert_{l-1}).
\]
For $l=0$, we have
\[
c)\quad \vert L\phi(F)\vert\leq\vert\nabla\phi(F)\vert\vert LF\vert+\sup_{\vert\beta\vert=2}\vert\partial^\beta\phi(F)\vert\vert F\vert_{1,1}^{2}.
\]
\end{lemma}

We denote by $\mathcal{D}_{l,p}$ the closure of $\mathcal{S}$ with respect to the norm $\left\Vert \circ
\right\Vert _{L,l,p}:$ 
\begin{equation}
\mathcal{D}_{l,p}=\overline{\mathcal{S}}^{\left\Vert \circ \right\Vert
_{L,l,p}},  \label{3.9}
\end{equation}%
and
\begin{equation}
    \mathcal{D}_{\infty}=\mathop{\bigcap}\limits_{l=1}^{\infty}\mathop{\bigcap}\limits_{p=1}^{\infty}\mathcal{D}_{l,p},\quad \mathcal{H}_{l}=\mathcal{D}_{l,2}. \label{Dinf}
\end{equation}

For an IbP framework $(\mathcal{S},D,L)$, we now extend the operators from $\mathcal{S}$ to $\mathcal{D}_{\infty}$.
For $F\in \mathcal{D}_{\infty}$, $p\geq2$, there exists a sequence $F_{n}\in \mathcal{S}$ such that $\left\Vert F-F_{n}\right\Vert _{p}\rightarrow 0$, $\left\Vert F_{m}-F_{n}\right\Vert _{q,p}\rightarrow 0$ and $\left\Vert LF_{m}-LF_{n}\right\Vert _{q-2,p}\rightarrow 0$. Since $%
D^{q} $ and $L$ are closable, we can define
\begin{equation}
D^{q}F=\lim_{n\rightarrow\infty}D^{q}F_{n}\quad in\quad L^p(\Omega;\mathcal{H}^{\otimes q}),\quad LF=\lim_{n\rightarrow\infty}LF_{n}\quad in\quad L^p(\Omega).  \label{3.10}
\end{equation}%
We still associate the same norms and covariance matrix introduced above for $F\in\mathcal{D}_\infty$.

\begin{lemma}
The triplet $(\mathcal{D}_\infty,D,L)$ is an IbP framework.
\end{lemma}
\begin{proof}
The proof is standard and we refer to the lemma 3.1 in $\cite{ref4}$ for details.
\end{proof}

The following lemma is useful in order to control the Sobolev norms and covariance matrices when  passing to the limit.
\begin{lemma}
\textbf{(A)}\quad We fix $p\geq 2,l\geq2.$ Let $F\in  L^1(\Omega)$ and let $F_{n}\in \mathcal{S}^d,n\in\mathbb{N}$ such that 
\begin{eqnarray*}
i)\quad \mathbb{E}\left\vert F_{n}-F\right\vert &\rightarrow &0, \\
ii)\quad \sup_{n}\left\Vert F_{n}\right\Vert _{L,l,p} &\leq &K_{l,p}<\infty. 
\end{eqnarray*}%
Then for every $1\leq \bar{p}<p,$ we have $F\in \mathcal{D}_{l,\bar{p}}^d$ and $\left\Vert F\right\Vert
_{L,l,\bar{p}}\leq K_{l,\bar{p}}$ . Moreover, there exists a convex combination
\[
G_{n}=\sum_{i=n}^{m_{n}}\gamma _{i}^{n}\times F_{i}\in \mathcal{S}^d,
\]%
with $\gamma _{i}^{n}\geq 0,i=n,....,m_{n}$ and $%
\sum\limits_{i=n}^{m_{n}}\gamma _{i}^{n}=1$,
such that 
\[
\left\Vert G_{n}-F\right\Vert _{L,l,2}\rightarrow 0.
\]
\textbf{(B)}\quad For $F\in\mathcal{D}_\infty^d$, we denote \[\lambda(F)=\inf_{\vert\zeta\vert=1}\langle \sigma_F\zeta,\zeta\rangle\]  the lowest eigenvalue of the covariance matrix $\sigma_F$. We consider some $F$ and $F_n$ which verify $i), ii)$ in \textbf{(A)}. We also suppose that \[iii)\quad(DF_n)_{n\in\mathbb{N}}\text{ is a Cauchy sequence in } L^2(\Omega;\mathcal{H}),\] and for every $p\geq1$, 
\begin{eqnarray}iv)\quad \sup_{n}\mathbb{E}(\lambda(F_n))^{-p}\leq Q_p<\infty.\label{iv}\end{eqnarray} Then we have \[\mathbb{E}(\lambda(F))^{-p}\leq Q_p<\infty,\quad\forall p\geq1.\]
\textbf{(C)}\quad We suppose that we have $(F, \bar{F})$ and $(F_n, \bar{F}_n)$ which verify the hypotheses of \textbf{(A)}.  If  we also have \begin{eqnarray}v)\quad \sup_n\Vert DF_n-D\bar{F}_n\Vert_{L^2(\Omega;\mathcal{H})}\leq\bar{\varepsilon},\label{v}\end{eqnarray}then\[\Vert DF-D\bar{F}\Vert_{L^2(\Omega;\mathcal{H})}\leq\bar{\varepsilon}.\]
\end{lemma}

\begin{proof}
\textbf{Proof of (A)} For the sake of the simplicity of notations, we only prove for the one dimensional case.
The Hilbert space $\mathcal{H}_{l}=\mathcal{D}_{l,2}$ equipped with the scalar product
\begin{eqnarray*}
\left\langle U,V\right\rangle _{L,l,2}&:=&\sum_{q=1}^{l}\mathbb{E} \langle D^qU, D^qV\rangle_{\mathcal{H}^{\otimes q}}+\mathbb{E}( UV)\\
&+&\sum_{q=1}^{l-2}\mathbb{E} \langle D^qLU, D^qLV\rangle_{\mathcal{H}^{\otimes q}}+\mathbb{E}( LU\times LV)
\end{eqnarray*}
is the space of the functionals
which are $l-$times differentiable in $L^{2}$ sense.
By $ii)$, for $p\geq2$, $\left\Vert
F_{n}\right\Vert _{L,l,2}\leq \left\Vert F_{n}\right\Vert _{L,l,p}\leq K_{l,p}$.
Then, applying Banach Alaoglu theorem, there exists  $G\in\mathcal{H}_l$ and a subsequence (we still denote it by $n$), such that $F_n\rightarrow G$ weakly in the Hilbert space $\mathcal{H}_l$. This means that for every $Q\in\mathcal{H}_l$, $\langle F_n,Q\rangle_{L,l,2}\rightarrow\langle G,Q\rangle_{L,l,2}$. Therefore, by Mazur theorem, we can construct some convex combination
\[
G_{n}=\sum_{i=n}^{m_{n}}\gamma _{i}^{n}\times F_{i}\in \mathcal{S}
\]%
with $\gamma _{i}^{n}\geq 0,i=n,....,m_{n}$ and $%
\sum\limits_{i=n}^{m_{n}}\gamma _{i}^{n}=1$,
such that 
\[
\left\Vert G_{n}-G\right\Vert _{L,l,2}\rightarrow 0.
\]
In particular we have 
\[
\mathbb{E}\left\vert G_{n}-G\right\vert \leq \left\Vert G_{n}-G\right\Vert
_{L,l,2}\rightarrow 0.
\]%
Also, we notice that by i),
\[
\mathbb{E}\left\vert G_{n}-F\right\vert \leq \sum_{i=n}^{m_{n}}\gamma
_{i}^{n}\times \mathbb{E}\left\vert F_{i}-F\right\vert \rightarrow 0.
\]%
So we conclude that 
$F=G\in \mathcal{H}_{l}.$
Thus, we have 
\[
\mathbb{E}(\left\vert G_{n}-F\right\vert
_{l}^{2})+\mathbb{E}(\left\vert LG_{n}-LF\right\vert
_{l-2}^{2})\leq \left\Vert G_{n}-F\right\Vert _{L,l,2}^{2}\rightarrow 0.
\]%
By passing to a subsequence, we have $\left\vert G_{n}-F\right\vert _{l}+\left\vert LG_{n}-LF\right\vert _{l-2}\rightarrow 0
$ almost surely. 
Now, for every $\bar{p}\in[1,p)$, we denote $Y_{n}:=\left\vert G_{n}\right\vert
_{l}^{\bar{p}}+\left\vert LG_{n}\right\vert
_{l-2}^{\bar{p}}$ and $Y:=\left\vert F\right\vert _{l}^{\bar{p}}+\left\vert LF\right\vert _{l-2}^{\bar{p}}$. Then, $Y_n\rightarrow Y$ almost surely, and for any $\Tilde{q}\in[\bar{p},p]$,
\begin{eqnarray}
\mathbb{E}\vert G_n\vert_{l}^{\Tilde{q}}+\mathbb{E}\vert LG_n\vert_{l-2}^{\Tilde{q}}&\leq&\left\Vert G_{n}\right\Vert_{L,l,\Tilde{q}}^{\Tilde{q}}=
\left\Vert \sum_{i=n}^{m_{n}}\lambda _{i}^{n}\times F_{i}\right\Vert_{L,l,\Tilde{q}}^{\Tilde{q}}
\leq (\sum_{i=n}^{m_{n}}\lambda _{i}^{n}\times \left\Vert
F_{i}\right\Vert_{L,l,\Tilde{q}})^{\Tilde{q}}  \nonumber    \\
&\leq &(\sup_{i}\left\Vert F_{i}\right\Vert_{L,l,\Tilde{q}}\times\sum_{i=n}^{m_{n}}\lambda _{i}^{n})^{\Tilde{q}} =\sup_{i}\left\Vert F_{i}\right\Vert_{L,l,\Tilde{q}}^{\Tilde{q}}\leq K_{l,\Tilde{q}}^{\Tilde{q}}. \nonumber
\end{eqnarray}
So $(Y_n)_{n\in\mathbb{N}}$ is uniformly integrable, and we have
\[
\left\Vert F\right\Vert _{L,l,\bar{p}}^{\bar{p}}=\mathbb{E}(\left\vert F\right\vert
_{l}^{\bar{p}})+\mathbb{E}(\left\vert LF\right\vert
_{l-2}^{\bar{p}})=\mathbb{E}(Y)=\lim_{n\rightarrow\infty}\mathbb{E}(Y_{n})\leq K_{l,\bar{p}}^{\bar{p}}.
\]%

\bigskip\bigskip

\textbf{Proof of (B)} We consider for a moment some general $F,G\in\mathcal{D}_{\infty}^d$. Notice that \[\langle \sigma(F)\zeta,\zeta\rangle=\vert\langle DF,\zeta\rangle\vert_{\mathcal{H}}^2
,\]so \[\lambda(F)=\inf_{\vert\zeta\vert=1}\vert\langle DF,\zeta\rangle\vert_{\mathcal{H}}^2.\]Now we check that \begin{eqnarray}
\vert\sqrt{\lambda(F)}-\sqrt{\lambda(G)}\vert\leq\vert D(F-G)\vert_{\mathcal{H}}. \label{beautiful}
\end{eqnarray}Indeed, $\vert\langle DF,\zeta\rangle\vert_{\mathcal{H}}\leq \vert\langle DG,\zeta\rangle\vert_{\mathcal{H}}+\vert D(F-G)\vert_{\mathcal{H}}\vert\zeta\vert$, so that by taking the infimum, we get $\sqrt{\lambda(F)}\leq\sqrt{\lambda(G)}+\vert D(F-G)\vert_{\mathcal{H}}$. And in a similar way, we have the inverse inequality, so (\ref{beautiful}) is proved. We now come back to our framework. Recalling that $G_{n}=\sum\limits_{i=n}^{m_{n}}\gamma _{i}^{n}\times F_{i}$, we observe that \[\Vert DG_n-DF_n\Vert_{L^2(\Omega;\mathcal{H})}\leq \sum_{i=n}^{m_{n}}\gamma _{i}^{n}\Vert DF_{i}-DF_n\Vert_{L^2(\Omega;\mathcal{H})}\rightarrow0.\]Here we use the fact that $(DF_n)_{n\in\mathbb{N}}$ is a Cauchy sequence in  $L^2(\Omega;\mathcal{H})$.
Meanwhile, we know from \textbf{(A)} that \[\Vert DG_n-DF\Vert_{L^2(\Omega;\mathcal{H})}\rightarrow0.\]So we conclude that $\Vert DF-DF_n\Vert_{L^2(\Omega;\mathcal{H})}\rightarrow0$. Thus, by (\ref{beautiful}), $\mathbb{E}\vert\sqrt{\lambda(F)}-\sqrt{\lambda(F_n)}\vert\rightarrow0.$ This gives that there exists a subsequence (also denote by $n$) such that $\sqrt{\lambda(F_n)}$ converges to $\sqrt{\lambda(F)}$ almost surely, and consequently $\vert\lambda(F_n)\vert^{-p}$ converges to $\vert\lambda(F)\vert^{-p}$ almost surely.
Since we have (\ref{iv}), $(\vert\lambda(F_n)\vert^{-p})_{n\in\mathbb{N}}$ is uniformly integrable. It follows that \begin{eqnarray*}
\mathbb{E}(\vert \lambda(F)\vert^{-p})=\lim_{n\rightarrow\infty}\mathbb{E}(\vert \lambda(F_n)\vert^{-p})\leq Q_p. 
\end{eqnarray*}

\bigskip\bigskip

\textbf{Proof of (C)} Since the couples $(F, \bar{F})$ and $(F_n, \bar{F}_n)$ verify the hypotheses of \textbf{(A)},  we know by the results of \textbf{(A)} that we may find a convex combination such that  \[\overline{\lim}_{n\rightarrow\infty}\Vert\sum\limits_{i=n}^{m_{n}}\gamma _{i}^{n}(D F_{i},D \bar{F}_{i})-(DF,D\bar{F})\Vert_{L^2(\Omega;\mathcal{H})}=0.\]Then it follows by (\ref{v}) that \begin{eqnarray*}\Vert DF-D\bar{F}\Vert_{L^2(\Omega;\mathcal{H})}&\leq& \overline{\lim}_{n\rightarrow\infty}\Vert\sum\limits_{i=n}^{m_{n}}\gamma _{i}^{n}(D F_{i}-D\bar{F}_i)\Vert_{L^2(\Omega;\mathcal{H})}\\
&\leq& \overline{\lim}_{n\rightarrow\infty}\sum\limits_{i=n}^{m_{n}}\gamma _{i}^{n}\Vert D F_{i}-D\bar{F}_i\Vert_{L^2(\Omega;\mathcal{H})}\\
&\leq&\bar{\varepsilon}.\end{eqnarray*}
\end{proof}

\subsubsection{Main consequences}
We will use the abstract framework presented above for the IbP framework $(\mathcal{D}_\infty^d,D,L)$, with $D$ and $L$ defined in (\ref{3.10}). 
We recall the notations $\Vert F\Vert_{L,l,p}$ in (\ref{sobnorm}),  $\Sigma_p(F)$ in (\ref{sigma}) and $\sigma_F$ in (\ref{Mcov}). For any $\eta>0$, we take $\Upsilon_\eta(x):(0,\infty)\rightarrow\mathbb{R}$ to be a smooth function such that \[\mathbbm{1}_{[\frac{\eta}{2},\infty)}\leq\Upsilon_\eta\leq\mathbbm{1}_{[\eta,\infty)}.\]We remark that $\sigma_F$ is invertible on the set $\{\Upsilon_\eta(\det \sigma_F)>0\}.$ We first establish an integration by parts formula. 
\begin{lemma}
\textbf{(A)} Let $F=(F_{1},\cdots,F_{d})\in \mathcal{D}_\infty^{d}$. We suppose that the Malliavin covariance matrix $\sigma_F$ is invertible. We denote \[\Gamma_F=(\Gamma_F^{j,i})_{j,i=1,\cdots,d}=\sigma_F^{-1}.\] We also assume that $\det \sigma_F$ is almost surely invertible and $(\det \sigma_F)^{-1}\in\mathcal{D}_\infty$. Then for every $f\in C_p^1(\mathbb{R}^d)$ and $G\in\mathcal{D}_\infty$,
\[\mathbb{E}(\partial_if(F)G)=\mathbb{E}(f(F)H_i(F,G)),\]with\[H_i(F,G)=\sum_{j=1}^dG(\Gamma_F^{j,i}LF_j-\langle D\Gamma_F^{j,i},DF_j\rangle_{\mathcal{H}})-\sum_{j=1}^d\Gamma_F^{j,i}\langle DG,DF_j\rangle_{\mathcal{H}}.\]Moreover, iterating this relation, for every multi-index $\beta$ and every $f\in C_p^{\vert\beta\vert}(\mathbb{R}^d)$, we get\begin{eqnarray}
\mathbb{E}(\partial_\beta f(F)G)=\mathbb{E}(f(F)H_\beta(F,G)),\label{IbP}
\end{eqnarray}where $H_\beta(F,G)$ is obtained by iterations: for $\beta=(\beta_1,\cdots,\beta_m)\in\{1,\cdots,d\}^m$ and $\bar{\beta}=(\beta_1,\cdots\,\beta_{m-1})$, we define $H_{\beta}(F,G)=H_{\beta_m}(F,H_{\bar{\beta}}(f,G))$.

\textbf{(B)} Let $F=(F_{1},\cdots,F_{d})\in \mathcal{D}_\infty^{d}$. For any ${j,i=1,\cdots,d}$ we define \[\Gamma_{F,\eta}^{j,i}=(\sigma_F^{-1})^{j,i}\Upsilon_\eta(\det \sigma_F).\]  Then for every $f\in C_p^1(\mathbb{R}^d)$ and $G\in\mathcal{D}_\infty$,
\[\mathbb{E}(\partial_if(F)G\Upsilon_\eta(\det \sigma_F))=\mathbb{E}(f(F)H_{\eta,i}(F,G)),\]with\[H_{\eta,i}(F,G)=\sum_{j=1}^dG(\Gamma_F^{j,i}LF_j-\langle D\Gamma_{F,\eta}^{j,i},DF_j\rangle_{\mathcal{H}})-\sum_{j=1}^d\Gamma_{F,\eta}^{j,i}\langle DG,DF_j\rangle_{\mathcal{H}}.\]Moreover, iterating this relation, for every multi-index $\beta$ and every $f\in C_p^{\vert\beta\vert}(\mathbb{R}^d)$, we get\begin{eqnarray}\mathbb{E}(\partial_\beta f(F)G\Upsilon_\eta(\det \sigma_F))=\mathbb{E}(f(F)H_{\eta,\beta}(F,G)),\label{IbP2}
\end{eqnarray}where $H_{\eta,\beta}(F,G)$ is obtained by iterations: for $\beta=(\beta_1,\cdots,\beta_m)\in\{1,\cdots,d\}^m$ and $\bar{\beta}=(\beta_1,\cdots\,\beta_{m-1})$, we define $H_{\eta,\beta}(F,G)=H_{\eta,\beta_m}(F,H_{\eta,\bar{\beta}}(f,G))$.
\end{lemma}
\begin{remark}
In \textbf{(A)}, we assume the non-degeneracy condition for $F$, so we have the standard integration by parts formula. However in \textbf{(B)}, we do not assume any non-degeneracy condition of $F$, and we obtain a localized form of integration by parts formula.
\end{remark}
\begin{proof}
The proof of this lemma is standard, and we refer to $\cite{ref2}$.
\end{proof}
As a consequence of the integration by parts formula, we obtain the following proposition based on some estimations of the weights $\mathbb{E}\vert H_\beta(F,1)\vert$ and $\mathbb{E}\vert H_{\eta,\beta}(F,1)\vert$.
\begin{proposition}
Let $F=(F_{1},\cdots,F_{d})\in \mathcal{D}_\infty^{d}$. We fix $q\in\mathbb{N}$. 

\textbf{(A)} Suppose that  there exists a constant $C_q$ (dependent on $q$) such that $\Vert F\Vert_{L,q+2,8dq}+\Sigma_{4q}(F)\leq C_q$. Then for any multi-index  $\beta$ with $\vert\beta\vert= q$ and any function $f\in C_b^q(\mathbb{R}^d)$, \[(\textbf{B}_q)\quad\vert\mathbb{E}(\partial^\beta f(F))\vert\leq C_q\Vert f\Vert_\infty,\quad\forall \vert\beta\vert= q.\]

\textbf{(B)} Suppose that  there exists a constant $C_q^\prime$ (dependent on $q$) such that $\Vert F\Vert_{L,q+2,(4d+1)q}\leq C_q^\prime$. Then for any $\eta>0$, any multi-index  $\beta$ with $\vert\beta\vert= q$ and any function $f\in C_b^q(\mathbb{R}^d)$, \[(\textbf{B}_q^\prime)\quad\vert\mathbb{E}(\partial^\beta f(F)\Upsilon_\eta(\det \sigma_F))\vert\leq C_q^\prime\Vert f\Vert_\infty\times\frac{1}{\eta^{2q}},\quad\forall \vert\beta\vert= q.\]
\end{proposition}
\begin{remark}
In \textbf{(A)}, we assume the non-degeneracy condition for $F$, so we can control the weight $H_\beta$ in the standard integration by parts formula (\ref{IbP}). In \textbf{(B)}, we no longer suppose non-degeneracy condition for $F$, so we apply (\ref{IbP2}) and obtain a localized form of estimate.
\end{remark}
As an immediate application of \textbf{Proposition 3.4.1}, we have the regularity of the density.
\begin{corollary}
We fix $p\in\mathbb{N}$. Let $F=(F_{1},\cdots,F_{d})\in \mathcal{D}_\infty^{d}$. We assume that $\Vert F\Vert_{L,p+d+2,8d(p+d)}+\Sigma_{4(p+d)}(F)\leq \infty$. 
Then, the law of random variable $F$ is absolutely continuous with respect to the Lebesgue measure and has a density $p_F(x)$ which is $p-$times differentiable. And one has \begin{eqnarray}
p_F(x)=\mathbb{E}(\partial^\beta\prod_{j=1}^d\mathbbm{1}_{[0,\infty)}(F_j-x_j)),\quad  \beta=(1,\cdots,d),\label{density}
\end{eqnarray}with $x=(x_1,\cdots,x_d)\in\mathbb{R}^d$.
\end{corollary}
\begin{proof}
\textbf{Proposition 3.4.1} is proved in $\cite{ref2}$ and \textbf{Corollary 3.4.1} follows by standard regularization arguments.
\end{proof}

We consider the $d-$dimensional regularization kernels 
\[
\varphi (x)=\frac{1}{(2\pi )^{d/2}}e^{-\frac{\left\vert x\right\vert ^{2}}{2}%
},\quad \varphi _{\delta }(x)=\frac{1}{\delta ^{d}}\varphi (\frac{x}{\delta }),\quad0<\delta\leq1,
\]%
and we denote%
\[
f_{\delta }(x)=f\ast \varphi _{\delta }(x)=\int_{\mathbb{R}^{d}}f(y)\varphi _{\delta
}(x-y)dy.
\]%
Then we have the following regularization lemma.
\begin{lemma}\textbf{(A)}
$i)$ For a multi index $\beta $, we suppose that $F$ satisfies $(\textbf{B}_{2+\left\vert \beta \right\vert
}).$ Then for any function $f\in C_b^{2+\left\vert \beta \right\vert}(\mathbb{R}^d)$, 
\begin{equation}
\left\vert \mathbb{E}(\partial ^{\beta }f(F))-\mathbb{E}(\partial ^{\beta }f_{\delta
}(F))\right\vert \leq dC_{2+\left\vert \beta \right\vert }\left\Vert
f\right\Vert _{\infty }\times \delta ^{2}.  \label{1}
\end{equation}%
$ii)$ \textbf{(Romberg)} For a multi-index $\beta $, we suppose that $F$ satisfies $(\textbf{B}_{4+\left\vert \beta
\right\vert }).$ Then for any function $f\in C_b^{4+\left\vert \beta \right\vert}(\mathbb{R}^d)$, 
\begin{equation}
\left\vert \mathbb{E}(\partial ^{\beta }f(F))+\mathbb{E}(\partial ^{\beta }f_{\delta
}(F))-2\mathbb{E}(\partial ^{\beta }f_{\delta /\sqrt{2}}(F))\right\vert \leq
6d^{2}C_{4+\left\vert \beta
\right\vert}\left\Vert f\right\Vert _{\infty }\times \delta ^{4}.  \label{2}
\end{equation}
\textbf{(B)}
$iii)$ We suppose that $F$ satisfies $(\textbf{B}^\prime_{2}).$ We fix $\rho>0$ and we take some $G\in\mathcal{D}_\infty^d$ such that for any $p\in\mathbb{N}$, $\Vert F\Vert_{1,p}+\Vert G\Vert_{1,p}+\Sigma_{\rho}(G)<\infty$. For any $\varepsilon_0>0$, we denote $q=2/(1+\varepsilon_0)$. Then  there exists a constant $C$ depending on $p,q,\rho$ and $d$ such that for any $\eta>0$ and $\delta>0$, for any function $f\in C_b^{2}(\mathbb{R}^d)$,  
\begin{equation}
\left\vert \mathbb{E}(f(F))-\mathbb{E}(f_{\delta
}(F))\right\vert \leq C\left\Vert
f\right\Vert _{\infty }\times (\frac{\delta ^{2}}{\eta^4}+(\eta^{-1}{\Vert DF-DG\Vert_{L^2(\Omega;\mathcal{H})}})^q+\eta^\rho).  \label{1*}
\end{equation}%
$iv)$ \textbf{(Romberg)} We suppose that $F$ satisfies $(\textbf{B}^\prime_{4}).$ Under the same hypotheses as $iii)$, for any function $f\in C_b^{4}(\mathbb{R}^d)$, we have 
\begin{equation}
\left\vert \mathbb{E}(f(F))+\mathbb{E}(f_{\delta
}(F))-2\mathbb{E}(f_{\delta /\sqrt{2}}(F))\right\vert \leq C\left\Vert
f\right\Vert _{\infty }\times (\frac{\delta ^{4}}{\eta^8}+(\eta^{-1}{\Vert DF-DG\Vert_{L^2(\Omega;\mathcal{H})}})^q+\eta^\rho).  \label{2*}
\end{equation}%
\end{lemma}
\begin{remark}
We remark that in \textbf{(A)}, we assume the non-degeneracy condition for $F$, and we have the standard regularization lemma (\ref{1}). While in \textbf{(B)}, we do not assume the non-degeneracy condition for $F$, but we need to assume that we have another random variable $G$ which is non-degenerated (such that $DG$ is close to $DF$). Then we obtain a variant form of regularization lemma (\ref{1*}). Moreover, applying Romberg method, we have (\ref{2}) and (\ref{2*}). We also remark that the regularization lemma here is slightly different from the one in $\cite{ref2}$. The kernel  considered in $\cite{ref2}$ is the super kernel, but we are not able to simulate the super kernel. So in our paper, we consider the Gaussian kernel $\varphi_\delta$ which allows us to do the simulation.
\end{remark}
\begin{proof}
Through all this proof we use the notation $g=\partial
^{\beta }f.$ 

\textbf{Proof of (A) $i)$ :}
We denote%
\[
R_{q}(\delta ,x)=\frac{1}{q!}\sum_{\left\vert \alpha \right\vert
=q}\int_{0}^{1}d\lambda (1-\lambda )^{q}\int_{\mathbb{R}^{d}}dy\varphi _{\delta
}(y)y^{\alpha }\partial ^{\alpha }g(x+\lambda y) 
\]%
with $y^{\alpha }=\prod_{i=1}^{q}y_{\alpha _{i}}$ for $\alpha =(\alpha
_{1},...,\alpha _{q}).$ Notice that if $F$ satisfies $(\textbf{B}_{q})$ then (recall
that $\partial ^{\alpha }g=\partial ^{\alpha }\partial ^{\beta }f)$ 
\begin{equation}
\left\vert \mathbb{E}(R_{q}(\delta ,F))\right\vert \leq C_{q+\left\vert \beta
\right\vert }\left\Vert f\right\Vert _{\infty }\int_{\mathbb{R}^{d}}dy\varphi
_{\delta }(y)\left\vert y\right\vert ^{q}=C_{q+\left\vert \beta \right\vert
}\int_{\mathbb{R}^{d}}\varphi (y)\left\vert y\right\vert ^{q}dy\left\Vert
f\right\Vert _{\infty }\delta ^{q}.  \label{5}
\end{equation}

We use a development in Taylor series of order two in order to get%
\begin{eqnarray*}
\partial ^{\beta }f(x)-\partial ^{\beta }f_{\delta }(x)
&=&\int_{\mathbb{R}^{d}}dy\varphi _{\delta }(x-y)(\partial ^{\beta }f(x)-\partial
^{\beta }f(y)) \\
&=&\int_{\mathbb{R}^{d}}dy\varphi _{\delta }(x-y)(g(x)-g(y)) \\
&=&R_{2}(\delta ,x).
\end{eqnarray*}%
Here we use the fact that $\int_{\mathbb{R}^{d}}y_{j}\varphi _{\delta }(y)dy=0.$
This, together with (\ref{5}) yields (\ref{1}).

\bigskip

\textbf{Proof of (A) $ii)$ :}
Using a development in Taylor series of order $4$ 
\[
\partial ^{\beta }f(x)-\partial ^{\beta }f_{\delta }(x)=\frac{\delta ^{2}}{%
2}\nabla^2 g(x)+R_{4}(\delta ,x). 
\]%
Here we have used the fact that the third moments of the normal distribution
are null and $\int_{\mathbb{R}^{d}}y_{j}^{2}\varphi _{\delta }(y)dy=\delta ^{2}.$ We
fix $a\in (0,1)$ and we use the above equality  for $%
a\delta :$ 
\[
\frac{1}{a^{2}}\partial ^{\beta }f(x)-\frac{1}{a^{2}}\partial ^{\beta
}f_{a\delta }(x)=\frac{\delta ^{2}}{2}\nabla^2 g(x)+\frac{1}{a^{2}}%
R_{4}(a\delta ,x).
\]%
Subtracting the equality for $\delta $ and  for $%
a\delta$, we obtain 
\[
(\frac{1}{a^{2}}-1)\partial ^{\beta }f(x)-(\frac{1}{a^{2}}\partial ^{\beta
}f_{a\delta }(x)-\partial ^{\beta }f_{\delta }(x))=\frac{1}{a^{2}}%
R_{4}(a\delta ,x)-R_{4}(\delta ,x). 
\]%
Taking $a=1/\sqrt{2}$ we get 
\[
\partial ^{\beta }f(x)=2\partial ^{\beta }f_{\delta /\sqrt{2}%
}(x)-\partial ^{\beta }f_{\delta }(x)+2R_{4}(\delta /\sqrt{2}%
,x)-R_{4}(\delta ,x). 
\]%
And using (\ref{5}) we get (\ref{2}) (we have also used $\int_{\mathbb{R}^{d}}\varphi
(y)\left\vert y\right\vert ^{4}dy\leq 3d^{2})$.

\bigskip

\textbf{Proof of (B) $iii)$ :} We take $\vert\beta\vert=0$.
Notice that if $F$ satisfies $(\textbf{B}^\prime_{q})$, then 
\begin{equation}
\left\vert \mathbb{E}(R_{q}(\delta ,F)\Upsilon_\eta(\det \sigma_F))\right\vert \leq C^\prime_{q }\frac{\left\Vert f\right\Vert _{\infty }}{\eta^{2q}}\int_{\mathbb{R}^{d}}dy\varphi
_{\delta }(y)\left\vert y\right\vert ^{q}=C^\prime_{q }\int_{\mathbb{R}^{d}}\varphi (y)\left\vert y\right\vert ^{q}dy\left\Vert
f\right\Vert _{\infty }\frac{\delta ^{q}}{\eta^{2q}}.  \label{5*}
\end{equation}

We use a development in Taylor series of order two in order to get%
\begin{eqnarray*}
\mathbb{E}(f(F)\Upsilon_\eta(\det \sigma_F))-\mathbb{E}(f_{\delta }(F)\Upsilon_\eta(\det \sigma_F))
&=&\mathbb{E}(\int_{\mathbb{R}^{d}}dy\varphi _{\delta }(F-y)(f(F)-f(y))\Upsilon_\eta(\det \sigma_F))\\
&=&\mathbb{E}(R_{2}(\delta ,F)\Upsilon_\eta(\det \sigma_F)).
\end{eqnarray*}Here we use the fact that $\int_{\mathbb{R}^{d}}y_{j}\varphi _{\delta }(y)dy=0.$
Using (\ref{5*}) for $q=2$, we have \begin{equation*}
\left\vert \mathbb{E}(R_{2}(\delta ,F)\Upsilon_\eta(\det \sigma_F))\right\vert \leq C^\prime_{2 }\int_{\mathbb{R}^{d}}\varphi (y)\left\vert y\right\vert ^{2}dy\left\Vert
f\right\Vert _{\infty }\frac{\delta ^{2}}{\eta^{4}}. 
\end{equation*}
So \begin{eqnarray}
\vert\mathbb{E}(f(F)\Upsilon_\eta(\det \sigma_F))-\mathbb{E}(f_{\delta }(F)\Upsilon_\eta(\det \sigma_F))\vert
\leq C\left\Vert
f\right\Vert _{\infty }\frac{\delta ^{2}}{\eta^{4}}.\label{5**}
\end{eqnarray}
On the other hand, we make a small computational trick as follows which is originally from $\cite{ref52}$ p14. This trick allows us to obtain a better result. We denote 
\[R=\frac{\det \sigma
_F-\det \sigma_G}{\det \sigma_G}.\]This is well-defined since $G$ is non-degenerated. For an arbitrary $\eta$, we write
\begin{eqnarray}\mathbb{P}(\det \sigma_F<\eta)\leq \mathbb{P}(\det \sigma_F<\eta, \vert R\vert<\frac{1}{4})+\mathbb{P}(\vert R\vert\geq\frac{1}{4}).\label{BT}\end{eqnarray} When $\vert R\vert<\frac{1}{4}$, $\vert\det\sigma_F-\det\sigma_G\vert<\frac{1}{4}\det\sigma_G$. This implies that $\det\sigma_F>\frac{1}{2}\det\sigma_G$. Recalling that $G$ is non-degenerated and using Markov inequality, for every $\rho\in\mathbb{N}$, it follows that 
\begin{eqnarray}\mathbb{P}(\det \sigma_F<\eta, \vert R\vert<\frac{1}{4})\leq\mathbb{P}(\det\sigma_G<2\eta)\leq 2^\rho\eta^\rho\mathbb{E}\vert\det\sigma_G\vert^{-\rho}\leq C\eta^\rho.\label{BT1}\end{eqnarray}For any $\eta>0$,  $\rho\in\mathbb{N}$, with $q=2/(1+\varepsilon_0)$, we write \begin{eqnarray}\mathbb{P}(\vert R\vert\geq\frac{1}{4})&=&\mathbb{P}(\vert\det\sigma_F-\det\sigma_G\vert\geq\frac{1}{4}\det\sigma_G)\nonumber\\
&\leq&\mathbb{P}(\det\sigma_G\leq\eta)+\mathbb{P}(\vert\det\sigma_F-\det\sigma_G\vert>\frac{1}{4}\eta)\nonumber\\
&\leq&C(\eta^\rho+\eta^{-q}\mathbb{E}\vert \det\sigma_F-\det\sigma_G\vert^q)\nonumber\\
&\leq&C(\eta^\rho+(\eta^{-1}\Vert DF-DG\Vert_{L^2(\Omega;\mathcal{H})})^q),\label{BT2}\end{eqnarray}where in the last two steps, we have used the fact that $G$ is non-degenerated, and $\Vert F\Vert_{1,p}+\Vert G\Vert_{1,p}<\infty,\ \forall p\geq1$, and H$\ddot{o}$lder inequality with conjugates ${1+\varepsilon_0}$ and $\frac{1+\varepsilon_0}{\varepsilon_0}$. 
Putting together (\ref{BT}), (\ref{BT1}) and (\ref{BT2}),  we obtain
\begin{eqnarray}
\mathbb{P}(\det \sigma_F<\eta)&\leq&C(\eta^\rho+(\eta^{-1}\Vert DF-DG\Vert_{L^2(\Omega;\mathcal{H})})^q).\label{BT0}
\end{eqnarray}
Then we have \begin{eqnarray}
\vert\mathbb{E}((1-\Upsilon_\eta(\det \sigma_F))f(F))\vert\leq\Vert f\Vert_\infty\mathbb{P}(\det \sigma_F<\eta)\leq C\Vert f\Vert_\infty(\eta^\rho+(\eta^{-1}\Vert DF-DG\Vert_{L^2(\Omega;\mathcal{H})})^q).\label{5***}
\end{eqnarray}Similarly, we also have \begin{eqnarray}
\vert\mathbb{E}((1-\Upsilon_\eta(\det \sigma_F))f_\delta(F))\vert\leq C\Vert f\Vert_\infty(\eta^\rho+(\eta^{-1}\Vert DF-DG\Vert_{L^2(\Omega;\mathcal{H})})^q).\label{5****}
\end{eqnarray}We conclude by combining (\ref{5**}), (\ref{5***}) and (\ref{5****}).

\bigskip

\textbf{Proof of (B) $iv)$ :} The proof is analogous to the proof of $ii)$.
Using a development in Taylor series of order $4$ 
\[
f(x)-f_{\delta }(x)=\frac{\delta ^{2}}{%
2}\nabla^2 f(x)+R_{4}(\delta ,x). 
\]%
We use the above equality for $\delta $ and  $\frac{\delta}{\sqrt{2}}$, 
then by subtracting them, we get 
\[
f(x)=2f_{\delta /\sqrt{2}%
}(x)-f_{\delta }(x)+2R_{4}(\delta /\sqrt{2}%
,x)-R_{4}(\delta ,x). 
\]%
So by (\ref{5*}), \begin{eqnarray}
&&\left\vert \mathbb{E}(f(F)\Upsilon_\eta(\det \sigma_F))+\mathbb{E}(f_{\delta
}(F)\Upsilon_\eta(\det \sigma_F))-2\mathbb{E}(f_{\delta /\sqrt{2}}(F)\Upsilon_\eta(\det \sigma_F))\right\vert\nonumber\\
&&\leq 2\mathbb{E}(R_{4}(\delta /\sqrt{2},x)\Upsilon_\eta(\det \sigma_F))-\mathbb{E}(R_{4}(\delta ,x)\Upsilon_\eta(\det \sigma_F))\nonumber\\
&&\leq C\Vert f\Vert_\infty\frac{\delta^4}{\eta^8}.
\end{eqnarray}We conclude together with (\ref{5***}) and (\ref{5****}).

\end{proof}The regularization lemma (\textbf{Lemma 3.5}) implies the following result concerning the approximation of the density function. 
\begin{corollary}
$i)$ Suppose that $F$ satisfies $(\textbf{B}_{2+d}).$ Then, for every $x,$ 
\begin{equation}
\left\vert p_{F}(x)-\mathbb{E}(\varphi _{\delta }(F-x))\right\vert \leq
dC_{2+d}\times \delta ^{2}.  \label{3}
\end{equation}%
$ii)$ \textbf{(Romberg)} Suppose that $F$ satisfies $(\textbf{B}_{4+d}).$ Then 
\begin{equation}
\left\vert p_{F}(x)+\mathbb{E}(\varphi _{\delta }(F-x))-2\mathbb{E}(\varphi _{\delta /\sqrt{2}%
}(F-x))\right\vert \leq 6d^{2}C_{4+d}\times \delta ^{4}.  \label{4}
\end{equation}
\end{corollary}
\begin{proof} We take a multi-index $\beta=(1,\cdots,d)$ and 
\begin{equation}
f(y)=\prod_{j=1}^{d}H(y_{j}),  \label{6}
\end{equation}where $H(y)=\mathbbm{1}_{[0,\infty)}(y)$ is the Heaviside function.
So by (\ref{density}),
\[
p_{F}(x)=\mathbb{E}(\partial ^{\beta }f(F-x)). 
\]%
Notice that 
\[
\partial ^{\beta }f_{\delta }(F-x)=\prod_{j=1}^{d}H_{\delta }^{\prime
}(F_{j}-x_{j}))=\varphi _{\delta }(F-x),
\]%
so that (\ref{1}) gives 
\begin{eqnarray*}
\left\vert p_{F}(x)-\mathbb{E}(\varphi _{\delta }(F-x)\right\vert &=&\left\vert
\mathbb{E}(\partial ^{\beta }f(F-x))-\mathbb{E}(\partial ^{\beta }f_{\delta
}(F-x))\right\vert \\
&\leq &dC_{2+d}\times \delta ^{2}.
\end{eqnarray*}%
In a similar way (\ref{2}) gives (\ref{4}).
\end{proof}

\bigskip\bigskip

In the following, we define the distances between random variables $F,G:\Omega
\rightarrow \mathbb{R}^d$:%
\[
d_{r}(F,G)=\sup \{\left\vert \mathbb{E}(f(F))-\mathbb{E}(f(G))\right\vert :\sum_{ \vert\beta\vert= r}\left\Vert  \partial^\beta f\right\Vert _{\infty }\leq 1\} 
\]%
For $r=1$, this is the  Wasserstein distance $W_1$, while for $r=0$, this is
the total variation distance $d_{TV}$.

Using the Malliavin integration by parts formula (\textbf{Lemma 3.4}), one proves in $\cite{ref2}$ (lemma 3.9)  the following results. 

\begin{lemma}
We fix some index $l$, some $r\in \mathbb{N}$ and some $\varepsilon >0.$ We define $p_1=2( r(\frac{1}{\varepsilon}-1)+2)$, $p_2=\max\{4(l+d),2(\frac{r+l}{\varepsilon}-r+2)\}$, $q_1\geq r(\frac{1}{\varepsilon}-1)+4$, $q_2\geq\max\{l+d+2, \frac{r+l}{\varepsilon}-r+4\}$. Let $F,G\in \mathcal{D}_\infty^d$.  One may find  $p\in \mathbb{N}$, $C\in \mathbb{R}_{+}$
 (depending on $r,l$ and $\varepsilon )$ such that 
\begin{equation}
i)\quad d_{TV}(F,G)\leq C(1+\Sigma _{p_1}(F)+\Sigma_{p_1}(G)+\left\Vert F\right\Vert
_{L,q_1,p}+\left\Vert G\right\Vert _{L,q_1,p})\times
d_{r}(F,G)^{1-\varepsilon },
\label{3.12}
\end{equation}%
and
\begin{equation}
ii)\quad\Vert p_F- p_G\Vert_{l,\infty}\leq C(1+\Sigma _{p_2}(F)+\Sigma_{p_2}(G)+\left\Vert F\right\Vert
_{L,q_2,p}+\left\Vert G\right\Vert _{L,q_2,p})\times
d_{r}(F,G)^{1-\varepsilon },  \label{3.12*}
\end{equation}%
where $p_F(x)$ and $p_G(x)$ denote the density functions of $F$ and $G$ respectively.
\end{lemma}

\begin{remark}
We explain  about the significance of this lemma. If we have already obtained an estimate of a "smooth" distance $d_{r}$ between two random vectors $F$ and $G$ but we would like to control the total variation distance between them, then we employ some integration by parts techniques which are
developed in [BCP] and conclude the following. If both $F$ and $G$
are "smooth" in the sense that $\left\Vert F\right\Vert _{L,q,p}+\left\Vert
G\right\Vert _{L,q,p}<\infty $ for sufficiently large $q,p;$ and both $F$ and $G$ are non-degenerated in the sense that $%
\Sigma _{p}(F)+\Sigma _{p}(G)<\infty $, with $p$ large enough, then (\ref{3.12}%
) asserts that one may control $d_{TV}$ by $d_{r},$ and the control is quasi
optimal: we loose just a power $\varepsilon >0$ which we may  take as small
as we want. And (\ref{3.12*}) says that we may also control the distance between the derivatives of density functions by $d_r$.
\end{remark} 

If we only assume the non-degeneracy condition on $F$ but no non-degeneracy condition for $G$, then we have a variant of the previous lemma (see $\cite{ref2}$ proposition 3.11 and remark 3.14).
\begin{proposition}
We fix some $r\in \mathbb{N}$ and some $\varepsilon >0.$ We define $p_1=2( \frac{8r}{\varepsilon}+2)$,  $q_1\geq \frac{8r}{\varepsilon}+4$. Let $F,G\in \mathcal{D}_\infty^d.$  One may find  $p\in \mathbb{N}$, $C\in \mathbb{R}_{+}$
 (depending on $r$ and $\varepsilon )$ such that 
\begin{equation}
 d_{TV}(F,G)\leq C(1+\Sigma _{p_1}(F)+\left\Vert F\right\Vert
_{L,q_1,p}+\left\Vert G\right\Vert _{L,q_1,p})\times
(d_{r}(F,G)+\Vert DF-DG\Vert_{L^2(\Omega;\mathcal{H})}^2)^{1-\varepsilon }.
\label{3.12.1}
\end{equation}%
\end{proposition}
\begin{remark}
The result in \textbf{Proposition 3.6.1} is better than proposition 3.11 and remark 3.14 in $\cite{ref2}$. We get $\Vert DF-DG\Vert_{L^2(\Omega;\mathcal{H})}^2$ instead of $\Vert DF-DG\Vert_{L^2(\Omega;\mathcal{H})}$. This is because rather than the estimate $(3.29)$ with $p^\prime=1$ in $\cite{ref2}$, we use a sharper estimate ($\ref{BT0}$) with $q=\frac{2}{1+\varepsilon_0}$ and $\varepsilon_0=\frac{\varepsilon}{2-2\varepsilon}$. The idea of ($\ref{BT0}$) comes from the paper $\cite{ref52}$ p14. We benefit a lot from this improvement in the paper. It guarantees that we are able to keep the speed of convergence $1-\varepsilon$ (instead of $\frac{1}{2}-\varepsilon$) in the final results \textbf{Theorem 2.1$\sim$2.4}. 
\end{remark}

\bigskip

\subsection{Malliavin calculus for the jump equations}
In this section, we present the integration by parts framework which will be used when we deal with the jump equations  (\ref{xMZ}), (\ref{xPMZ}) and (\ref{x}). There are several approaches given in $\cite{ref5}$, $\cite{ref9}$, $\cite{ref10}$, $\cite{ref11}$, $\cite{ref15}$, $\cite{ref17}$ and $\cite{ref18}$ for example. Here we give a framework analogous to $\cite{ref3}$.

To begin we define a regularization function.
\begin{eqnarray}
	a(y) &=&1-\frac{1}{1-(4y-1)^{2}}\quad for\quad y\in \lbrack \tfrac{1}{4}, \tfrac{1}{2}),  \label{3.14} \\
\psi (y) &=&\mathbbm{1}_{\{\left\vert y\right\vert \leq \frac{1}{4}\}}+\mathbbm{1}_{\{\frac{1}{4}%
< \left\vert y\right\vert \leq \frac{1}{2}\}}e^{a(\left\vert y\right\vert
)}.  \label{3.14*}
\end{eqnarray}%
We notice that $\psi \in C_{b}^{\infty }(\mathbb{R})$ and that its support is included in $[-\frac{1}{2},\frac{1}{2}]$. 
We denote
\begin{equation}
\Psi_k(y)=\psi(\vert y\vert-(k-\tfrac{1}{2})),\ \forall k\in\mathbb{N}.  \label{3.19*}
\end{equation}%
Then for any $l\in\mathbb{N}$, there exists a constant $C_l$ such that \begin{eqnarray}
\sup_{k\in\mathbb{N}}\Vert\Psi_k\Vert_{l,\infty}\leq C_l<\infty.\label{psiborn}
\end{eqnarray}

We focus on ${x}^{\mathcal{P},M}_{t}$ and ${x}^{M}_{t}$ (solutions of  the equations (\ref{xPMZ}) and (\ref{xMZ})) which are  functions of random variables $T^k_i,W^k_i,$ $Z^k_i$, $\Delta$ and $X_0$ (see Section 2.7). 
Now we introduce the space of simple functionals $\mathcal{S}.$ 
We take $\mathcal{G}=\sigma(T^k_i,W_i^k,X_0: k,i\in \mathbb{N})$\ to be the $\sigma-$algebra associated to the noises which will not be involved in our calculus. In the following, we will do the calculus based on $Z^k_i=(Z^k_{i,1},\cdots,Z^k_{i,d}),k,i\in\mathbb{N}$ and $\Delta=(\Delta_1,\cdots,\Delta_d)$. We denote by $C_{\mathcal{G},p}$ the space of the functions $f:\Omega \times \mathbb{R}^{m\times m^{\prime}\times d+d}\rightarrow\mathbb{R}$ such that for each $\omega ,$ the function $(z_{1,1}^1,...,z_{m,d}^{m^{\prime}},\delta_1,\cdots,\delta_d)\mapsto f(\omega ,z_{1,1}^1,...,z_{m,d}^{m^{\prime}},\delta_1,\cdots,\delta_d)$ belongs to $C_{p}^{\infty }(\mathbb{R}^{m\times m^{\prime}\times d+d})$ (the space of smooth functions which, together with all the derivatives, have polynomial growth), and for each $(z_{1,1}^1,...,z_{m,d}^{m^{\prime}},\delta_1,\cdots,\delta_d)$, the function $\omega \mapsto f(\omega ,z_{1,1}^1,...,z_{m,d}^{m^{\prime}},\delta_1,\cdots,\delta_d)$ is $%
\mathcal{G}$-measurable. And we consider the weights  \[\xi_{i}^k=\Psi_k({Z}^k_i).\] Then we define the space of simple functionals 
\begin{equation*}
\mathcal{S}=\{F=f (\omega ,({Z}_{i}^{k})_{\substack{1\leq k\leq m^{\prime}\\1\leq i\leq m}},\Delta) :f\in C_{%
\mathcal{G},p}, m,m^{\prime}\in \mathbb{N}\}.
\end{equation*}%
\begin{remark}
The simple functional $F$ is actually a function of $(T^k_i)_{k,i\in\mathbb{N}},(W^k_i)_{k,i\in\mathbb{N}},$ $(Z^k_i)_{k,i\in\mathbb{N}}$, $\Delta$ and $X_0$. By taking $m=J^k_t$ and $m^{\prime}=M$, we notice that for any $0<t\leq T$,  ${x}^{\mathcal{P},M}_{t}$ and ${x}^{M}_{t}$ (solutions of  the equations (\ref{xPMZ}) and (\ref{xMZ})) both belong to $\mathcal{S}^d$.
\end{remark}
On the space $\mathcal{S}$ we define the derivative operator $DF=(D^ZF,D^{\Delta}F)$, where
\begin{eqnarray}
D^Z_{(\bar{k},\bar{i},\bar{j})}F&=&\xi _{\bar{i}}^{\bar{k}}\frac{\partial f}{\partial z^{\bar{k}}_{\bar{i},\bar{j}}}(\omega ,({Z}_{i}^{k})_{\substack{1\leq k\leq m^{\prime}\\1\leq i\leq m}},\Delta),\quad \bar{k},\bar{i}\in\mathbb{N},\bar{j}\in\{1,\cdots,d\}, \label{D}\\
D^{\Delta}_{\Tilde{j}}F&=&\frac{\partial f}{\partial \delta_{\Tilde{j}}}(\omega ,({Z}_{i}^{k})_{\substack{1\leq k\leq m^{\prime}\\1\leq i\leq m}},\Delta),\quad \Tilde{j}\in\{1,\cdots,d\}.\nonumber
\end{eqnarray}%
We regard $D^ZF$ as an element of the Hilbert space $l_{2}$\ (the space of
the sequences $u=(u_{k,i,j})_{k,i\in \mathbb{N},{j}\in\{1,\cdots,d\}}$ with $\left\vert u\right\vert
_{l_{2}}^{2}:=\sum_{k=1}^{\infty }\sum_{i=1}^{\infty }\sum_{j=1}^{d}\vert u_{k,i,j}\vert^{2}<\infty$) and $DF$ as an element of $l_{2}\times\mathbb{R}^d$, so we have%
\begin{eqnarray}
\left\langle DF,DG\right\rangle_{l_2\times\mathbb{R}^d} =\sum_{j=1}^dD_j^{\Delta}F\times D_j^{\Delta}G+\sum_{k=1}^{\infty }\sum_{i=1}^{\infty }\sum_{j=1}^{d }D^Z_{(k,i,j)}F\times D^Z_{(k,i,j)}G. \label{Cov}
\end{eqnarray}%
We also denote $D^1F=DF$, and we define the derivatives of order $q\in \mathbb{N}$ recursively: 
$D^{q}F:=DD^{q-1}F.$ And we denote $D^{Z,q}$ (respectively $D^{\Delta,q}$) as the derivative $D^Z$ (respectively $D^\Delta$) of order $q$.

We recall the function $h$ given in \textbf{Hypothesis 2.4 $b)$}.
We also define the Ornstein-Uhlenbeck operator $LF=L^ZF+L^{\Delta}F$ with
\begin{eqnarray}
L^ZF&=&-\sum_{k=1}^{m^{\prime}}\sum_{i=1}^{m }\sum_{j=1}^{d }(\partial_{z^k_{i,j}}(\xi^k_{i}D^Z_{(k,i,j)}F)+ D^Z_{(k,i,j)}F\times D^Z_{(k,i,j)}\ln [{h(Z^k_i)}]), \label{L}\\
L^{\Delta}F&=&\sum_{j=1}^dD^{\Delta}_jF\times\Delta_j-\sum_{j=1}^dD^{\Delta}_jD^{\Delta}_jF.\nonumber
\end{eqnarray}

One can check that the triplet $(\mathcal{S},D,L)$ is consistent with the IbP framework given in Section 3.1. In particular the duality formula (\ref{0.01}) holds true. We refer to $\cite{ref4}$(Appendix 5.3). We say that $F$ is a "Malliavin smooth functional" if $F\in\mathcal{D}_\infty$ (with the definition given in (\ref{Dinf})).

We will use the IbP framework defined here for $x_t$, ${x}_{t}^{M}$ and ${x}_{t}^{\mathcal{P},M}$ (solutions of equations (\ref{x}),(\ref{xM}) and (\ref{xPM})). We recall that they are obtained in Section 2.7 by optimal coupling in $W_{2+\varepsilon_{\ast}}$ distance between $X^{\mathcal{P},M}_{\tau(t)-}$  and $X_{t-}$. 
Here we give two lemmas, concerning the Sobolev norms and the covariance matrices of $x_t$, ${x}_{t}^{M}$ and ${x}_{t}^{\mathcal{P},M}$. 

\begin{lemma}
Assuming \textbf{Hypothesis 2.1} and \textbf{Hypothesis 2.4} $b)$,  for all $p\geq1, l\geq 0$, there exists a constant $C_{l,p}$ depending on $l,p,d$ and $T$, such that for any $0<t\leq T$,  \[i)\quad\sup\limits_{{\mathcal{P}}}\sup\limits_{{M}}(\Vert {x}_{t}^{\mathcal{P},M}\Vert_{L,l,p}+\Vert {x}_{t}^{M}\Vert_{L,l,p})\leq C_{l,p}.\] Moreover, ${x}_{t}$ belongs to $\mathcal{D}_{\infty}^d$ and \[ii)\quad\Vert {x}_{t}\Vert_{L,l,p}\leq C_{l,p}.\]
\end{lemma}

\begin{lemma}
Assume that \textbf{Hypothesis 2.1, 2.2, 2.3} and \textbf{2.4} hold true. 
Then for every $p\geq1$, $0<t\leq T$ such that $t>\frac{2dp}{\theta}$ (with $\theta$ defined in (\ref{undergamma})), we have (recalling by (\ref{Mcov}) that $\sigma_F$ denotes the covariance matrix of $F$)
\begin{eqnarray}
i)\quad
\sup\limits_{M}\mathbb{E}(1/ \det\sigma_{{x}_{t}^{M}})^p\leq C_p, \label{cM8}\\
ii)\quad\mathbb{E}(1/ \det\sigma_{{x}_{t}})^p\leq C_p, \label{cov}
\end{eqnarray}
with $C_p$ a constant depending on $p,d,T$.
\end{lemma}
\begin{remark}
In the case $\theta=\infty$, the results in \textbf{Lemma 3.8}  hold for every $0<t\leq T$.
\end{remark}
\begin{remark}
We are not able to prove that $\sup\limits_{\mathcal{P}}\sup\limits_{M}\mathbb{E}(1/ \det\sigma_{{x}_{t}^{\mathcal{P},M}})^p\leq C_p$, since the tangent flow of the Euler scheme is not invertible (see (\ref{inverse}), the inverse tangent flow plays an important role in our proof). This is why we need \textbf{Proposition 3.6.1} instead of \textbf{Lemma 3.6}. Fortunately, thanks to (\ref{BT0}) (inspired from $\cite{ref52}$ p14), we are able to keep the same speed of convergence $1-\varepsilon$ as if the tangent flow of the Euler scheme were invertible. 
\end{remark}

The proofs of these two lemmas are rather technical and are postponed to Section 4.1 and 4.2.

Before we end this section, we establish an auxiliary result. We recall by (\ref{epsM}) that
$\varepsilon_M=\int_{\{\vert z\vert>M\}}\vert\bar{c}(z)\vert^2\mu(dz)+\vert\int_{\{\vert z\vert>M\}}\bar{c}(z)\mu(dz)\vert^2).$
\begin{lemma}We assume that \textbf{Hypothesis 2.1}  and \textbf{Hypothesis 2.4} $b)$ hold true. Then for any $\varepsilon_{\ast}>0$, there exists a constant $C$ dependent on $T,d,\varepsilon_{\ast}$ such that for every $\vert\mathcal{P}\vert\leq1$,  every $M$ with $\varepsilon_M\leq 1$ and $\vert\bar{c}(z)\vert^2\mathbbm{1}_{\{\vert z\vert>M\}}\leq 1$,
\[i)\quad\Vert D{x}_{t}^{\mathcal{P},M}-D{x}_{t}^{M}\Vert_{L^2(\Omega;l_2\times\mathbb{R}^d)}\leq C(\varepsilon_M+\vert\mathcal{P}\vert)^{\frac{1}{2+\varepsilon_{\ast}}},\]
\[ii)\quad\Vert D{x}_{t}^{M}-D{x}_{t}\Vert_{L^2(\Omega;l_2\times\mathbb{R}^d)}\leq C(\varepsilon_M)^{\frac{1}{2+\varepsilon_{\ast}}},\]
\[iii)\quad\Vert D{x}_{t}^{\mathcal{P},M}-D{x}_{t}\Vert_{L^2(\Omega;l_2\times\mathbb{R}^d)}\leq C(\varepsilon_M+\vert\mathcal{P}\vert)^{\frac{1}{2+\varepsilon_{\ast}}}.\]
\end{lemma}The proof is also technical and we put it in the Appendix.

\subsection{Proofs of  \textbf{Theorem 2.1$\sim$2.4}}
Before the proofs of \textbf{Theorem 2.1$\sim$2.4}, we first give the following lemma. We recall $X_{t}^{\mathcal{P},M}$ in (\ref{truncation}) and $X_{t}$ in (\ref{1.1}).
\begin{lemma}
Assume that the \textbf{Hypothesis 2.1} holds true.
Then there exists a  constant $C$ dependent on $T$ such that for every partition $\mathcal{P}$ and $M\in\mathbb{N}$ we have \[W_1(X_{t}^{\mathcal{P},M}, X_{t})\leq C(\vert\mathcal{P}\vert+\sqrt{\varepsilon_M}).\]
\end{lemma}
\begin{proof}
We make a coupling argument similar to Section 2.7. We will do optimal coupling between $X^{\mathcal{P},M}_{\tau(t)-}$ and $X_{t-}$ in $W_{1}$ distance. This is the same strategy as the optimal coupling between $X^{\mathcal{P},M}_{\tau(t)-}$ and $X_{t-}$ in $W_{2+\varepsilon_{\ast}}$ distance in Section 2.7.
We  take $\widetilde{\Pi}^{\mathcal{P},M}_t(dv_5,dv_6)$ to be the optimal $W_{1}-$coupling of $\rho^{\mathcal{P},M}_{\tau(t)-}(dv_5)$ and $\rho_{t-}(dv_6)$, that is
\[W_{1}(\rho^{\mathcal{P},M}_{\tau(t)-},\rho_{t-})=\int_{\mathbb{R}^d\times\mathbb{R}^d}\vert v_5-v_6\vert\widetilde{\Pi}^{\mathcal{P},M}_t(dv_5,dv_6).\]Then we  construct $(\eta^5_t(w),\eta^6_t(w))$ which represents  $\widetilde{\Pi}^{\mathcal{P},M}_t$ in the sense of \textbf{Lemma 2.5}. So we have \[\int_0^1\phi(\eta^5_t(w),\eta^6_t(w)) dw=\int_{\mathbb{R}^d\times\mathbb{R}^d}\phi(v_5,v_6)\widetilde{\Pi}^{\mathcal{P},M}_t(dv_5,dv_6).\]We consider the equations (with  $\mathcal{N}(dw,dz,dr)$ the Poisson point measure on the state space $[0,1]\times\mathbb{R}^d$ with intensity measure $dw\mu(dz)dr$ defined in Section 2.7):
\begin{eqnarray}
\widetilde{x}_{t}&=&X_0+\int_{0}^{t}b(r,\widetilde{x}_{r},\rho_{r})dr  +\int_{0}^{t}\int_{[0,1]\times \mathbb{R}^d}{c}(r,\eta^6_r(w),z,\widetilde{x}_{r-},\rho_{r-})\mathcal{N}(dw,dz,dr),  \label{xtilde}
\end{eqnarray}
\begin{eqnarray}
\widetilde{x}^{\mathcal{P},M}_{t}&=&X_0+a^M_T\Delta+\int_{0}^{t}b(\tau(r),\widetilde{x}^{\mathcal{P},M}_{\tau(r)},\rho^{\mathcal{P},M}_{\tau(r)})dr  \nonumber\\
&+&\int_{0}^{t}\int_{[0,1]\times \mathbb{R}^d}{c}_M(\tau(r),\eta^5_r(w),z,\widetilde{x}^{\mathcal{P},M}_{\tau(r)-},\rho^{\mathcal{P},M}_{\tau(r)-})\mathcal{N}(dw,dz,dr),  \label{xPMtilde}
\end{eqnarray}with $\rho^{\mathcal{P},M}_{t}$ the law of $X^{\mathcal{P},M}_{t}$ (see (\ref{truncation})) and $\rho_{t}$ the law of $X_{t}$ (see (\ref{1.1})).
One can check that $\widetilde{x}_{t}$ and  $\widetilde{x}^{\mathcal{P},M}_{t}$ have the same law as ${x}_{t}$  and  ${X}^{\mathcal{P},M}_{t}$ respectively. We remark that  $\widetilde{x}_{t}$, and  $\widetilde{x}^{\mathcal{P},M}_{t}$ are different from  ${x}_{t}$,  and  ${x}^{\mathcal{P},M}_{t}$ (see (\ref{x}) and (\ref{xPM}))  since we take different couplings and $\eta^1_r(w)\neq\eta^5_r(w)$, $\eta^2_r(w)\neq\eta^6_r(w)$. Then we have \[W_1(X_{t}^{\mathcal{P},M}, X_{t})=W_1(\widetilde{x}_{t}^{\mathcal{P},M}, \widetilde{x}_{t})\leq\mathbb{E}\vert \widetilde{x}_{t}^{\mathcal{P},M}-\widetilde{x}_{t}\vert\leq C(\vert\mathcal{P}\vert+\sqrt{\varepsilon_M}),\]
where the last inequality is obtained in a standard way (see the proof of \textbf{Lemma 2.6}).
\end{proof}

\bigskip

\textbf{Proofs of Theorem 2.1 and Theorem 2.2:}

\begin{proof}
We first prove (\ref{002*}).  We recall that by the discussion in Section 2.7, $x^{\mathcal{P},M}_{t}$ has the same law as $X^{\mathcal{P},M}_{t}$ and $x_{t}$ has the same law as $X_{t}$. 
Thanks to \textbf{Lemma 3.7} and \textbf{Lemma 3.8},  using \textbf{Proposition 3.6.1},  for  any partition $\mathcal{P}$ of the interval $[0,T]$ with $\vert\mathcal{P}\vert\leq1$,  every $M\in\mathbb{N}$ with $\varepsilon_M\leq 1$ and $\vert\bar{c}(z)\vert^2\mathbbm{1}_{\{\vert z\vert>M\}}\leq 1$, for $\varepsilon>0$, when $t>\frac{8d}{\theta}(\frac{4}{\varepsilon}+1)$ (with $\theta$ defined in (\ref{undergamma})), \begin{eqnarray*}
d_{TV}(X^{\mathcal{P},M}_{t},X_{t})&=&d_{TV}(x^{\mathcal{P},M}_{t},x_{t})\\
&\leq& C[W_1(x^{\mathcal{P},M}_{t},x_{t})+\Vert D{x}_{t}^{\mathcal{P},M}-D{x}_{t}\Vert_{L^2(\Omega;l_2\times\mathbb{R}^d)}^2]^{1-\varepsilon}.
\end{eqnarray*}For any $\bar{\varepsilon}>0$, we take $\varepsilon,\varepsilon_\ast>0$ such that $\varepsilon_\ast=\frac{\bar{\varepsilon}}{1-\bar{\varepsilon}}$ and $\varepsilon=\frac{\bar{\varepsilon}}{2}$. So $\frac{2}{2+\varepsilon_\ast}(1-\varepsilon)=1-\bar{\varepsilon}.$  Then by \textbf{Lemma 3.9} and \textbf{Lemma 3.10}, when $t>\frac{8d}{\theta}(\frac{8}{\bar{\varepsilon}}+1)$, we have \begin{eqnarray*}
d_{TV}(X^{\mathcal{P},M}_{t},X_{t})
&=& C[W_{1}(X^{\mathcal{P},M}_{t},X_{t})+\Vert D{x}_{t}^{\mathcal{P},M}-D{x}_{t}\Vert_{L^2(\Omega;l_2\times\mathbb{R}^d)}^2]^{1-\varepsilon}\\
&\leq&C[\vert\mathcal{P}\vert+\sqrt{\varepsilon_M}+(\varepsilon_M+\vert\mathcal{P}\vert)^{\frac{2}{2+\varepsilon_\ast}}]^{1-\varepsilon}\\
&\leq&C[\sqrt{\varepsilon_M}+\vert\mathcal{P}\vert]^{1-\bar{\varepsilon}}\rightarrow0,
\end{eqnarray*}with $C$  a constant depending on $\bar{\varepsilon},d$ and $T$. So (\ref{002*}) is proved.

On the other hand, by \textbf{Lemma 3.8} and \textbf{Corollary 3.4.1}, when $t>\frac{8d(l+d)}{\theta}$, the law of $X_{t}$ has a $l-$times differentiable density $p_{t}$ and the density $p_{t}$ is a function solution of the equation (\ref{1.4}). So (\ref{001}) is proved. We notice that $(\mathcal{S}, D^\Delta, L^\Delta)$ is also an IbP framework. If we only make Malliavin integration by parts on the Gaussian random variable $\Delta$, then standard arguments give that the law of $X_{t}^{\mathcal{P},M}$ has a smooth density $p_{t}^{\mathcal{P},M}$.

\bigskip\bigskip

Now only (\ref{002}) is left to be proved. The proof is analogous to the proof of (\ref{002*}).  The main strategy  is as follows (this is similar to Section 2.7 and Section 3.2). We  define an intermediate equation $\bar{X}^{\mathcal{P},M}_t$  (see (\ref{truncationbar}) in the following). There is a difficulty appears here: the equations (\ref{1.1}) and (\ref{1.3}); (\ref{truncationM}) and (\ref{truncationbar}) are defined with respect to different Poisson point measures (on different probability spaces). To overcome this difficulty, it is convenient to use similar
equations  driven by the same Poisson point measure.  We make a coupling argument to construct $\bar{x}^M_t$, $\bar{x}^{\mathcal{P}}_{t}$, $\bar{x}^{\mathcal{P},M}_{t}$ and $\bar{x}_t$ (see (\ref{xMbar}), (\ref{xPbar}), (\ref{xPMbar}) and (\ref{xbar}) below) which have the same law as $X^M_t$, $X^{\mathcal{P}}_{t}$, $\bar{X}^{\mathcal{P},M}_{t}$ and $X_t$ (see (\ref{truncationM}), (\ref{1.3}), (\ref{truncationbar}) and (\ref{1.1})) respectively but are defined on the same probability space and
verify equations driven by the same Poisson point measure.  So to estimate the total variation distance between $X^{\mathcal{P}}_{t}$ and $X_{t}$, it is equivalent to estimate the total variation distance between $\bar{x}^{\mathcal{P}}_{t}$ and $\bar{x}_{t}$. We will see that $\bar{x}^M_t$ and $\bar{x}^{\mathcal{P},M}_{t}$ are simple functionals (belong to $\mathcal{S}^d$) in the sense of Section 3.2.  We prove below in \textbf{Lemma 3.12}  that the Malliavin-Sobolev norms of $\bar{x}^M_t$ and $\bar{x}^{\mathcal{P},M}_{t}$ are bounded (uniformly in $M,\mathcal{P}$) and that the Malliavin covariance matrix of $\bar{x}^M_t$ is non-degenerate (uniformly in $M$). Passing to the limit $M\rightarrow\infty$, we give below in \textbf{Lemma 3.11}  that $\bar{x}^{\mathcal{P},M}_{t}\rightarrow\bar{x}^{\mathcal{P}}_{t}$ and $\bar{x}^{M}_{t}\rightarrow\bar{x}_{t}$ in $L^1$ distance. Then by using \textbf{Lemma 3.3}, $\bar{x}_t$ and $\bar{x}^{\mathcal{P}}_{t}$ are "Malliavin smooth functionals" (belong to $\mathcal{D}_\infty^d$), and  we prove below in \textbf{Lemma 3.13} that the Malliavin-Sobolev norms of $\bar{x}_t$ and $\bar{x}^{\mathcal{P}}_{t}$ are bounded (uniformly in $\mathcal{P}$) and  that the Malliavin covariance matrix of $\bar{x}_t$ is non-degenerate. So applying \textbf{Proposition 3.6.1},  the Euler scheme $X^{\mathcal{P}}_{t}$  converges to $X_t$ in total variation distance. 

Now we  give the proof of (\ref{002}).
We first introduce an intermediate equation.
\begin{eqnarray}
\bar{X}^{\mathcal{P},M}_{t}&=&X_0+a^M_{T}\Delta+\int_{0}^{t}b(\tau(r),\bar{X}^{\mathcal{P},M}_{\tau(r)},\rho^{\mathcal{P}}_{\tau(r)})dr  \nonumber\\
&+&\int_{0}^{t}\int_{\mathbb{R}^d\times \mathbb{R}^d}{c}_M(\tau(r),v,z,\bar{X}^{\mathcal{P},M}_{\tau(r)-},\rho^{\mathcal{P}}_{\tau(r)-})N_{\rho^{\mathcal{P}}_{\tau(r)-}}(dv,dz,dr).  \label{truncationbar}
\end{eqnarray}We notice that we take $\rho^{\mathcal{P}}_{\tau(r)}$ (the law of $X^{\mathcal{P}}_{\tau(r)}$) instead of $\rho^{\mathcal{P},M}_{\tau(r)}$ (the law of $X^{\mathcal{P},M}_{\tau(r)}$) in the above equation, so (\ref{truncationbar}) a variant of (\ref{truncation}).

Now we make a coupling argument similar to Section 2.7. We will do optimal coupling between $X^{\mathcal{P}}_{\tau(t)-}$ and $X_{t-}$ in $W_{2+\varepsilon_{\ast}}$ distance. This is the same strategy as the optimal coupling between $X^{\mathcal{P},M}_{\tau(t)-}$ and $X_{t-}$ in $W_{2+\varepsilon_{\ast}}$ distance in Section 2.7.
For a small $\varepsilon_\ast>0$,  we  take $\Pi^{\mathcal{P}}_t(dv_3,dv_4)$ to be the optimal $W_{2+\varepsilon_{\ast}}-$coupling of $\rho^{\mathcal{P}}_{\tau(t)-}(dv_3)$ and $\rho_{t-}(dv_4)$, that is
\[(W_{2+\varepsilon_{\ast}}(\rho^{\mathcal{P}}_{\tau(t)-},\rho_{t-}))^{2+\varepsilon_{\ast}}=\int_{\mathbb{R}^d\times\mathbb{R}^d}\vert v_3-v_4\vert^{2+\varepsilon_{\ast}}\Pi^{\mathcal{P}}_t(dv_3,dv_4).\]Then we  construct $(\eta^3_t(w),\eta^4_t(w))$ which represents  $\Pi^{\mathcal{P}}_t$ in the sense of \textbf{Lemma 2.5}. So we have \[\int_0^1\phi(\eta^3_t(w),\eta^4_t(w)) dw=\int_{\mathbb{R}^d\times\mathbb{R}^d}\phi(v_3,v_4)\Pi^{\mathcal{P}}_t(dv_3,dv_4).\]We consider some auxiliary equations (with  $\mathcal{N}(dw,dz,dr)$ the Poisson point measure on the state space $[0,1]\times\mathbb{R}^d$ with intensity measure $dw\mu(dz)dr$ defined in Section 2.7):
\begin{eqnarray}
\bar{x}_{t}&=&X_0+\int_{0}^{t}b(r,\bar{x}_{r},\rho_{r})dr  +\int_{0}^{t}\int_{[0,1]\times \mathbb{R}^d}{c}(r,\eta^4_r(w),z,\bar{x}_{r-},\rho_{r-})\mathcal{N}(dw,dz,dr),  \label{xbar}
\end{eqnarray}
\begin{eqnarray}
\bar{x}^{\mathcal{P}}_{t}&=&X_0+\int_{0}^{t}b(\tau(r),\bar{x}^{\mathcal{P}}_{\tau(r)},\rho^{\mathcal{P}}_{\tau(r)})dr  \nonumber\\
&+&\int_{0}^{t}\int_{[0,1]\times \mathbb{R}^d}{c}(\tau(r),\eta^3_r(w),z,\bar{x}^{\mathcal{P}}_{\tau(r)-},\rho^{\mathcal{P}}_{\tau(r)-})\mathcal{N}(dw,dz,dr).  \label{xPbar}
\end{eqnarray}
\begin{eqnarray}
\bar{x}^M_{t}&=&X_0+a^M_T\Delta+\int_{0}^{t}b(r,\bar{x}^M_{r},\rho_{r})dr  +\int_{0}^{t}\int_{[0,1]\times \mathbb{R}^d}{c}_M(r,\eta^4_r(w),z,\bar{x}^M_{r-},\rho_{r-})\mathcal{N}(dw,dz,dr),  \label{xMbar}
\end{eqnarray}
\begin{eqnarray}
\bar{x}^{\mathcal{P},M}_{t}&=&X_0+a^M_T\Delta+\int_{0}^{t}b(\tau(r),\bar{x}^{\mathcal{P},M}_{\tau(r)},\rho^{\mathcal{P}}_{\tau(r)})dr  \nonumber\\
&+&\int_{0}^{t}\int_{[0,1]\times \mathbb{R}^d}{c}_M(\tau(r),\eta^3_r(w),z,\bar{x}^{\mathcal{P},M}_{\tau(r)-},\rho^{\mathcal{P}}_{\tau(r)-})\mathcal{N}(dw,dz,dr).  \label{xPMbar}
\end{eqnarray}
One can check that $\bar{x}_{t}$, $\bar{x}^{\mathcal{P}}_{t}$, $\bar{x}^{M}_{t}$, and  $\bar{x}^{\mathcal{P},M}_{t}$ have the same law as ${X}_{t}$, ${X}^{\mathcal{P}}_{t}$, $X^{M}_{t}$, and  $\bar{X}^{\mathcal{P},M}_{t}$ (solutions of the equations (\ref{1.1}), (\ref{1.3}), (\ref{truncationM}) and (\ref{truncationbar})) respectively.
We stress that  $\bar{x}_{t}$,  $\bar{x}^{M}_{t}$, and  $\bar{x}^{\mathcal{P},M}_{t}$ are different from  ${x}_{t}$,  ${x}^{M}_{t}$, and  ${x}^{\mathcal{P},M}_{t}$ (see (\ref{x}), (\ref{xM}) and (\ref{xPM})). This is because we take different couplings so $\eta^1_r(w)\neq\eta^3_r(w)$  and $\eta^2_r(w)\neq\eta^4_r(w)$. We also remark that we take $\rho_r$ instead of $\rho^M_r$ in (\ref{xMbar}) and take $\rho^{\mathcal{P}}_{t}$ instead of $\rho^{\mathcal{P},M}_{t}$ in (\ref{xPMbar}), so that 
 we can obtain the following lemma.
\begin{lemma}
Assume that the \textbf{Hypothesis 2.1} holds true.
Then   \[i)\quad \sup_{\mathcal{P}}\mathbb{E}\vert \bar{x}_{t}^{\mathcal{P},M}-\bar{x}_{t}^{\mathcal{P}}\vert\rightarrow0,\]\[ii)\quad \mathbb{E}\vert \bar{x}_{t}^{M}-\bar{x}_{t}\vert\rightarrow0,\]as $M\rightarrow\infty$. 

\end{lemma}
\begin{proof}
These results are obtained in a standard way (see the proof of \textbf{Lemma 2.6}).
\end{proof}
We notice that $\bar{x}_{t}^{\mathcal{P},M}$ and $\bar{x}_{t}^{M}$ are simple functionals (belong to $\mathcal{S}^d$) in the sense of Section 3.2. Then we have
\begin{lemma}
Assume that \textbf{Hypothesis 2.1, 2.2, 2.3} and \textbf{2.4} hold true. 

$a)$ For any $p\geq1,l\geq0$, there exists a constant $C_{l,p}$ depending on $l,p,d$ and $T$ such that for every $0<t\leq T$, \[\sup_{\mathcal{P}}\sup_{M}(\Vert\bar{x}_{t}^{\mathcal{P},M}\Vert_{L,l,p}+\Vert\bar{x}_{t}^{M}\Vert_{L,l,p})\leq C_{l,p}.\]

$b)$ For any $p\geq1,0<t\leq T$ such that $t>\frac{2dp}{\theta}$, there exists a constant $C_p$ depending on $p,d,T$ such that  \[\sup_{M}(\mathbb{E}(1/\det \sigma_{\bar{x}_{t}^{M}})^p)\leq C_p.\]

$c)$ For any $\varepsilon_\ast>0$, there exists a constant $C_p$ depending on $\varepsilon_\ast,d,T$ such that for every $\vert\mathcal{P}\vert\leq1$, we have \[i)\quad\Vert D\bar{x}_{t}^{\mathcal{P},M}-D\bar{x}_{t}^{M}\Vert_{L^2(\Omega;l_2\times\mathbb{R}^d)}\leq C\vert\mathcal{P}\vert^{\frac{1}{2+\varepsilon_{\ast}}},\]\[ii)\quad\Vert D\bar{x}_{t}^{M}-D\bar{x}_{t}\Vert_{L^2(\Omega;l_2\times\mathbb{R}^d)}\rightarrow0,\quad \text{as }M\rightarrow\infty.\]
\end{lemma}
\begin{proof}
We get $a)$ by an analogous argument to the proof of \textbf{Lemma 3.7} $i)$. We have $b)$ in a similar way to the proof of \textbf{Lemma 3.8} $i)$. we obtain $c)$ $i)$ and  $ii)$ by some analogous arguments to the proofs of \textbf{Lemma 3.9} $i)$ and $ii)$ respectively. 
\end{proof}
Then applying \textbf{Lemma 3.3}, by passing to the limit $M\rightarrow\infty$, we obtain the following consequence.
\begin{lemma}
Assume that \textbf{Hypothesis 2.1, 2.2, 2.3} and \textbf{2.4} hold true. 

$a)$ $\bar{x}_{t}^{\mathcal{P}}$ and $\bar{x}_{t}$ both belong to $\mathcal{D}_{\infty}^d$. For any $p\geq1,l\geq0$, there exists a constant $C_{l,p}$ depending on $l,p,d$ and $T$ such that for every $0<t\leq T$, \[\sup_{\mathcal{P}}(\Vert\bar{x}_{t}^{\mathcal{P}}\Vert_{L,l,p}+\Vert\bar{x}_{t}\Vert_{L,l,p})\leq C_{l,p}.\]

$b)$ For any $p\geq1,0<t\leq T$ such that $t>\frac{2dp}{\theta}$, there exists a constant $C_p$ depending on $p,d,T$ such that  \[\mathbb{E}(1/\det \sigma_{\bar{x}_{t}})^p\leq C_p.\]

$c)$ For any $\varepsilon_\ast>0$, there exists a constant $C_p$ depending on $\varepsilon_\ast,d,T$ such that for every $\vert\mathcal{P}\vert\leq1$, we have \[\Vert D\bar{x}_{t}^{\mathcal{P}}-D\bar{x}_{t}\Vert_{L^2(\Omega;l_2\times\mathbb{R}^d)}\leq C\vert\mathcal{P}\vert^{\frac{1}{2+\varepsilon_{\ast}}}.\]
\end{lemma}
\begin{proof}
\textbf{Proof of }$a)$: We apply \textbf{Lemma 3.3 (A)} with $F_M=(\bar{x}_{t}^{\mathcal{P},M},\bar{x}_{t}^{M})$ and $F=(\bar{x}_{t}^{\mathcal{P}},\bar{x}_{t})$. By \textbf{Lemma 3.11} $i)$, $ii)$ and \textbf{Lemma 3.12} $a)$, we obtain our results.

\textbf{Proof of }$b)$: We apply \textbf{Lemma 3.3 (B)} with $F_M=\bar{x}_{t}^{M}$ and $F=\bar{x}_{t}$. By \textbf{Lemma 3.12} $b)$ and \textbf{Lemma 3.12} $c$ $ii)$, it follows that $\mathbb{E}(1/\det \sigma_{\bar{x}_{t}})^p\leq C_p.$

\textbf{Proof of }$c)$: We apply \textbf{Lemma 3.3 (C)} with $(\bar{F}_M,F_M)=(\bar{x}_{t}^{\mathcal{P},M},\bar{x}_{t}^{M})$ and $(\bar{F},F)=(\bar{x}_{t}^{\mathcal{P}},\bar{x}_{t})$. By \textbf{Lemma 3.12} $c)$ $i)$, we have $\Vert D\bar{x}_{t}^{\mathcal{P}}-D\bar{x}_{t}\Vert_{L^2(\Omega;l_2\times\mathbb{R}^d)}\leq C\vert\mathcal{P}\vert^{\frac{1}{2+\varepsilon_{\ast}}}.$
\end{proof}
Finally, we can give the proof of (\ref{002}). We recall that $\bar{x}_{t}$ and $\bar{x}^{\mathcal{P}}_{t}$ have the same law as ${X}_{t}$ and ${X}^{\mathcal{P}}_{t}$ respectively. For any $\bar{\varepsilon}>0$, we take $\varepsilon,\varepsilon_\ast>0$ such that $\varepsilon_\ast=\frac{\bar{\varepsilon}}{1-\bar{\varepsilon}}$ and $\varepsilon=\frac{\bar{\varepsilon}}{2}$. So  $\frac{2}{2+\varepsilon_\ast}(1-\varepsilon)=1-\bar{\varepsilon}.$
Thanks to \textbf{Lemma 3.13} $a),b)$, using \textbf{Proposition 3.6.1},  there exists a constant $C$ dependent on $\bar{\varepsilon},d,T$ such that for any partition $\mathcal{P}$ of the interval $[0,T]$ with $\vert\mathcal{P}\vert\leq1$, when $t>\frac{8d}{\theta}(\frac{8}{\bar{\varepsilon}}+1)$, we have \begin{eqnarray*}
d_{TV}(X^{\mathcal{P}}_{t},X_{t})&=&d_{TV}(\bar{x}^{\mathcal{P}}_{t},\bar{x}_{t})\\
&\leq& C[W_1(\bar{x}^{\mathcal{P}}_{t},\bar{x}_{t})+\Vert D{\bar{x}}_{t}^{\mathcal{P}}-D{\bar{x}}_{t}\Vert_{L^2(\Omega;l_2\times\mathbb{R}^d)}^2]^{1-\varepsilon}\\
&=& C[W_{1}({X}^{\mathcal{P}}_{t},{X}_{t})+\Vert D{\bar{x}}_{t}^{\mathcal{P}}-D{\bar{x}}_{t}\Vert_{L^2(\Omega;l_2\times\mathbb{R}^d)}^2]^{1-\varepsilon}\\
&\leq& C[\vert\mathcal{P}\vert+\vert\mathcal{P}\vert^\frac{2}{2+\varepsilon_{\ast}}]^{1-\varepsilon}\\
&\leq&C\vert\mathcal{P}\vert^{1-\bar{\varepsilon}}\rightarrow0,
\end{eqnarray*}where the second last inequality is obtained by \textbf{Lemma 3.13} $c)$ and (\ref{AB}). So (\ref{002}) is proved.
\end{proof}

\bigskip\bigskip

\textbf{Proofs of Theorem 2.3 and Theorem 2.4:}
\begin{proof}
\textbf{Proof of Theorem 2.3 $i)$:} 
We recall in Section 2.7 that  $x_t$ (solution of (\ref{x})) has the same law as $X_t$ (solution of (\ref{1.1})) and by \textbf{Theorem 2.1} $a)$, $\mathcal{L}(x_t)(dx)=\mathcal{L}(X_t)(dx)=p_t(x)dx$.  When $t>\frac{8d}{\theta}(2+d)$, \textbf{Lemma 3.7} $ii)$ and \textbf{Lemma 3.8} $ii)$ give that $\Vert x_t\Vert_{L,d+4,8d(2+d)}+\Sigma_{4(2+d)}(x_t)<\infty$ (with the notation $\Sigma_p(F)$ given in (\ref{sigma})). Then we  apply \textbf{Corollary 3.5.1} $i)$ and obtain that \begin{eqnarray}
\vert p_t(x)-\mathbb{E}(\varphi_\delta(X_t-x))\vert=\vert p_t(x)-\mathbb{E}(\varphi_\delta(x_t-x))\vert\leq C\delta^2,\label{31}
\end{eqnarray}where $C$ is a constant  dependent on $d$.

We recall by (\ref{W1}) the definition of the Wasserstein distance of order 1. Noticing $\Vert\nabla\varphi_\delta\Vert_{\infty}\leq \frac{1}{\delta^{d+1}}$, we get  \[\vert\mathbb{E}(\varphi_\delta(X_t-x))-\mathbb{E}(\varphi_\delta(X_t^{\mathcal{P},M}-x))\vert\leq W_1(X_t^{\mathcal{P},M},X_t)\frac{1}{\delta^{d+1}}.\]
So together with \textbf{Lemma 3.10}, there exists a constant $C$ dependent on $d$ and $T$ such that \begin{eqnarray}
\vert\mathbb{E}(\varphi_\delta(X_t-x))-\mathbb{E}(\varphi_\delta(X_t^{\mathcal{P},M}-x))\vert\leq C(\vert\mathcal{P}\vert+\sqrt{\varepsilon_M})\frac{1}{\delta^{d+1}}.\label{32}
\end{eqnarray}

Finally, applying the estimate $(4.6)$ in Theorem $4.1$ of $\cite{ref1}$ with $X^i_n=X^{\mathcal{P},M,i}_t$,  $\Theta^n_{s,s_n}(\rho)(dx)=\rho^{\mathcal{P},M}_t(dx)$ and $f(x)=\varphi_\delta(x)$, we get
\begin{eqnarray}
\vert\mathbb{E}(\varphi_\delta(X^{\mathcal{P},M}_t-x))-\frac{1}{N}\sum_{i=1}^N\mathbb{E}(\varphi_\delta(X_t^{\mathcal{P},M,i}-x))\vert\leq CV_N\frac{1}{\delta^{d+1}}.\label{33}
\end{eqnarray}Combining (\ref{31}), (\ref{32}) and (\ref{33}), \[\vert p_t(x)-\frac{1}{N}\sum_{i=1}^N\mathbb{E}(\varphi_\delta(X_t^{\mathcal{P},M,i}-x))\vert\leq C[(\vert\mathcal{P}\vert+\sqrt{\varepsilon_M})\frac{1}{\delta^{d+1}}+V_N\frac{1}{\delta^{d+1}}+\delta^2].\]

Then we optimize over $\delta$ and $N$. We choose \[\delta=(\vert\mathcal{P}\vert+\sqrt{\varepsilon_M})^{\frac{1}{d+3}}\] such that \[(\vert\mathcal{P}\vert+\sqrt{\varepsilon_M})\frac{1}{\delta^{d+1}}=\delta^2.\]So \[\vert p_t(x)-\frac{1}{N}\sum_{i=1}^N\mathbb{E}(\varphi_\delta(X_t^{\mathcal{P},M,i}-x))\vert\leq C[(\vert\mathcal{P}\vert+\sqrt{\varepsilon_M})^{\frac{2}{d+3}}+V_N(\vert\mathcal{P}\vert+\sqrt{\varepsilon_M})^{-\frac{d+1}{d+3}}].\] And we choose $N$  such that\[V_N\leq\vert\mathcal{P}\vert+\sqrt{\varepsilon_M},\] so \[(\vert\mathcal{P}\vert+\sqrt{\varepsilon_M})^{\frac{2}{d+3}}\geq V_N(\vert\mathcal{P}\vert+\sqrt{\varepsilon_M})^{-\frac{d+1}{d+3}}.\]Hence, eventually we have (\ref{003}).

\bigskip

\textbf{Proof of Theorem 2.3 $ii)$:} 
(\ref{003*}) is obtained in a similar way by using \textbf{Corollary 3.5.1} $ii)$.

\bigskip\bigskip

\textbf{Proof of Theorem 2.4 $i)$:}  We take $f\in C_b^\infty(\mathbb{R}^d)$.

\textit{Step 1:} We recall in Section 2.7 that $x^{\mathcal{P},M}_t$ (solution of (\ref{xPM})) has the same law as $X^{\mathcal{P},M}_t$ (solution of (\ref{truncation})) and $x_t$ (solution of (\ref{x})) has the same law as $X_t$ (solution of (\ref{1.1})).
We  notice by \textbf{Theorem 2.2} that for any $\varepsilon>0$, there exists a constant $C$ dependent on $\varepsilon,d,T$ such that when $t>\frac{8d}{\theta}(\frac{8}{\varepsilon}+1)$,
\begin{eqnarray}
\Big\vert\int_{\mathbb{R}^d}f(x)\rho_t(dx)-\int_{\mathbb{R}^d}f(x)\rho^{\mathcal{P},M}_t(dx)\Big\vert\leq C\left\Vert
f\right\Vert _{\infty }\times(\vert\mathcal{P}\vert+\sqrt{\varepsilon_M})^{1-\varepsilon},\label{2den1}
\end{eqnarray}with $\rho_t$ the law of $x_t$ (also of $X_t$) and $\rho^{\mathcal{P},M}_t$ the law of $x^{\mathcal{P},M}_t$ (also of $X^{\mathcal{P},M}_t$).

\textit{Step 2:}
We apply the regularization lemma \textbf{Lemma 3.5 (B)} $iii)$ with $F=x^{\mathcal{P},M}_t$ and $G=x_t$. For any $\bar{\varepsilon}>0,$ we take $\varepsilon,\varepsilon_\ast,\varepsilon_0>0$ such that \[\varepsilon_0=\frac{{\varepsilon}}{2-2{\varepsilon}},\quad\text{and}\quad\varepsilon_\ast=\frac{\bar{\varepsilon}}{1-\bar{\varepsilon}},\quad {\varepsilon}=\frac{\bar{\varepsilon}}{2}.\] So  \[\frac{2}{2+\varepsilon_\ast}(1-{\varepsilon})=1-\bar{\varepsilon}.\] Recalling in \textbf{Lemma 3.5 (B)} $iii)$ that $q=\frac{2}{1+\varepsilon_0}=\frac{4(1-\varepsilon)}{2-\varepsilon}$, we have
\begin{eqnarray}
\Big\vert\int_{\mathbb{R}^d}f(x)\rho^{\mathcal{P},M}_t(dx)-\int_{\mathbb{R}^d}f_\delta(x)\rho^{\mathcal{P},M}_t(dx)\Big\vert\leq C\left\Vert
f\right\Vert _{\infty }\times (\frac{\delta ^{2}}{\eta^4}+\eta^{-q}(\vert\mathcal{P}\vert+\varepsilon_M)^{\frac{q}{2+\varepsilon_{\ast}}}+\eta^\rho).\label{2den2}
\end{eqnarray} Here we have used the non-degenerated condition of $x_t$ and the fact that the Sobolev norms of $x^{\mathcal{P},M}_t$ and $x_t$ are bounded (uniformly in $\mathcal{P},M$). We have also taken advantage of \textbf{Lemma 3.9} $iii)$.

Then we optimize over $\delta,\eta$ and $\rho$. In order to keep the notations clear, we denote temporary that \[\mathcal{E}=(\vert\mathcal{P}\vert+\varepsilon_M)^{\frac{2}{2+\varepsilon_{\ast}}}.\]We take \[\eta=\mathcal{E}^\frac{\varepsilon}{4}\quad\text{and}\quad\delta=\sqrt{\mathcal{E}}\] such that \[\frac{\delta ^{2}}{\eta^4}=\eta^{-q}\mathcal{E}^\frac{q}{2}=\mathcal{E}^{1-\varepsilon}.\]We take moreover \[\rho=\frac{4(1-\varepsilon)}{\varepsilon}\] such that \[\eta^\rho=\mathcal{E}^{1-\varepsilon}.\]So (\ref{2den2}) becomes \begin{eqnarray}
\Big\vert\int_{\mathbb{R}^d}f(x)\rho^{\mathcal{P},M}_t(dx)-\int_{\mathbb{R}^d}f_\delta(x)\rho^{\mathcal{P},M}_t(dx)\Big\vert\leq C\left\Vert
f\right\Vert _{\infty }\times \mathcal{E}^{1-\varepsilon}=C\left\Vert
f\right\Vert _{\infty }\times (\vert\mathcal{P}\vert+\varepsilon_M)^{1-\bar{\varepsilon}},\label{2den2*}
\end{eqnarray}with  $C$ a constant depending on $\bar{\varepsilon},d,T$ and \begin{eqnarray}\delta=(\vert\mathcal{P}\vert+\varepsilon_M)^{\frac{1-\bar{\varepsilon}}{2-\bar{\varepsilon}}}=(\vert\mathcal{P}\vert+\varepsilon_M)^{\frac{1}{2}(1-\frac{\bar{\varepsilon}}{2-\bar{\varepsilon}})}.\label{2dendel}\end{eqnarray}

\textit{Step 3:} We apply the estimate $(4.6)$ in Theorem $4.1$ of $\cite{ref1}$. We notice that $\Vert\nabla f_\delta\Vert_\infty\leq C\Vert f\Vert_\infty\times\frac{1}{\delta^{d+1}}$. Then we obtain \begin{eqnarray}
\Big\vert\int_{\mathbb{R}^d}f_\delta(x)\rho^{\mathcal{P},M}_t(dx)-\frac{1}{N}\sum_{i=1}^N\mathbb{E}(f_\delta(X_t^{\mathcal{P},M,i}))\Big\vert\leq C\Vert f\Vert_\infty V_N\times\frac{1}{\delta^{d+1}}.\label{2den3}
\end{eqnarray}Now we optimize over $N$. We take $N$ such that \begin{eqnarray}V_N\leq(\vert\mathcal{P}\vert+\varepsilon_M)^{\frac{(1-\bar{\varepsilon})(d+3-\bar{\varepsilon})}{2-\bar{\varepsilon}}}=(\vert\mathcal{P}\vert+\varepsilon_M)^{\frac{d+3}{2}(1-\frac{(d+5)\bar{\varepsilon}-2\bar{\varepsilon}^2}{(2-\bar{\varepsilon})(d+3)})},\label{2denVN}\end{eqnarray}so\[V_N\times\frac{1}{\delta^{d+1}}\leq(\vert\mathcal{P}\vert+\varepsilon_M)^{1-\bar{\varepsilon}}.\]Then we have \begin{eqnarray}
\Big\vert\int_{\mathbb{R}^d}f_\delta(x)\rho^{\mathcal{P},M}_t(dx)-\frac{1}{N}\sum_{i=1}^N\mathbb{E}(f_\delta(X_t^{\mathcal{P},M,i}))\Big\vert\leq C\Vert f\Vert_\infty\times (\vert\mathcal{P}\vert+\varepsilon_M)^{1-\bar{\varepsilon}}.\label{2den3*}
\end{eqnarray}Combining (\ref{2den1}), (\ref{2den2*}) and (\ref{2den3*}), for all $f\in C_b^\infty(\mathbb{R}^d)$, with $\delta$ and $N$ given in (\ref{2dendel}) and (\ref{2denVN}), when $t>\frac{8d}{\theta}(\frac{16}{\bar{\varepsilon}}+1)$, we have \begin{eqnarray}
\int_{\mathbb{R}^d}f(x)\rho_t(dx)&=&\frac{1}{N}\sum_{i=1}^N\mathbb{E}(f_\delta(X_t^{\mathcal{P},M,i}))+\Vert f\Vert_\infty\times \mathcal{O}((\vert\mathcal{P}\vert+\sqrt{\varepsilon_M})^{1-\bar{\varepsilon}})\nonumber\\
&=&\frac{1}{N}\sum_{i=1}^N\mathbb{E}\int_{\mathbb{R}^d}f(X_t^{\mathcal{P},M,i}+y)\varphi_\delta(y)dy+\Vert f\Vert_\infty\times \mathcal{O}((\vert\mathcal{P}\vert+\sqrt{\varepsilon_M})^{1-\bar{\varepsilon}})\nonumber\\
&=&\frac{1}{N}\sum_{i=1}^N\mathbb{E}f(X_t^{\mathcal{P},M,i}+\delta\widetilde{\Delta})+\Vert f\Vert_\infty\times \mathcal{O}((\vert\mathcal{P}\vert+\sqrt{\varepsilon_M})^{1-\bar{\varepsilon}}),\label{2den}
\end{eqnarray}where $\widetilde{\Delta}$ is a $d-$dimensional standard Gaussian random variable independent of $X_t^{\mathcal{P},M,i}, i=1,\cdots,N$, and $\mathcal{O}(\bullet)$ is the Big $\mathcal{O}$ notation.

Since $C_b^\infty(\mathbb{R}^d)$ is dense in $C_b(\mathbb{R}^d)$, (\ref{2den})  holds for $f\in C_b(\mathbb{R}^d)$. Finally, by Lusin theorem, (\ref{2den}) also holds for any measurable and bounded function $f$.

\bigskip

\textbf{Proof of Theorem 2.4 $ii)$:} 
(\ref{004*}) is obtained in the same way as \textbf{Theorem 2.4} $i)$ by using \textbf{Lemma 3.5 (B)} $iv)$ in \textit{Step 2}. 

\end{proof}

\bigskip\bigskip

\section{Proofs}
\subsection{Sobolev norms}
In this section, we give the proof of \textbf{Lemma 3.7}. We explain   our strategy of the proof. 
We will first prove that $\sup\limits_{{\mathcal{P}}}\sup\limits_{{M}}\Vert {x}_{t}^{\mathcal{P},M}\Vert_{L,l,p}\leq C_{l,p}$, then by an analogous argument, we also have $\sup\limits_{{M}}\Vert {x}_{t}^{M}\Vert_{L,l,p}\leq C_{l,p}$. Afterwards, recalling $\mathbb{E}\vert x^M_t-x_t\vert\rightarrow0$ in \textbf{Lemma 2.6} $i)$, and applying \textbf{Lemma 3.3} with $F_M=x^{M}_t$ and $F=x_t$, we get that ${x}_{t}$ belongs to $\mathcal{D}_{\infty}^d$ and $\Vert {x}_{t}\Vert_{L,l,p}\leq C_{l,p}$.

So now we only need to prove the following lemma.
\begin{lemma}
Under the \textbf{Hypothesis 2.1} and \textbf{Hypothesis 2.4} $b)$, for all $p\geq2, l\geq0$, there exists a constant $C_{l,p}$ depending on $l,p,d$ and $T$, such that 
\begin{eqnarray}
a)\quad \sup\limits_{\mathcal{P}}\sup\limits_{M}\mathbb{E}\sup\limits_{0<t\leq T}\vert {x}_{t}^{\mathcal{P},M}\vert_{l}^p\leq C_{l,p}, \label{EuD}
\end{eqnarray}and 
\begin{eqnarray}
b)\quad \sup\limits_{\mathcal{P}}\sup\limits_{M}\mathbb{E}\sup\limits_{0<t\leq T}\vert L{x}_{t}^{\mathcal{P},M}\vert_{l}^p\leq C_{l,p}. \label{EuL}
\end{eqnarray}
\end{lemma}
Before we prove this lemma, we give some pre-estimations concerning the Sobolev norms of $Z^k_i$.
\begin{lemma}
Under the \textbf{Hypothesis 2.4} $b)$,  for every $l\geq0$, there exists a constant $C_{l}$ dependent on $l,d$ such that
\begin{eqnarray}
i)\quad\sup\limits_{k,i\in\mathbb{N}}\vert{Z}^k_i\vert_{1,l}\leq C_{l}, \label{LLL0}
\end{eqnarray}
\begin{eqnarray}
ii)\quad\sup\limits_{k,i\in\mathbb{N}}\vert L{Z}_i^{k}\vert_l\leq C_{l}. \label{LZTrue}
\end{eqnarray}
\end{lemma}
\begin{proof}$i)$ We notice by the definition (\ref{D}) that $D^Z_{(k,i,j)}{Z}^k_{i,j}=\xi_{i}^k$,\quad $D^Z_{(k^\prime,i^\prime,j^\prime)}{Z}^k_{i,j}=0$, for $k^\prime\neq k$, $i^\prime\neq i$ or $j^\prime\neq j$,\quad $D^{\Delta} {Z}^k_i=0$. We recall that $\xi^k_i=\Psi_k({Z}^k_i)$ in Section 3.2. We observe that  using \textbf{Lemma 3.1} $a)$, for any $k,i\in\mathbb{N}$ we have
\[
\vert{Z}^k_i\vert_{1,l}\leq \vert\xi^k_i\vert_{l-1}=\vert\Psi_k({Z}^k_i)\vert_{l-1}\leq 1+C_l(\vert{Z}^k_i\vert_{1,l-1}+\vert{Z}^k_i\vert_{1,l-2}^{l-1}).
\]Since $\vert{Z}^k_i\vert_{1,1}=\vert\xi_i^k\vert\leq1$, there is a constant $C_l$ such that
$\vert{Z}^k_i\vert_{1,l}\leq C_{l}$. 

$ii)$
We notice by the definition (\ref{L}) that
\[L{{Z}}_{i,j}^{k}=-\partial_{z^k_{i,j}}(\xi_{i}^k)^2-\xi^k_{i}D^Z_{(k,i,j)}\ln [h({Z}^k_{i})].\]We observe that  $\vert\xi_{i}^k\vert\leq1$, and we have $\vert\partial_{z^k_{i,j}}(\xi_{i}^k)^2\vert=2\Psi_k(Z^k_i)\partial_{z^k_{i,j}}\Psi_k(Z^k_i)$ is bounded by a universal constant (see (\ref{psiborn})). These lead to \begin{eqnarray*}
\vert L{Z}_{i,j}^{k}\vert_l\leq C_l(1+\vert D^Z_{(k,i,j)}\ln [h({Z}^k_{i})]\vert_l).
\end{eqnarray*}

We recall by \textbf{Hypothesis 2.4} $b)$ that $h$ is infinitely differentiable and $\ln h$ has bounded derivatives of any order.  Applying \textbf{Lemma 3.1} $a)$ and using (\ref{LLL0}), 
\begin{eqnarray*}
\vert D^Z_{(k,i,j)}\ln [h({Z}^k_{i})]\vert_l&\leq& \vert \ln [h({Z}^k_{i})]\vert_{l+1}\\
&\leq& C_l+\vert\nabla\ln h(Z^k_i)\vert\vert Z^k_i\vert_{1,l+1} +C_l\sup_{2\leq\vert\beta\vert\leq l+1}\vert\partial^\beta\ln h(Z^k_i)\vert\vert Z^k_i\vert_{1,l}^{l+1}\\
&\leq&C_l.
\end{eqnarray*}
Then for any $k,i\in\mathbb{N}$, we obtain that
$\vert L{Z}^k_{i}\vert_l\leq C_l.$
\end{proof}
\bigskip
Now we give the proof of \textbf{Lemma 4.1}.
\begin{proof}

\textbf{Proof of $a)$:} We first prove (\ref{EuD}). We will prove by recurrence on $l$. One can easily check by  (\ref{BurkM}) with \[\bar{\Phi}(r,w,z,\omega,\rho)={c}_M(\tau(r),\eta^1_r(w),z,X^{\mathcal{P},M}_{\tau(r)-},\rho^{\mathcal{P},M}_{\tau(r)-})\] and by \textbf{Hypothesis 2.1} that for $l=0$, $\mathbb{E}\sup\limits_{0<t\leq T}\vert {x}_{t}^{\mathcal{P},M}\vert^p\leq C_{0,p}$. Then we assume that (\ref{EuD}) holds for $l-1$ with $l\geq1$ and for every $p\geq2$. We will show that (\ref{EuD}) also holds for $l$ and for every $p\geq2$. 
We notice by the definitions (\ref{D}) that  $D^\Delta_j\Delta=\bm{e_j}$, where $\bm{e_j}=(0,\cdots,0,1,0,\cdots,0)$ with value $1$ at the $j-$th component, $D^{\Delta,q}\Delta=0$ with $q\geq2$ and $D^Z\Delta=0$. Recalling the equation (\ref{xPMZ}), we write $\mathbb{E}\sup\limits_{0<t\leq T}\vert {x}_{t}^{\mathcal{P},M}\vert_{1,l}^p\leq C_{l,p}(1+A_1+A_2)$, with
\begin{eqnarray*}
&&A_1=\mathbb{E}\int_{0}^{T}\vert b(\tau(r),{x}^{\mathcal{P},M}_{\tau(r)},\rho^{\mathcal{P},M}_{\tau(r)})\vert_{1,l}^pdr,\\
&&A_2=\mathbb{E}(\sum_{k=1}^M\sum_{i=1}^{J^k_T}\vert{c}(\tau({T}_i^{k}),\eta^1_{T^k_i}({W}_i^k),{Z}_i^{k},{x}^{\mathcal{P},M}_{\tau({T}_i^{k})-},\rho^{\mathcal{P},M}_{\tau({T}_i^{k})-})\vert_{1,l})^p.
\end{eqnarray*}
Using \textbf{Lemma 3.1} $a)$, \textbf{Hypothesis 2.1} and the recurrence hypothesis, 
\begin{eqnarray}
A_1&\leq&C_{l,p}[\mathbb{E}\int_0^T\vert\nabla_xb(\tau(r),{x}^{\mathcal{P},M}_{\tau(r)},\rho^{\mathcal{P},M}_{\tau(r)})\vert^p\vert {x}^{\mathcal{P},M}_{\tau(r)}\vert_{1,l}^pdr\nonumber\\
&+&\mathbb{E}\int_0^T\sup_{2\leq\vert\beta\vert\leq l}\vert\partial^\beta_xb(\tau(r),{x}^{\mathcal{P},M}_{\tau(r)},\rho^{\mathcal{P},M}_{\tau(r)})\vert^p\vert {x}^{\mathcal{P},M}_{\tau(r)}\vert_{1,l-1}^{lp}dr] \nonumber\\
&\leq&C_{l,p}[\int_0^T\mathbb{E}\vert {x}^{\mathcal{P},M}_{\tau(r)}\vert_{1,l}^pdr+\int_0^T\mathbb{E}\vert {x}^{\mathcal{P},M}_{\tau(r)}\vert_{1,l-1}^{lp}dr]\nonumber\\
&\leq& C_{l,p}[1+\int_0^T\mathbb{E}\vert {x}^{\mathcal{P},M}_{\tau(r)}\vert_{1,l}^pdr]. \label{A1}
\end{eqnarray}
Next, we estimate $A_2$.  By \textbf{Lemma 3.1} $a)$, \textbf{Hypothesis 2.1} and \textbf{Lemma 4.2}, for any $k,i\in\mathbb{N}$,
\begin{eqnarray*}
&&\vert{c}(\tau({T}_i^{k}),\eta^1_{T^k_i}({W}_i^k),{Z}_i^{k},{x}^{\mathcal{P},M}_{\tau({T}_i^{k})-},\rho^{\mathcal{P},M}_{\tau({T}_i^{k})-})\vert_{1,l}\\
&&\leq (\vert\nabla_{z}{c}(\tau({T}_i^{k}),\eta^1_{T^k_i}({W}_i^k),{Z}_i^{k},{x}^{\mathcal{P},M}_{\tau({T}_i^{k})-},\rho^{\mathcal{P},M}_{\tau({T}_i^{k})-})\vert+\vert\nabla_{x}{c}(\tau({T}_i^{k}),\eta^1_{T^k_i}({W}_i^k),{Z}_i^{k},{x}^{\mathcal{P},M}_{\tau({T}_i^{k})-},\rho^{\mathcal{P},M}_{\tau({T}_i^{k})-})\vert)\\
&&\times(\vert {Z}_i^{k}\vert_{1,l}+\vert {x}^{\mathcal{P},M}_{\tau({T}_i^{k})-}\vert_{1,l})\\
&&+{C}_{l}\sup\limits_{2\leq\vert\beta_1+\beta_2\vert\leq l}(\vert\partial_{z}^{\beta_1}\partial_{x}^{\beta_2}{c}(\tau({T}_i^{k}),\eta^1_{T^k_i}({W}_i^k),{Z}_i^{k},{x}^{\mathcal{P},M}_{\tau({T}_i^{k})-},\rho^{\mathcal{P},M}_{\tau({T}_i^{k})-})\vert)\times(\vert {Z}_i^{k}\vert_{1,l-1}^{l}+\vert {x}^{\mathcal{P},M}_{\tau({T}_i^{k})-}\vert_{1,l-1}^{l})\\
&&\leq{C}_{l}\bar{c}({Z}_i^{k})(1+\vert {x}^{\mathcal{P},M}_{\tau({T}_i^{k})-}\vert_{1,l}+\vert {x}^{\mathcal{P},M}_{\tau({T}_i^{k})-}\vert_{1,l-1}^{l}).
\end{eqnarray*}
Hence, using (\ref{BurkM}) with \[\bar{\Phi}(r,w,z,\omega,\rho)=\bar{c}(z)(1+\vert {x}^{\mathcal{P},M}_{\tau(r)-}\vert_{1,l}+\vert {x}^{\mathcal{P},M}_{\tau(r)-}\vert_{1,l-1}^{l}),\] \textbf{Hypothesis 2.1} and the recurrence hypothesis, it follows that 
\begin{eqnarray}
A_2&\leq& {C}_{l,p}\mathbb{E}\vert\sum_{k=1}^M\sum_{i=1}^{J^k_T}\bar{c}({Z}_i^{k})(1+\vert {x}^{\mathcal{P},M}_{\tau({T}_i^{k})-}\vert_{1,l}+\vert {x}^{\mathcal{P},M}_{\tau({T}_i^{k})-}\vert_{1,l-1}^{l})\vert^p \nonumber\\
&\leq& {C}_{l,p}\mathbb{E}\vert\int_0^T\int_{[0,1]\times\mathbb{R}^d}\bar{c}(z)(1+\vert {x}^{\mathcal{P},M}_{\tau(r)-}\vert_{1,l}+\vert {x}^{\mathcal{P},M}_{\tau(r)-}\vert_{1,l-1}^{l})\mathcal{N}(dw,dz,dr)\vert^p\nonumber\\
&\leq& C_{l,p}[1+\int_0^T\mathbb{E}\vert {x}^{\mathcal{P},M}_{\tau(r)-}\vert_{1,l}^pdr+\int_0^T\mathbb{E}\vert {x}^{\mathcal{P},M}_{\tau(r)-}\vert_{1,l-1}^{lp}dr]\nonumber\\
&\leq& C_{l,p}[1+\int_0^T\mathbb{E}\vert {x}^{\mathcal{P},M}_{\tau(r)}\vert_{1,l}^pdr].  \label{A2}
\end{eqnarray}
Combining (\ref{A1}) and (\ref{A2}),  one has
\begin{eqnarray*}
\mathbb{E}\sup\limits_{0<t\leq T}\vert {x}_{t}^{\mathcal{P},M}\vert_{1,l}^p\leq {C}_{l,p}[1+\int_0^T\mathbb{E}\vert {x}_{\tau(r)}^{\mathcal{P},M}\vert_{1,l}^pdr]. 
\end{eqnarray*}
So  we conclude by Gronwall lemma that 
\begin{eqnarray}
\sup_{\mathcal{P},M}\mathbb{E}\sup\limits_{0<t\leq T}\vert {x}_{t}^{\mathcal{P},M}\vert_{1,l}^p\leq C_{l,p}. \label{D*}
\end{eqnarray}

\bigskip\bigskip

\textbf{Proof of $b)$:} 
Now we pass to the proof of (\ref{EuL}). We also prove it by recurrence on $l$.

\textbf{Step 1:} We take first $l=0$.
We notice by the definition (\ref{L})  that $L\Delta=\Delta$. So having in mind that $\Delta$ has finite moments of any order, and recalling the equation (\ref{xPMZ}), we write $\mathbb{E}\sup\limits_{0<t\leq T}\vert L{x}_{t}^{\mathcal{P},M}\vert^p\leq C_{0,p}(1+S_1+S_2)$, with
\begin{eqnarray*}
&&S_1=\mathbb{E}\int_{0}^{T}\vert Lb(\tau(r),{x}^{\mathcal{P},M}_{\tau(r)},\rho^{\mathcal{P}}_{\tau(r)})\vert^pdr,\\
&&S_2=\mathbb{E}(\sum_{k=1}^M\sum_{i=1}^{J_T^k}\vert L{c}(\tau({T}_i^{k}),\eta^1_{T^k_i}({W}_i^k),{Z}_i^{k},{x}^{\mathcal{P},M}_{\tau({T}_i^{k})-},\rho^{\mathcal{P},M}_{\tau({T}_i^{k})-})\vert)^p.
\end{eqnarray*}
Using \textbf{Lemma 3.1} $c)$, \textbf{Hypothesis 2.1} and (\ref{EuD}), 
\begin{eqnarray}
S_1&\leq&\mathbb{E}\int_0^T\vert\nabla_xb(\tau(r),{x}^{\mathcal{P},M}_{\tau(r)},\rho^{\mathcal{P},M}_{\tau(r)})\vert^p\vert  L{x}^{\mathcal{P},M}_{\tau(r)}\vert^pdr \nonumber\\
&+&\mathbb{E}\int_0^T\sup_{\vert\beta\vert=2}\vert\partial^\beta_xb(\tau(r),{x}^{\mathcal{P},M}_{\tau(r)},\rho^{\mathcal{P},M}_{\tau(r)})\vert^p\vert {x}^{\mathcal{P},M}_{\tau(r)}\vert_{1,1}^{2p}dr \nonumber\\
&\leq&C_{0,p}[1+\int_0^T\mathbb{E}\vert  L{x}^{\mathcal{P},M}_{\tau(r)}\vert^pdr]. \label{S1}
\end{eqnarray}
For $S_2$, we observe that using \textbf{Lemma 3.1} $c)$, \textbf{Lemma 4.2},  and \textbf{Hypothesis 2.1}, for any $k,i\in\mathbb{N}$,
\begin{eqnarray*}
&&\vert L{c}(\tau({T}_i^{k}),\eta^1_{T^k_i}({W}_i^k),{Z}_i^{k},{x}^{\mathcal{P},M}_{\tau({T}_i^{k})-},\rho^{\mathcal{P},M}_{\tau({T}_i^{k})-})\vert\\
&&\leq (\vert\nabla_{z}{c}(\tau({T}_i^{k}),\eta^1_{T^k_i}({W}_i^k),{Z}_i^{k},{x}^{\mathcal{P},M}_{\tau({T}_i^{k})-},\rho^{\mathcal{P},M}_{\tau({T}_i^{k})-})\vert+\vert\nabla_{x}{c}(\tau({T}_i^{k}),\eta^1_{T^k_i}({W}_i^k),{Z}_i^{k},{x}^{\mathcal{P},M}_{\tau({T}_i^{k})-},\rho^{\mathcal{P},M}_{\tau({T}_i^{k})-})\vert)\\
&&\times(\vert L{Z}_i^{k}\vert+\vert  L{x}^{\mathcal{P},M}_{\tau({T}_i^{k})-}\vert)\\
&&+\sup\limits_{\vert\beta_1+\beta_2\vert=2}(\vert\partial_{z}^{\beta_1}\partial_{x}^{\beta_2}{c}(\tau({T}_i^{k}),\eta^1_{T^k_i}({W}_i^k),{Z}_i^{k},{x}^{\mathcal{P},M}_{\tau({T}_i^{k})-},\rho^{\mathcal{P},M}_{\tau({T}_i^{k})-})\vert)\times(\vert {Z}_i^{k}\vert_{1,1}^{2}+\vert {x}^{\mathcal{P},M}_{\tau({T}_i^{k})-}\vert_{1,1}^{2})\\
&&\leq{C}\bar{c}({Z}_i^{k})(1+\vert L{x}^{\mathcal{P},M}_{\tau({T}_i^{k})-}\vert+\vert {x}^{\mathcal{P},M}_{\tau({T}_i^{k})-}\vert_{1,1}^{2}).
\end{eqnarray*}
Therefore, using (\ref{BurkM}) with \[\bar{\Phi}(r,w,z,\omega,\rho)=\bar{c}(z)(1+\vert L{x}^{\mathcal{P},M}_{\tau(r)-}\vert+\vert {x}^{\mathcal{P},M}_{\tau(r)-}\vert_{1,1}^{2}),\] using \textbf{Hypothesis 2.1} and (\ref{EuD}), it follows that 
\begin{eqnarray}
S_2&\leq& C_{0,p}\mathbb{E}(\sum_{k=1}^M\sum_{i=1}^{J^k_T}\bar{c}({Z}_i^{k})(1+\vert L{x}^{\mathcal{P},M}_{\tau({T}_i^{k})-}\vert+\vert {x}^{\mathcal{P},M}_{\tau({T}_i^{k})-}\vert_{1,1}^{2}))^p \nonumber\\
&=& {C}_{0,p}\mathbb{E}\vert\int_0^T\int_{[0,1]\times\mathbb{R}^d}\bar{c}(z)(1+\vert L{x}^{\mathcal{P},M}_{\tau(r)-}\vert+\vert {x}^{\mathcal{P},M}_{\tau(r)-}\vert_{1,1}^{2})\mathcal{N}(dw,dz,dr)\vert^p\nonumber\\
&\leq& C_{0,p}[1+\int_0^T\mathbb{E}\vert L{x}^{\mathcal{P},M}_{\tau(r)-}\vert^pdr].  \label{S2}
\end{eqnarray}
Combining (\ref{S1}) and (\ref{S2}),  one has
\begin{eqnarray*}
\mathbb{E}\sup\limits_{0<t\leq T}\vert L{x}_{t}^{\mathcal{P},M}\vert^p\leq {C}_{0,p}[1+\int_0^T\mathbb{E}\vert L{x}_{\tau(r)}^{\mathcal{P},M}\vert^pdr].
\end{eqnarray*}
Applying Gronwall lemma, we obtain 
\begin{eqnarray*}
\sup_{\mathcal{P},M}\mathbb{E}\sup\limits_{0<t\leq T}\vert L{x}_{t}^{\mathcal{P},M}\vert^p\leq C_{0,p}.
\end{eqnarray*} 
\textbf{Step 2:}
Now we assume that (\ref{EuL}) holds for $l-1$ with $l\geq1$ and for every $p\geq2$. We will show that (\ref{EuL}) also holds for $l$ and for every $p\geq2$. We recall the equation (\ref{xPMZ}) and that $L\Delta=\Delta$,  $D^\Delta_j\Delta=\bm{e_j}$, where $\bm{e_j}=(0,\cdots,0,1,0,\cdots,0)$ with value $1$ at the $j-$th component, $D^{\Delta,q}\Delta=0$ with $q\geq2$ and $D^Z\Delta=0$. 
Having in mind that $\Delta$ has finite moments of any order, we write $\mathbb{E}\sup\limits_{0<t\leq T}\vert L{x}_{t}^{\mathcal{P},M}\vert_{l}^p\leq C_{l,p}(1+B_1+B_2)$, with
\begin{eqnarray*}
&&B_1=\mathbb{E}\int_{0}^{T}\vert Lb(\tau(r),{x}^{\mathcal{P},M}_{\tau(r)},\rho^{\mathcal{P}}_{\tau(r)})\vert_{l}^pdr,\\
&&B_2=\mathbb{E}(\sum_{k=1}^M\sum_{i=1}^{J_T^k}\vert L{c}(\tau({T}_i^{k}),\eta^1_{T^k_i}({W}_i^k),{Z}_i^{k},{x}^{\mathcal{P},M}_{\tau({T}_i^{k})-},\rho^{\mathcal{P},M}_{\tau({T}_i^{k})-})\vert_{l})^p.
\end{eqnarray*}
Using \textbf{Lemma 3.1} $b)$, \textbf{Hypothesis 2.1}, (\ref{EuD}) and the recurrence hypothesis, 
\begin{eqnarray}
B_1&\leq&C_{l,p}[\mathbb{E}\int_0^T\vert\nabla_xb(\tau(r),{x}^{\mathcal{P},M}_{\tau(r)},\rho^{\mathcal{P},M}_{\tau(r)})\vert^p\vert  L{x}^{\mathcal{P},M}_{\tau(r)}\vert_{l}^pdr \nonumber\\
&+&\mathbb{E}\int_0^T\sup_{2\leq\vert\beta\vert\leq l+2}\vert\partial^\beta_xb(\tau(r),{x}^{\mathcal{P},M}_{\tau(r)},\rho^{\mathcal{P},M}_{\tau(r)})\vert^p(1+\vert {x}^{\mathcal{P},M}_{\tau(r)}\vert_{l+1}^{(l+2)p})(1+\vert L{x}^{\mathcal{P},M}_{\tau(r)}\vert_{l-1}^p)dr] \nonumber\\
&\leq&C_{l,p}[1+\int_0^T\mathbb{E}\vert  L{x}^{\mathcal{P},M}_{\tau(r)}\vert_{l}^pdr]. \label{B1}
\end{eqnarray}
Next, we estimate $B_2$. We observe that using \textbf{Lemma 3.1} $b)$, \textbf{Lemma 4.2},  and \textbf{Hypothesis 2.1}, for any $k,i\in\mathbb{N}$,
\begin{eqnarray*}
&&\vert L{c}(\tau({T}_i^{k}),\eta^1_{T^k_i}({W}_i^k),{Z}_i^{k},{x}^{\mathcal{P},M}_{\tau({T}_i^{k})-},\rho^{\mathcal{P},M}_{\tau({T}_i^{k})-})\vert_{l}\\
&&\leq (\vert\nabla_{z}{c}(\tau({T}_i^{k}),\eta^1_{T^k_i}({W}_i^k),{Z}_i^{k},{x}^{\mathcal{P},M}_{\tau({T}_i^{k})-},\rho^{\mathcal{P},M}_{\tau({T}_i^{k})-})\vert+\vert\nabla_{x}{c}(\tau({T}_i^{k}),\eta^1_{T^k_i}({W}_i^k),{Z}_i^{k},{x}^{\mathcal{P},M}_{\tau({T}_i^{k})-},\rho^{\mathcal{P},M}_{\tau({T}_i^{k})-})\vert)\\
&&\times(\vert L{Z}_i^{k}\vert_{l}+\vert  L{x}^{\mathcal{P},M}_{\tau({T}_i^{k})-}\vert_{l})\\
&&+{C}_{l}\sup\limits_{2\leq\vert\beta_1+\beta_2\vert\leq l+2}(\vert\partial_{z}^{\beta_1}\partial_{x}^{\beta_2}{c}(\tau({T}_i^{k}),\eta^1_{T^k_i}({W}_i^k),{Z}_i^{k},{x}^{\mathcal{P},M}_{\tau({T}_i^{k})-},\rho^{\mathcal{P},M}_{\tau({T}_i^{k})-})\vert)\\
&&\times(1+\vert {Z}_i^{k}\vert_{l+1}^{l+2}+\vert {x}^{\mathcal{P},M}_{\tau({T}_i^{k})-}\vert_{l+1}^{l+2})(1+\vert L{Z}_i^{k}\vert_{l-1}+\vert  L{x}^{\mathcal{P},M}_{\tau({T}_i^{k})-}\vert_{l-1})\\
&&\leq{C}_{l}\bar{c}({Z}_i^{k})(1+\vert L{x}^{\mathcal{P},M}_{\tau({T}_i^{k})-}\vert_{l}+(1+\vert {x}^{\mathcal{P},M}_{\tau({T}_i^{k})-}\vert_{l+1}^{l+2})(1+\vert  L{x}^{\mathcal{P},M}_{\tau({T}_i^{k})-}\vert_{l-1})).
\end{eqnarray*}
Hence, using (\ref{BurkM}) with \[\bar{\Phi}(r,w,z,\omega,\rho)=\bar{c}(z)(1+\vert L{x}^{\mathcal{P},M}_{\tau(r)-}\vert_{l}+(1+\vert {x}^{\mathcal{P},M}_{\tau(r)-}\vert_{l+1}^{l+2})(1+\vert  L{x}^{\mathcal{P},M}_{\tau(r)-}\vert_{l-1})),\] using \textbf{Hypothesis 2.1}, (\ref{EuD}), and the recurrence hypothesis, it follows that 
\begin{eqnarray}
B_2&\leq& C_{l,p}\mathbb{E}(\sum_{k=1}^M\sum_{i=1}^{J^k_T}\bar{c}({Z}_i^{k})(1+\vert L{x}^{\mathcal{P},M}_{\tau({T}_i^{k})-}\vert_{l}+(1+\vert {x}^{\mathcal{P},M}_{\tau({T}_i^{k})-}\vert_{l+1}^{l+2})(1+\vert  L{x}^{\mathcal{P},M}_{\tau({T}_i^{k})-}\vert_{l-1})))^p \nonumber\\
&\leq& {C}_{l,p}\mathbb{E}\vert\int_0^T\int_{[0,1]\times\mathbb{R}^d}\bar{c}(z)(1+\vert L{x}^{\mathcal{P},M}_{\tau(r)-}\vert_{l}+(1+\vert {x}^{\mathcal{P},M}_{\tau(r)-}\vert_{l+1}^{l+2})(1+\vert  L{x}^{\mathcal{P},M}_{\tau(r)-}\vert_{l-1}))\mathcal{N}(dw,dz,dr)\vert^p\nonumber\\
&\leq& C_{l,p}[1+\int_0^T\mathbb{E}\vert L{x}^{\mathcal{P},M}_{\tau(r)-}\vert_{l}^pdr].  \label{B2}
\end{eqnarray}
Combining (\ref{B1}) and (\ref{B2}),  one has
\begin{eqnarray*}
\mathbb{E}\sup\limits_{0<t\leq T}\vert L{x}_{t}^{\mathcal{P},M}\vert_{l}^p\leq {C}_{l,p}[1+\int_0^T\mathbb{E}\vert L{x}_{\tau(r)}^{\mathcal{P},M}\vert_{l}^pdr].
\end{eqnarray*}
Then we conclude by Gronwall lemma that 
\begin{eqnarray}
\sup_{\mathcal{P},M}\mathbb{E}\sup\limits_{0<t\leq T}\vert L{x}_{t}^{\mathcal{P},M}\vert_{l}^p\leq C_{l,p}. \label{L*}
\end{eqnarray}
\end{proof}
So now \textbf{Lemma 4.1} is proved. Then by an analogous argument, we also have $\sup\limits_{{M}}\Vert {x}_{t}^{M}\Vert_{L,l,p}\leq C_{l,p}$. Finally, recalling that $\mathbb{E}\vert x^M_t-x_t\vert\rightarrow0$ in \textbf{Lemma 2.6} $i)$, and applying \textbf{Lemma 3.3} with $F_M=x^{M}_t$ and $F=x_t$, we get that ${x}_{t}$ belongs to $\mathcal{D}_{\infty}^d$ and $\Vert {x}_{t}\Vert_{L,l,p}\leq C_{l,p}$.

\bigskip\bigskip

\subsection{Covariance matrices}
In this section, we give the proof of \textbf{Lemma 3.8}. 

\textbf{Proof of $i)$:} We proceed in 4 steps.

\textbf{Step 1} 
We notice by the definitions (\ref{D}) and the equation (\ref{xMZ}) that  for any  $M\in\mathbb{N}$, any  $k_0,{i_0}\in\mathbb{N},{j}\in\{1,\cdots,d\}$,
\begin{eqnarray}
&&D^Z_{(k_0,{i_0},{j})}{x}^{M}_{t}=\int_{{T}_{{i_0}}^{k_0}}^{t}\nabla_xb(r,{x}^{M}_{r},\rho_{r})D^Z_{(k_0,{i_0},{j})}{x}^{M}_{r}dr \nonumber\\
&&+\mathbbm{1}_{\{0<T_{{i_0}}^{k_0}\leq t\}}\mathbbm{1}_{\{1\leq k_0\leq M\}}\xi_{{i_0}}^{k_0}\partial_{z_{{j}}}{c}({T}_{{i_0}}^{k_0},\eta^2_{T_{i_0}^{k_0}}(W_{i_0}^{k_0}),{Z}_{{i_0}}^{k_0},{x}^{M}_{{T}_{{i_0}}^{k_0}-},\rho_{{T}_{{i_0}}^{k_0}-}) \nonumber\\
&&+\sum_{k=1}^M\sum_{{T}_{{i_0}}^{k_0}<{T}_{i}^{k}\leq t}\nabla_x{c}({T}_{i}^{k},\eta^2_{T_{i}^{k}}(W_{i}^{k}),{Z}_{i}^{k},{x}^{M}_{{T}_{i}^{k}-},\rho_{{T}_{i}^{k}-})D^Z_{(k_0,{i_0},{j})}{x}^{M}_{{T}_{i}^{k}-}, \quad\quad \label{cM1}
\end{eqnarray} 
\begin{eqnarray}
&&D^\Delta_j{x}_{t}^{M}=a^M_{T}\bm{e_j}+\int_{0}^{t}\nabla_xb(r,{x}^{M}_{r},\rho_{r})D^\Delta_{j}{x}^{M}_{r}dr +\sum_{k=1}^M\sum_{i=1}^{J_t^k}\nabla_x{c}({T}_{i}^{k},\eta^2_{T_{i}^{k}}(W_{i}^k),{Z}_{i}^{k},{x}^{M}_{{T}_{i}^{k}-},\rho_{{T}_{i}^{k}-})D^\Delta_j{x}^{M}_{{T}_{i}^{k}-}, \quad\quad\quad \label{cM1*}
\end{eqnarray}
where $\bm{e_j}=(0,\cdots,0,1,0,\cdots,0)$ with value $1$ at the $j-$th component.

Now we introduce $(Y^{M}_{t})_{t\in[0,T]}$ (this is a variant of the tangent flow and for simplicity of the expression, we still call it the tangent flow) which is the matrix solution of the linear equation
\begin{eqnarray*}
Y^{M}_{t}=I_d&+&\int_{0}^{t}\nabla_xb(r,x^{M}_{r},\rho_{r})Y^{M}_{r}dr+\sum_{k=1}^M\sum_{i=1}^{J^k_t}\nabla_x{c}({T}_{{i}}^{k},\eta^2_{T^k_i}({W}_{{i}}^k),{Z}_{{i}}^{k},{x}^{M}_{{T}_{{i}}^{k}-},\rho_{{T}_{{i}}^{k}-})Y^{M}_{{T}_{{i}}^{k}-}.
\end{eqnarray*}
And using It$\hat{o}$'s formula, the inverse matrix $\widetilde{Y}^{M}_{t}=(Y^{M}_{t})^{-1}$ verifies the equation
\begin{eqnarray}
\widetilde{Y}^{M}_{t}=I_d&-&\int_{o}^{t}\widetilde{Y}^{M}_{r}\nabla_xb(r,{x}^{M}_{r},\rho_{r})dr-\sum_{k=1}^M\sum_{i=1}^{J_t^k}\widetilde{Y}^{M}_{{T}_{{i}}^{k}-}\nabla _{x}{c}(I_d+\nabla
_{x}{c})^{-1}({T}_{{i}}^{k},\eta^2_{T^k_i}({W}_{{i}}^k),{Z}_{{i}}^{k},{x}^{M}_{{T}_{{i}}^{k}-},\rho_{{T}_{{i}}^{k}-}).\quad\quad\quad \label{inverse}
\end{eqnarray}

\bigskip

Applying \textbf{Hypothesis 2.1} and \textbf{Hypothesis 2.2}, with $C_p$ a constant not dependent on $M$, one also has 
\begin{equation}
\mathbb{E}(\sup_{0<t\leq T}(\left\Vert Y^{M}_{t}\right\Vert ^{p}+\left\Vert \widetilde{Y}^{M}_{t}\right\Vert ^{p}))\leq C_p<\infty.   \label{cM2}
\end{equation}%

Then using the uniqueness of solution to the equation (\ref{cM1}) and (\ref{cM1*}), one
obtains 
\begin{equation}
D^Z_{(k,{i},{j})}{x}^{M}_{t}=\mathbbm{1}_{\{0<T_{{i}}^{k}\leq t\}}\mathbbm{1}_{\{1\leq k\leq M\}}\xi _{{i}}^kY^{M}_{t}\widetilde{Y}^{M}_{T_{{i}}^{k}}\partial_{z_{{j}}}{c}({T}_{{i}}^{k},\eta^2_{T_{i}^{k}}(W_{i}^k),{Z}_{{i}}^{k},{x}^{M}_{{T}_{{i}}^{k}-},\rho_{{T}_{{i}}^{k}-}),  \label{cM3}
\end{equation}and $D^\Delta_j{x}^{M}_{t}=a^M_{T}Y^{M}_{t}\bm{e_j}$.

In the following, we denote the lowest eigenvalue of the Malliavin covariance matrix $\sigma_{{x}_{t}^{M}}$ by $\lambda_{t}^{M}$. Then we have (recalling the definitions (\ref{Mcov}) and (\ref{Cov})) \[\lambda_{t}^{M}=\inf\limits_{\vert\zeta\vert=1}\langle \sigma_{{x}_{t}^{M}}\zeta,\zeta\rangle\geq\inf\limits_{\vert\zeta\vert=1}\sum_{k=1}^M\sum_{i=1}^{J_t^k}\sum_{j=1}^{d }\langle D^Z_{(k,i,j)}{{x}_{t}^{M}},\zeta\rangle^2+\inf\limits_{\vert\zeta\vert=1}\sum_{j=1}^{d }\langle D^\Delta_{j}{{x}_{t}^{M}},\zeta\rangle^2.\]
By (\ref{cM3}),
\begin{eqnarray*}
\lambda _{t}^{M} &\geq&\inf\limits_{\vert\zeta\vert=1}\sum_{k=1}^M\sum_{i=1}^{J_t^k}\sum_{j=1}^{d }\vert\xi_i^k\vert^2\langle \partial_{z_{{j}}}{c}({T}_{{i}}^{k},\eta^2_{T_{i}^{k}}(W_{i}^k),{Z}_{{i}}^{k},{x}^{M}_{{T}_{{i}}^{k}-},\rho_{{T}_{{i}}^{k}-}),(Y^{M}_{t}\widetilde{Y}^{M}_{T_{{i}}^{k}})^{\ast}\zeta\rangle^2+\inf\limits_{\vert\zeta\vert=1}\sum_{j=1}^{d }\vert a^M_{T}\vert^2\langle \bm{e_j},(Y^{M}_{t})^{\ast}\zeta\rangle^2,
\end{eqnarray*}%
where $Y^\ast$ denotes the transposition of a matrix $Y$.

We recall the ellipticity hypothesis (\textbf{Hypothesis 2.3}): there exists a non-negative function $\underline{c}(z)$ such that 
\[
\sum_{j=1}^d\langle\partial_{z_j}{c}(r,v,z,x,\rho),\zeta\rangle^{2}\geq \underline{c}%
(z)\vert\zeta\vert^2.
\]
So we deduce that%
\begin{eqnarray*}
\lambda_{t}^{M}&\geq& \inf\limits_{\vert\zeta\vert=1}\sum_{k=1}^M\sum_{i=1}^{J_t^k}\vert\xi_i^k\vert^2
\underline{c}({Z}_{i}^{k})\vert(Y^{M}_{t}\widetilde{Y}^{M}_{T_{{i}}^{k}})^{\ast}\zeta\vert^2+\vert a^M_{T}\vert^2\inf\limits_{\vert\zeta\vert=1}\vert(Y^{M}_{t})^{\ast}\zeta\vert^2\\
&\geq& \sum_{k=1}^M\sum_{i=1}^{J_t^k}
\vert\xi_i^k\vert^2\underline{c}({Z}^k_i)\Vert\widetilde{Y}^{M}_{t}\Vert^{-2}\Vert{Y}^{M}_{T_{{i}}^{k}}\Vert^{-2}+ \vert a^M_{T}\vert^2\Vert\widetilde{Y}^{M}_{t}\Vert^{-2}\\
&\geq& (\inf\limits_{0<t\leq T}\Vert\widetilde{Y}^{M}_{t}\Vert^{-2}\Vert{Y}^{M}_{t}\Vert^{-2})(\sum_{k=1}^M\sum_{i=1}^{J_t^k}
\vert\xi_i^k\vert^2\underline{c}({Z}^k_i)+ \vert a^M_{T}\vert^2).
\end{eqnarray*}

We denote 
\begin{eqnarray}
\chi _{t}^{M}=\sum_{k=1}^M\sum_{i=1}^{J_t^k}
\vert\xi_i^k\vert^2\underline{c}({Z}^k_i). \label{chi}
\end{eqnarray}
By (\ref{cM2}), $(\mathbb{E}\sup\limits_{0\leq
t\leq T}\Vert\widetilde{Y}^M_{t}\Vert^{4dp}\Vert{Y}^M_{t}\Vert^{4dp})^{1/2}\leq C_{d,p}<\infty,$ so that using Schwartz inequality, we have
\begin{eqnarray}
\mathbb{E}\vert\frac{1}{\det \sigma_{{x}^{M}_{t}}}\vert^p\leq\mathbb{E}(\vert \lambda_{t}^{M}\vert^{-dp})\leq C(\mathbb{E}(\vert \chi _{t}^{M}+\vert a^M_{T}\vert^2 \vert^{-2dp}))^{\frac{1}{2}}. \label{Step1}
\end{eqnarray}

\textbf{Step 2 }Since it is not easy to compute $\mathbb{E}(\vert \chi _{t}^{M}+\vert a^M_{T}\vert^2 \vert^{-2dp}))$ directly, we make the following argument where the idea comes originally  from $\cite{ref5}$. Let $\Gamma(p)=\int_0^\infty {s}^{p-1}e^{-{s}}d{s}$ be the Gamma function. By a change of variables, we have the numerical
equality 
\[
\frac{1}{\vert \chi _{t}^{M}+\vert a^M_{T}\vert^2\vert^{2dp}}=\frac{1}{\Gamma (2dp)}\int_{0}^{\infty
}{s}^{2dp-1}e^{-{s}(\chi _{t}^{M}+\vert a^M_{T}\vert^2)}d{s}, 
\]%
which, by taking expectation, gives 
\begin{eqnarray}
\mathbb{E}(\frac{1}{\vert \chi _{t}^{M}+\vert a^M_{T}\vert^2\vert^{2dp}})=\frac{1}{\Gamma (2dp)}\int_{0}^{\infty
}{s}^{2dp-1}\mathbb{E}(e^{-{s}(\chi _{t}^{M}+\vert a^M_{T}\vert^2)})d{s}.  \label{Gamma}
\end{eqnarray}

\textbf{Step 3} Now we compute $\mathbb{E}(e^{-{s}(\chi _{t}^{M}+\vert a^M_{T}\vert^2)})$ for any ${s}>0$. We recall that $I_1=B_1$, $I_k=B_{k}-B_{k-1}, k\geq2$ (given in Section 2.4), and $\xi_i^k=\Psi_k(Z^k_i)$ (given in Section 3.2). 
We take $\Lambda_k(dz,dr)$ to be a Poisson point measure  with intensity
\[\widehat{\Lambda}_k(dz,dr):=\mathbbm{1}_{I_k}(z)\mu(dz)dr.\]Since for different $k\in\mathbb{N}$, $I_k$ are disjoint, the Poisson point measures $\Lambda_{k},k\in\mathbb{N}$ are independent. And we put $\Theta_M(dz,dr)=\sum\limits_{k=1}^M\Lambda_k(dz,dr)$. Then
\begin{eqnarray*}
\chi _{t}^M &=&\sum_{k=1}^{M}\int_{0}^{t}\int_{I_k} \vert\Psi_k(z)\vert^2 {\underline{c}}(z)\Lambda_k(dz,dr)=\int_{0}^{t}\int_{B_M} \Psi(z) {\underline{c}}(z)\Theta_M(dz,dr),
\end{eqnarray*}with $\Psi(z)=\sum\limits_{k=1}^{\infty}\vert\Psi_k(z)\vert^2\mathbbm{1}_{I_k}(z)$.
Using It\^{o} formula,
\begin{eqnarray*}
\mathbb{E}(e^{-{s}\chi _{t}^{M}}) &=&1+\mathbb{E}\int_0^t\int_{B_M} (e^{-{s}({\chi _{r-}^{M}}+\Psi(z)\underline{c}(z))}-e^{-{s}{\chi _{r-}^{M}}})\widehat{\Theta}_M(dz,dr) \\
&=&1-\int_{0}^{t}\mathbb{E}(e^{-{s}{\chi _{r}^{M}}})dr\int_{B_M} (1-e^{-{s}\Psi(z)\underline{c}(z)})\sum_{k=1}^M\mathbbm{1}_{I_k}(z)\mu(dz).
\end{eqnarray*}%
Solving the above equation we obtain 
\begin{eqnarray*}
\mathbb{E}(e^{-{s}{\chi _{t}^{M}}})&=&\exp(-t\int_{B_M} (1-e^{-{s}\Psi(z)\underline{c}(z)})\sum_{k=1}^M\mathbbm{1}_{I_k}(z)\mu(dz))\\
&=&\exp (-t\sum_{k=1}^M\int_{I_k} (1-e^{-{s}\vert\Psi_k(z)\vert^2\underline{c}(z)})\mu(dz))\\
&\leq&\exp (-t\sum_{k=1}^M\int_{I_k} (1-e^{-{s}\mathbbm{1}_{[k-\frac{3}{4},k-\frac{1}{4}]}(\vert z\vert)\underline{c}(z)})\mu(dz))\\
&=&\exp (-t\sum_{k=1}^M\int_{I_k} (1-e^{-{s}\underline{c}(z)})\mathbbm{1}_{[k-\frac{3}{4},k-\frac{1}{4}]}(\vert z\vert)\mu(dz))\\
&=&\exp (-t\int_{B_{M}} (1-e^{-{s}\underline{c}(z)})\nu(dz)),
\end{eqnarray*}with
\[\nu(dz)=\sum_{k=1}^{\infty}\mathbbm{1}_{[k-\frac{3}{4},k-\frac{1}{4}]}(\vert z\vert)\mu(dz).\]

On the other hand, we denote 
\begin{eqnarray*}
\bar{\chi} _{t}^{M}=\int_{0}^{t}\int_{B_M^c} \Psi(z) {\underline{c}}(z)\Theta(dz,dr),
\end{eqnarray*}where $B_{M}^c$ denote the complementary set of $B_{M}$ and $\Theta$ is a Poisson point measure with intensity $\mu(dz)dr$. Then in the same way, 
\begin{eqnarray*}
\mathbb{E}(e^{-{s}\bar{\chi} _{t}^{M}}) \leq \exp (-t\int_{B_{M}^c} (1-e^{-{s}\underline{c}(z)})\nu(dz)).
\end{eqnarray*}
We recall by (\ref{aMt}) that $a^M_{T}=\sqrt{T\int_{\{\vert z\vert>M\}}\underline{c}(z)  \mu(dz)}\geq \sqrt{\mathbb{E}\bar{\chi} _{t}^{M}}$.
Notice that using Jensen inequality for the convex function $f(x)=e^{-{s}x}$,  ${s},x>0$, we have
\[
e^{-{s}\vert a^M_{T}\vert^2}\leq e^{-{s}\mathbb{E}\bar{\chi} _{t}^{M}}\leq \mathbb{E}(e^{-{s}\bar{\chi} _{t}^{M}})\leq \exp (-t\int_{B_{M}^c} (1-e^{-{s}\underline{c}(z)})\nu(dz)).
\]
So for every  $M\in\mathbb{N},$ we deduce that
\begin{eqnarray}
&&\mathbb{E}(e^{-{s}(\chi _{t}^{M}+\vert a^M_{T}\vert^2)})=\mathbb{E}(e^{-{s}\chi _{t}^{M}})\times e^{-{s}\vert a^M_{T}\vert^2} \nonumber\\
&&\leq \exp (-t\int_{B_{M}} (1-e^{-{s}\underline{c}(z)})\nu(dz))\times \exp (-t\int_{B_{M}^c} (1-e^{-{s}\underline{c}(z)})\nu(dz)) \nonumber\\
&&=\exp (-t\int_{\mathbb{R}^d} (1-e^{-{s}\underline{c}(z)})\nu(dz)), \label{Step3}
\end{eqnarray}%
and the last term does not depend on $M.$

\bigskip

\textbf{Step 4} 
Now we use the Lemma 14 from $\cite{ref3}$, which states the following.
\begin{lemma}
We consider an abstract measurable space $B$, a $\sigma$-finite measure $\mathcal{M}$ on this space and a non-negative measurable function $f:B\rightarrow\mathbb{R}_+$ such that $\int_Bfd\mathcal{M}<\infty.$ For $t>0$ and $p\geq1$, we note
\[
\beta_f({s})=\int_B(1-e^{-{s}f(x)})\mathcal{M}(dx)\quad and\quad I_{t}^p(f)=\int_0^\infty {s}^{p-1}e^{-t\beta_f({s})}d{s}.
\]
We suppose that for some $t>0$ and $p\geq1$, 
\begin{equation}
\underline{\lim }_{u\rightarrow \infty }\frac{1}{\ln u}\mathcal{M}(f\geq 
\frac{1}{u})>\frac{p}{t},  \label{cm}
\end{equation}%
then $I_{t}^p(f)<\infty.$
\end{lemma}
\bigskip

We will use the above lemma for $\mathcal{M}(dz)=\nu(dz)$, $f(z)=\underline{c}(z)$ and $B=\mathbb{R}^d$. Thanks to (\ref{undergamma}) in \textbf{Hypothesis 2.4},  
\begin{equation}
\underline{\lim }_{u\rightarrow \infty }\frac{1}{\ln u}\nu(\underline{c}\geq \frac{1}{u})=\theta>0. \label{cm6}
\end{equation}%
Then for every $p\geq 1, 0<t\leq T$, when $\theta>\frac{2dp}{t}$ (i.e. $t>\frac{2dp}{\theta}$),  we deduce from (\ref{Step1}),(\ref{Gamma}),(\ref{Step3}) and \textbf{Lemma 4.3} that 
\begin{eqnarray}
\sup\limits_{M}\mathbb{E}\vert\frac{1}{\det \sigma_{{x}^{M}_{t}}}\vert^p&\leq&\sup\limits_{M}\mathbb{E}(\vert \lambda_{t}^{M}\vert^{-dp})\leq C\sup\limits_{M}(\mathbb{E}(\vert \chi _{t}^{M}+\vert a^M_{T}\vert^2 \vert^{-2dp}))^{\frac{1}{2}}\nonumber\\
&\leq& C(\frac{1}{\Gamma (2dp)}\int_{0}^{\infty
}{s}^{2dp-1}\mathbb{E}(e^{-{s}(\chi _{t}^{M}+\vert a^M_{T}\vert^2)})d{s})^\frac{1}{2}\nonumber \\
&\leq &C(\frac{1}{\Gamma (2dp)}\int_{0}^{\infty
}{s}^{2dp-1}\exp (-t\int_{\mathbb{R}^d }(1-e^{-{s}\underline{c}(z)})\nu(dz)d{s})^\frac{1}{2}<\infty.\quad \quad\quad \label{detXt}
\end{eqnarray}

\bigskip\bigskip

\textbf{Proof of $ii)$:} We recall in \textbf{Lemma 2.6} $i)$ that $\mathbb{E}\vert x^M_t-x_t\vert\rightarrow0$, and in \textbf{Lemma 3.7} that $\Vert x^M_t\Vert_{L,l,p}\leq C_{l,p}$.
Moreover, by \textbf{Lemma 3.9} $ii)$, we know that $(Dx^M_t)_{M\in\mathbb{N}}$ is a Cauchy sequence in $L^2(\Omega;l^2\times\mathbb{R}^d)$. Then applying \textbf{Lemma 3.3 (B)} with $F_M=x^{M}_t$ and $F=x_t$, by (\ref{cM8}), we obtain (\ref{cov}).

 \qed

\bigskip\bigskip

\section{Appendix}
In the Appendix, we give the proof of \textbf{Lemma 3.9}.
\begin{proof}
\textbf{Proof of $i)$:} We notice by the definitions (\ref{D}) and the equations (\ref{xPMZ}), (\ref{xMZ}) that  for any partition $\mathcal{P}=\{0=r_0<r_1<\cdots<r_{n-1}<r_n=T\}$, $M\in\mathbb{N}$, any  $k_0,{i_0}\in\mathbb{N},{j}\in\{1,\cdots,d\}$,
\begin{eqnarray}
&&D^Z_{(k_0,{i_0},{j})}{x}^{\mathcal{P},M}_{t}=\int_{0}^{t}\nabla_xb(\tau(r),{x}^{\mathcal{P},M}_{\tau(r)},\rho^{\mathcal{P},M}_{\tau(r)})D^Z_{(k_0,{i_0},{j})}{x}^{\mathcal{P},M}_{\tau(r)}dr \nonumber\\
&&+\mathbbm{1}_{\{0<T_{{i_0}}^{k_0}\leq t\}}\mathbbm{1}_{\{1\leq k_0\leq M\}}\xi_{{i_0}}^{k_0}\partial_{z_{{j}}}{c}(\tau({T}_{{i_0}}^{k_0}),\eta^1_{T_{i_0}^{k_0}}(W_{i_0}^{k_0}),{Z}_{{i_0}}^{k_0},{x}^{{\mathcal{P}},M}_{\tau^{{\mathcal{P}}}({T}_{{i_0}}^{k_0})-},\rho^{\mathcal{P},M}_{\tau^{{\mathcal{P}}}({T}_{{i_0}}^{k_0})-}) \nonumber\\
&&+\sum_{k=1}^M\sum_{i=1}^{J^k_t}\nabla_x{c}(\tau({T}_{i}^{k}),\eta^1_{T_{i}^{k}}(W_{i}^{k}),{Z}_{i}^{k},{x}^{{\mathcal{P}},M}_{\tau^{{\mathcal{P}}}({T}_{i}^{k})-},\rho^{\mathcal{P},M}_{\tau^{{\mathcal{P}}}({T}_{i}^{k})-})D^Z_{(k_0,{i_0},{j})}{x}^{\mathcal{P},M}_{\tau({T}_{i}^{k})-}, \quad\quad \label{DZxPM}
\end{eqnarray}
\begin{eqnarray}
&&D^\Delta_j{x}_{t}^{\mathcal{P},M}=a^M_{T}\bm{e_j}+\int_{0}^{t}\nabla_xb(\tau(r),{x}^{\mathcal{P},M}_{\tau(r)},\rho^{\mathcal{P},M}_{\tau(r)})D^\Delta_{j}{x}^{\mathcal{P},M}_{\tau(r)}dr \nonumber\\
&&+\sum_{k=1}^M\sum_{i=1}^{J^k_t}\nabla_x{c}(\tau({T}_{i}^{k}),\eta^1_{T_{i}^{k}}(W_{i}^k),{Z}_{i}^{k},{x}^{{\mathcal{P}},M}_{\tau^{{\mathcal{P}}}({T}_{i}^{k})-},\rho^{\mathcal{P},M}_{\tau^{{\mathcal{P}}}({T}_{i}^{k})-})D^\Delta_j{x}^{\mathcal{P},M}_{\tau({T}_{i}^{k})-}, \quad\quad \label{DdxPM}
\end{eqnarray}where $\bm{e_j}=(0,\cdots,0,1,0,\cdots,0)$ with value $1$ at the $j-$th component. And 
\begin{eqnarray}
&&D^Z_{(k_0,{i_0},{j})}{x}^{M}_{t}=\int_{0}^{t}\nabla_xb(r,{x}^{M}_{r},\rho_{r})D^Z_{(k_0,{i_0},{j})}{x}^{M}_{r}dr \nonumber\\
&&+\mathbbm{1}_{\{0<T_{{i_0}}^{k_0}\leq t\}}\mathbbm{1}_{\{1\leq k_0\leq M\}}\xi_{{i_0}}^{k_0}\partial_{z_{{j}}}{c}({T}_{{i_0}}^{k_0},\eta^2_{T_{i_0}^{k_0}}(W_{i_0}^{k_0}),{Z}_{{i_0}}^{k_0},{x}^{M}_{{T}_{{i_0}}^{k_0}-},\rho_{{T}_{{i_0}}^{k_0}-}) \nonumber\\
&&+\sum_{k=1}^M\sum_{i=1}^{J^k_t}\nabla_x{c}({T}_{i}^{k},\eta^2_{T_{i}^{k}}(W_{i}^k),{Z}_{i}^{k},{x}^{M}_{{T}_{i}^{k}-},\rho_{{T}_{i}^{k}-})D^Z_{(k_0,{i_0},{j})}{x}^{M}_{{T}_{i}^{k}-}, \quad\quad \label{DZxM}
\end{eqnarray} 
\begin{eqnarray}
&&D^\Delta_j{x}_{t}^{M}=a^M_{T}\bm{e_j}+\int_{0}^{t}\nabla_xb(r,{x}^{M}_{r},\rho_{r})D^\Delta_{j}{x}^{M}_{r}dr +\sum_{k=1}^M\sum_{i=1}^{J_t^k}\nabla_x{c}({T}_{i}^{k},\eta^2_{T_{i}^{k}}(W_{i}^k),{Z}_{i}^{k},{x}^{M}_{{T}_{i}^{k}-},\rho_{{T}_{i}^{k}-})D^\Delta_j{x}^{M}_{{T}_{i}^{k}-}. \quad\quad\quad \label{DdxM}
\end{eqnarray}

For $u\in l_2$, we will use the notation $\vert u\vert_{l_2}^2=\vert u_{({\bullet},{\circ},{\diamond})}\vert_{l_2}^2=\sum\limits_{k=1}^\infty\sum\limits_{i=1}^\infty\sum\limits_{j=1}^d\vert u_{(k,i,j)}\vert^2$.
We write $\mathbb{E}\vert D^Z_{({\bullet},{\circ},{\diamond})}{x}^{\mathcal{P},M}_{t}-D^Z_{({\bullet},{\circ},{\diamond})}{x}^{M}_{t}\vert_{l_2}^2\leq C[H_1+H_2+H_3]$, with
\begin{eqnarray*}
H_1=\mathbb{E}\vert \int_{0}^{t}\nabla_xb(\tau(r),{x}^{\mathcal{P},M}_{\tau(r)},\rho^{\mathcal{P},M}_{\tau(r)})D^Z_{({\bullet},{\circ},{\diamond})}{x}^{\mathcal{P},M}_{\tau(r)}dr-\int_{0}^{t}\nabla_xb(r,{x}^{M}_{r},\rho_{r})D^Z_{({\bullet},{\circ},{\diamond})}{x}^{M}_{r}dr \vert_{l_2}^2,
\end{eqnarray*}
\begin{eqnarray*}
H_2&=&\mathbb{E}\vert\mathbbm{1}_{\{0<T_{{\circ}}^{\bullet}\leq t\}}\mathbbm{1}_{\{1\leq {\bullet}\leq M\}}(\partial_{z_{\diamond}}{c}(\tau({T}_{\circ}^{\bullet}),\eta^1_{T_{\circ}^{\bullet}}(W_{\circ}^{\bullet}),{Z}_{\circ}^{\bullet},{x}^{{\mathcal{P}},M}_{\tau^{{\mathcal{P}}}({T}_{\circ}^{\bullet})-},\rho^{\mathcal{P},M}_{\tau^{{\mathcal{P}}}({T}_{\circ}^{\bullet})-})\\
&-&\partial_{z_{\diamond}}{c}({T}_{\circ}^{\bullet},\eta^2_{T_{\circ}^{\bullet}}(W_{\circ}^{\bullet}),{Z}_{\circ}^{\bullet},{x}^{M}_{{T}_{\circ}^{\bullet}-},\rho_{{T}_{\circ}^{\bullet}-}))\vert_{l_2}^2,
\end{eqnarray*}
\begin{eqnarray*}
H_3&=&\mathbb{E}\vert\sum_{k=1}^M \sum_{i=1}^{J^k_t}\nabla_x{c}(\tau({T}_{i}^{k}),\eta^1_{T_{i}^{k}}(W_{i}^k),{Z}_{i}^{k},{x}^{{\mathcal{P}},M}_{\tau^{{\mathcal{P}}}({T}_{i}^{k})-},\rho^{\mathcal{P},M}_{\tau^{{\mathcal{P}}}({T}_{i}^{k})-})D^Z_{({\bullet},{\circ},{\diamond})}{x}^{\mathcal{P},M}_{\tau({T}_{i}^{k})-}\\
&-&\sum_{k=1}^M\sum_{i=1}^{J^k_t}\nabla_x{c}({T}_{i}^{k},\eta^2_{T_{i}^{k}}(W_{i}^k),{Z}_{i}^{k},{x}^{M}_{{T}_{i}^{k}-},\rho_{{T}_{i}^{k}-})D^Z_{({\bullet},{\circ},{\diamond})}{x}^{M}_{{T}_{i}^{k}-}\vert_{l_2}^2.
\end{eqnarray*}

We take a small $\varepsilon_\ast>0$. We recall $\varepsilon_M$ in (\ref{epsM}). Firstly, using \textbf{Hypothesis 2.1},   we get
\begin{eqnarray*}
H_1&\leq& C[\mathbb{E} \int_{0}^{t}\vert\nabla_xb(\tau(r),{x}^{\mathcal{P},M}_{\tau(r)},\rho^{\mathcal{P},M}_{\tau(r)})-\nabla_xb(r,{x}^{M}_{r},\rho_{r})\vert^2\vert D^Z_{({\bullet},{\circ},{\diamond})}{x}^{M}_{r}\vert_{l_2}^2dr\nonumber\\
&+&\mathbb{E} \int_{0}^{t}\vert\nabla_xb(\tau(r),{x}^{\mathcal{P},M}_{\tau(r)},\rho^{\mathcal{P},M}_{\tau(r)})\vert^2\vert D^Z_{({\bullet},{\circ},{\diamond})}{x}^{\mathcal{P},M}_{\tau(r)}-D^Z_{({\bullet},{\circ},{\diamond})}{x}^{M}_{r}\vert_{l_2}^2dr]\nonumber\\
&\leq& C[\mathbb{E}\int_0^t[
\vert\mathcal{P}\vert^2+\vert{x}^{\mathcal{P},M}_{\tau(r)}-{x}^{M}_{r}\vert^{2}+(W_1({\rho}^{\mathcal{P},M}_{\tau(r)},{\rho}_{r}))^2]\vert D^Z_{({\bullet},{\circ},{\diamond})}{x}^{M}_{r}\vert_{l_2}^2dr\nonumber\\
&+& \int_{0}^{t}\mathbb{E}\vert D^Z_{({\bullet},{\circ},{\diamond})}{x}^{\mathcal{P},M}_{\tau(r)}-D^Z_{({\bullet},{\circ},{\diamond})}{x}^{M}_{r}\vert_{l_2}^2dr].
\end{eqnarray*}Then by \textbf{Lemma 3.7}, using H$\ddot{o}$lder inequality with conjugates $1+\frac{\varepsilon_\ast}{2}$ and $\frac{2+\varepsilon_\ast}{\varepsilon_\ast}$, by \textbf{Lemma 2.6} and (\ref{Eul}), we have
\begin{eqnarray}
H_1&\leq& C[
\vert\mathcal{P}\vert^2+\int_0^t(\mathbb{E}\vert{x}^{\mathcal{P},M}_{\tau(r)}-{x}^{M}_{r}\vert^{2+\varepsilon_\ast})^{\frac{2}{2+\varepsilon_\ast}}dr+\int_0^t(W_{2+\varepsilon_\ast}({\rho}^{\mathcal{P},M}_{\tau(r)},{\rho}_{r}))^2dr\nonumber\\
&+& \int_{0}^{t}\mathbb{E}\vert D^Z_{({\bullet},{\circ},{\diamond})}{x}^{\mathcal{P},M}_{\tau(r)}-D^Z_{({\bullet},{\circ},{\diamond})}{x}^{M}_{r}\vert_{l_2}^2dr]\nonumber\\
&\leq& C[(\vert\mathcal{P}\vert+\varepsilon_M)^{\frac{2}{2+\varepsilon_\ast}}+\int_{0}^{t}\mathbb{E}\vert D^Z_{({\bullet},{\circ},{\diamond})}{x}^{\mathcal{P},M}_{\tau(r)}-D^Z_{({\bullet},{\circ},{\diamond})}{x}^{M}_{r}\vert_{l_2}^2dr].\label{H1}
\end{eqnarray}

Secondly, using \textbf{Hypothesis 2.1} and the isometry of the Poisson point measure $\mathcal{N}$,   we get
\begin{eqnarray*}
H_2&=& \mathbb{E}\sum_{k=1}^M\sum_{i=1}^{J^k_t}\sum_{j=1}^d\vert\partial_{z_{j}}{c}(\tau({T}_{i}^{k}),\eta^1_{T_{i}^{k}}(W_{i}^k),{Z}_{i}^{k},{x}^{{\mathcal{P}},M}_{\tau^{{\mathcal{P}}}({T}_{i}^{k})-},\rho^{\mathcal{P},M}_{\tau^{{\mathcal{P}}}({T}_{i}^{k})-})-\partial_{z_{j}}{c}({T}_{i}^{k},\eta^2_{T_{i}^{k}}(W_{i}^k),{Z}_{i}^{k},{x}^{M}_{{T}_{i}^{k}-},\rho_{{T}_{i}^{k}-})\vert^2\nonumber\\
&\leq& C\mathbb{E}\sum_{k=1}^M\sum_{i=1}^{J^k_t}\vert\bar{c}(Z^k_i)\vert^2[\vert \tau^{{\mathcal{P}}}({T}_{i}^{k})-T^k_i\vert+\vert \eta^1_{T_{i}^{k}}(W_{i}^k)-\eta^2_{T_{i}^{k}}(W_{i}^k)\vert+\vert{x}^{{\mathcal{P}},M}_{\tau^{{\mathcal{P}}}({T}_{i}^{k})-}-{x}^{M}_{{T}_{i}^{k}-}\vert+W_1(\rho^{\mathcal{P},M}_{\tau^{{\mathcal{P}}}({T}_{i}^{k})-},\rho_{{T}_{i}^{k}-})]^2\nonumber\\
&\leq& C\mathbb{E}\int_0^t\int_{[0,1]\times\mathbb{R}^d}\vert\bar{c}(z)\vert^2[\vert \tau^{{\mathcal{P}}}(r)-r\vert^2+\vert \eta^1_{r}(w)-\eta^2_{r}(w)\vert^2\nonumber\\
&+&\vert{x}^{{\mathcal{P}},M}_{\tau^{{\mathcal{P}}}(r)-}-{x}^{M}_{r-}\vert^2+(W_1(\rho^{\mathcal{P},M}_{\tau^{{\mathcal{P}}}(r)-},\rho_{r-}))^2]\mathcal{N}(dw,dz,dr)\nonumber\\
&=& C\mathbb{E}\int_0^t\int_{[0,1]\times\mathbb{R}^d}\vert\bar{c}(z)\vert^2[\vert \tau^{{\mathcal{P}}}(r)-r\vert^2+\vert \eta^1_{r}(w)-\eta^2_{r}(w)\vert^2\nonumber\\
&+&\vert{x}^{{\mathcal{P}},M}_{\tau^{{\mathcal{P}}}(r)-}-{x}^{M}_{r-}\vert^2+(W_1(\rho^{\mathcal{P},M}_{\tau^{{\mathcal{P}}}(r)-},\rho_{r-}))^2]dw\mu(dz)dr.
\end{eqnarray*}Then by (\ref{coupling}), (\ref{Coupling}), \textbf{Lemma 2.6}, (\ref{Eul}), and H$\ddot{o}$lder inequality with conjugates $1+\frac{\varepsilon_\ast}{2}$ and $\frac{2+\varepsilon_\ast}{\varepsilon_\ast}$,  we have
\begin{eqnarray}
H_2&\leq& C[\vert\mathcal{P}\vert^2+\int_0^t(\mathbb{E}\vert{x}^{\mathcal{P},M}_{\tau(r)}-{x}_{r}\vert^{2+\varepsilon_\ast})^{\frac{2}{2+\varepsilon_\ast}}dr+\int_0^t(\mathbb{E}\vert{x}^{\mathcal{P},M}_{\tau(r)}-{x}^{M}_{r}\vert^{2+\varepsilon_\ast})^{\frac{2}{2+\varepsilon_\ast}}dr+\int_0^t(W_{2+\varepsilon_\ast}(\rho^{\mathcal{P},M}_{\tau^{{\mathcal{P}}}(r)},\rho_{r}))^2dr]\nonumber\\
&\leq& C(\vert\mathcal{P}\vert+\varepsilon_M)^{\frac{2}{2+\varepsilon_\ast}}.\label{H2}
\end{eqnarray}

Thirdly, we write $H_3\leq C[H_{3,1}+H_{3,2}]$, where
\begin{eqnarray*}
H_{3,1}&=& \mathbb{E} (\sum_{k=1}^M\sum_{i=1}^{J_t^k}\vert\nabla_x{c}(\tau({T}_{i}^{k}),\eta^1_{T_{i}^{k}}(W_{i}^k),{Z}_{i}^{k},{x}^{{\mathcal{P}},M}_{\tau^{{\mathcal{P}}}({T}_{i}^{k})-},\rho^{\mathcal{P},M}_{\tau^{{\mathcal{P}}}({T}_{i}^{k})-})\\
&-&\nabla_x{c}({T}_{i}^{k},\eta^2_{T_{i}^{k}}(W_{i}^k),{Z}_{i}^{k},{x}^{M}_{{T}_{i}^{k}-},\rho_{{T}_{i}^{k}-})\vert\vert D^Z_{({\bullet},{\circ},{\diamond})}{x}^{M}_{{T}_{i}^{k}-}\vert_{l_2})^2,
\end{eqnarray*}
\[H_{3,2}=\mathbb{E} (\sum_{k=1}^M\sum_{i=1}^{J^k_t}\vert\nabla_x{c}(\tau({T}_{i}^{k}),\eta^1_{T_{i}^{k}}(W_{i}^k),{Z}_{i}^{k},{x}^{{\mathcal{P}},M}_{\tau^{{\mathcal{P}}}({T}_{i}^{k})-},\rho^{\mathcal{P},M}_{\tau^{{\mathcal{P}}}({T}_{i}^{k})-})\vert\vert D^Z_{({\bullet},{\circ},{\diamond})}{x}^{\mathcal{P},M}_{\tau({T}_{i}^{k})-}-D^Z_{({\bullet},{\circ},{\diamond})}{x}^{M}_{{T}_{i}^{k}-}\vert_{l_2})^2.\]
Using \textbf{Hypothesis 2.1} and (\ref{BurkM}) with \[\bar{\Phi}(r,w,z,\omega,\rho)=\vert\nabla_x{c}(\tau(r),\eta^1_{r}(w),z,{x}^{{\mathcal{P}},M}_{\tau^{{\mathcal{P}}}(r)-},\rho^{\mathcal{P},M}_{\tau^{{\mathcal{P}}}(r)-})-\nabla_x{c}(r,\eta^2_{r}(w),z,{x}^{M}_{r-},\rho_{r-})\vert\vert D^Z_{({\bullet},{\circ},{\diamond})}{x}^{M}_{r-}\vert_{l_2},\]    we get
\begin{eqnarray*}
H_{3,1}&\leq& \mathbb{E}(\int_0^t\int_{[0,1]\times\mathbb{R}^d}\vert\nabla_x{c}(\tau(r),\eta^1_{r}(w),z,{x}^{{\mathcal{P}},M}_{\tau^{{\mathcal{P}}}(r)-},\rho^{\mathcal{P},M}_{\tau^{{\mathcal{P}}}(r)-})\\
&-&\nabla_x{c}(r,\eta^2_{r}(w),z,{x}^{M}_{r-},\rho_{r-})\vert\vert D^Z_{({\bullet},{\circ},{\diamond})}{x}^{M}_{r-}\vert_{l_2}\mathcal{N}(dw,dz,dr))^2\\
&\leq& C\mathbb{E}\int_0^t\int_{[0,1]}[\vert \tau^{{\mathcal{P}}}(r)-r\vert^2+\vert \eta^1_{r}(w)-\eta^2_{r}(w)\vert^2\\
&+&\vert{x}^{{\mathcal{P}},M}_{\tau^{{\mathcal{P}}}(r)-}-{x}^{M}_{r-}\vert^2+(W_1(\rho^{\mathcal{P},M}_{\tau^{{\mathcal{P}}}(r)-},\rho_{r-}))^2]\vert D^Z_{({\bullet},{\circ},{\diamond})}{x}^{M}_{r-}\vert_{l_2}^2dwdr.
\end{eqnarray*} Then using (\ref{coupling}), (\ref{Coupling}), \textbf{Lemma 3.7},  and H$\ddot{o}$lder inequality with conjugates $1+\frac{\varepsilon_\ast}{2}$ and $\frac{2+\varepsilon_\ast}{\varepsilon_\ast}$,    we have 
\begin{eqnarray*}
H_{3,1}&\leq& C[\vert\mathcal{P}\vert^2+\int_0^t(\mathbb{E}\vert{x}^{\mathcal{P},M}_{\tau(r)-}-{x}_{r-}\vert^{2+\varepsilon_\ast})^{\frac{2}{2+\varepsilon_\ast}}dr\\
&+&\int_0^t(\mathbb{E}\vert{x}^{\mathcal{P},M}_{\tau(r)-}-{x}^{M}_{r-}\vert^{2+\varepsilon_\ast})^{\frac{2}{2+\varepsilon_\ast}}dr+\int_0^t(W_{2+\varepsilon_\ast}(\rho^{\mathcal{P},M}_{\tau^{{\mathcal{P}}}(r)-},\rho_{r-}))^2dr]\\
&\leq&C(\vert\mathcal{P}\vert+\varepsilon_M)^{\frac{2}{2+\varepsilon_\ast}},
\end{eqnarray*}where the last inequality is a consequence of \textbf{Lemma 2.6} and (\ref{Eul}).\\
Using \textbf{Hypothesis 2.1},   (\ref{BurkM}) with  \[\bar{\Phi}(r,w,z,\omega,\rho)=\vert\nabla_x{c}(\tau(r),\eta^1_{r}(w),z,{x}^{{\mathcal{P}},M}_{\tau^{{\mathcal{P}}}(r)-},\rho^{\mathcal{P},M}_{\tau^{{\mathcal{P}}}(r)-})\vert\vert D^Z_{({\bullet},{\circ},{\diamond})}{x}^{\mathcal{P},M}_{\tau(r)-}-D^Z_{({\bullet},{\circ},{\diamond})}{x}^{M}_{r-}\vert_{l_2},\]  we have 
\begin{eqnarray*}
H_{3,2}&\leq& \mathbb{E} (\int_0^t\int_{[0,1]\times\mathbb{R}^d}\vert\nabla_x{c}(\tau(r),\eta^1_{r}(w),z,{x}^{{\mathcal{P}},M}_{\tau^{{\mathcal{P}}}(r)-},\rho^{\mathcal{P},M}_{\tau^{{\mathcal{P}}}(r)-})\vert\vert D^Z_{({\bullet},{\circ},{\diamond})}{x}^{\mathcal{P},M}_{\tau(r)-}-D^Z_{({\bullet},{\circ},{\diamond})}{x}^{M}_{r-}\vert_{l_2}\mathcal{N}(dw,dz,dr))^2\\
&\leq& C\int_{0}^{t}\mathbb{E}\vert D^Z_{({\bullet},{\circ},{\diamond})}{x}^{\mathcal{P},M}_{\tau(r)-}-D^Z_{({\bullet},{\circ},{\diamond})}{x}^{M}_{r-}\vert_{l_2}^2dr.\end{eqnarray*}
Therefore,
\begin{eqnarray}
H_3&\leq& C[H_{3,1}+H_{3,2}]\leq C[(\vert\mathcal{P}\vert+\varepsilon_M)^{\frac{2}{2+\varepsilon_\ast}}+\int_{0}^{t}\mathbb{E}\vert D^Z_{({\bullet},{\circ},{\diamond})}{x}^{\mathcal{P},M}_{\tau(r)}-D^Z_{({\bullet},{\circ},{\diamond})}{x}^{M}_{r}\vert_{l_2}^2dr].\label{H3}
\end{eqnarray}

Combining (\ref{H1}), (\ref{H2}) and (\ref{H3}), 
\[\mathbb{E}\vert D^Z_{({\bullet},{\circ},{\diamond})}{x}^{\mathcal{P},M}_{t}-D^Z_{({\bullet},{\circ},{\diamond})}{x}^{M}_{t}\vert_{l_2}^2\leq C[(\vert\mathcal{P}\vert+\varepsilon_M)^{\frac{2}{2+\varepsilon_\ast}}+\int_{0}^{t}\mathbb{E}\vert D^Z_{({\bullet},{\circ},{\diamond})}{x}^{\mathcal{P},M}_{\tau(r)}-D^Z_{({\bullet},{\circ},{\diamond})}{x}^{M}_{r}\vert_{l_2}^2dr].\]

In a similar way, we notice by (\ref{EuD}), the isometry of the Poisson point measure $N$, and (\ref{BurkM}) with \[\bar{\Phi}(r,w,z,\omega,\rho)=\nabla_x{c}(\tau(r),\eta^1_{r}(w),z,{x}^{{\mathcal{P}},M}_{\tau^{{\mathcal{P}}}(r)-},\rho^{\mathcal{P},M}_{\tau^{{\mathcal{P}}}(r)-})D^Z_{(k_0,{i_0},{j})}{x}^{\mathcal{P},M}_{\tau(r)-}\] that \begin{eqnarray}\mathbb{E}\vert D^Z_{({\bullet},{\circ},{\diamond})}{x}^{\mathcal{P},M}_{\tau(t)}-D^Z_{({\bullet},{\circ},{\diamond})}{x}^{\mathcal{P},M}_{t}\vert_{l_2}^2\leq C\vert\mathcal{P}\vert,\label{EulD}\end{eqnarray} so \[\mathbb{E}\vert D^Z_{({\bullet},{\circ},{\diamond})}{x}^{\mathcal{P},M}_{t}-D^Z_{({\bullet},{\circ},{\diamond})}{x}^{M}_{t}\vert_{l_2}^2\leq C[(\vert\mathcal{P}\vert+\varepsilon_M)^{\frac{2}{2+\varepsilon_\ast}}+\int_{0}^{t}\mathbb{E}\vert D^Z_{({\bullet},{\circ},{\diamond})}{x}^{\mathcal{P},M}_{r}-D^Z_{({\bullet},{\circ},{\diamond})}{x}^{M}_{r}\vert_{l_2}^2dr].\]We conclude by Gronwall lemma that $\mathbb{E}\vert D^Z_{({\bullet},{\circ},{\diamond})}{x}^{\mathcal{P},M}_{t}-D^Z_{({\bullet},{\circ},{\diamond})}{x}^{M}_{t}\vert_{l_2}^2\leq C(\vert\mathcal{P}\vert+\varepsilon_M)^{\frac{2}{2+\varepsilon_\ast}}$. Finally, by a similar argument, $\mathbb{E}\vert D^{\Delta}_{({\diamond})}{x}^{\mathcal{P},M}_{t}-D^{\Delta}_{({\diamond})}{x}^{M}_{t}\vert_{\mathbb{R}^d}^2\leq C(\vert\mathcal{P}\vert+\varepsilon_M)^{\frac{2}{2+\varepsilon_\ast}}$, and we obtain what we need.

\bigskip\bigskip

\textbf{Proof of $ii)$:} We only need to prove that for any $M_1,M_2\in\mathbb{N}$ with $\varepsilon_{M_1\wedge M_2}\leq 1$ and $\vert\bar{c}(z)\vert^2\mathbbm{1}_{\{\vert z\vert>M_1\wedge M_2\}}\leq 1$, we have \begin{eqnarray}
\Vert Dx^{M_1}_t-Dx^{M_2}_t\Vert_{L^2(\Omega;l^2\times\mathbb{R}^d)}\leq (\varepsilon_{M_1}+\varepsilon_{M_2})^{\frac{1}{2+\varepsilon_\ast}}.\label{Cauchy}
\end{eqnarray} In fact, if $(Dx^M_t)_{M\in\mathbb{N}}$ is a Cauchy sequence in $L^2(\Omega;l^2\times\mathbb{R}^d)$, then it has a limit $Y$ in $L^2(\Omega;l^2\times\mathbb{R}^d)$. But when we apply \textbf{Lemma 3.3 (A)} with $F_M=X^M_t$ and $F=X_t$, we know that there exists a convex combination
$\sum\limits_{M^\prime=M}^{m_{M}}\gamma _{M^\prime}^{M}\times F^{M^\prime}_t,$
with $\gamma _{M^\prime}^{M}\geq 0,M^\prime=M,....,m_{M}$ and $%
\sum\limits_{M^\prime=M}^{m_{M}}\gamma _{M^\prime}^{M}=1$,
such that 
\[
\left\Vert \sum_{M^\prime=M}^{m_{M}}\gamma _{M^\prime}^{M}\times Dx^{M^\prime}_t-Dx_t\right\Vert _{L^2(\Omega;l^2\times\mathbb{R}^d)}\rightarrow 0,
\]as $M\rightarrow\infty$. Meanwhile, we have \[\left\Vert \sum_{M^\prime=M}^{m_{M}}\gamma _{M^\prime}^{M}\times DF^{M^\prime}_t-Y\right\Vert _{L^2(\Omega;l^2\times\mathbb{R}^d)}\leq  \sum_{M^\prime=M}^{m_{M}}\gamma _{M^\prime}^{M}\left\Vert Dx^{M^\prime}_t-Y\right\Vert _{L^2(\Omega;l^2\times\mathbb{R}^d)}\rightarrow0.\]So $Y=Dx_t$ and we conclude by passing to the limit $M_2\rightarrow\infty$ in (\ref{Cauchy}). 

Now we prove (\ref{Cauchy}). We recall the equation (\ref{DZxM}) and we write $\mathbb{E}\vert D^Z_{({\bullet},{\circ},{\diamond})}{x}^{M_1}_{t}-D^Z_{({\bullet},{\circ},{\diamond})}{x}^{M_2}_{t}\vert_{l_2}^2\leq C[O_1+O_2+O_3]$, with
\begin{eqnarray*}
O_1=\mathbb{E}\vert \int_{0}^{t}\nabla_xb(r,{x}^{M_1}_{r},\rho_{r})D^Z_{({\bullet},{\circ},{\diamond})}{x}^{M_1}_{r}dr-\int_{0}^{t}\nabla_xb(r,{x}^{M_2}_{r},\rho_{r})D^Z_{({\bullet},{\circ},{\diamond})}{x}^{M_2}_{r}dr \vert_{l_2}^2,
\end{eqnarray*}
\begin{eqnarray*}
O_2&=&\mathbb{E}\vert\mathbbm{1}_{\{0<T_{{\circ}}^{\bullet}\leq t\}}(\mathbbm{1}_{\{1\leq {\bullet}\leq M_1\}}\partial_{z_{\diamond}}{c}({T}_{\circ}^{\bullet},\eta^2_{T_{\circ}^{\bullet}}(W_{\circ}^{\bullet}),{Z}_{\circ}^{\bullet},{x}^{M_1}_{{T}_{\circ}^{\bullet}-},\rho_{{T}_{\circ}^{\bullet}-})\\
&-&\mathbbm{1}_{\{1\leq {\bullet}\leq M_2\}}\partial_{z_{\diamond}}{c}({T}_{\circ}^{\bullet},\eta^2_{T_{\circ}^{\bullet}}(W_{\circ}^{\bullet}),{Z}_{\circ}^{\bullet},{x}^{M_2}_{{T}_{\circ}^{\bullet}-},\rho_{{T}_{\circ}^{\bullet}-}))\vert_{l_2}^2,
\end{eqnarray*}
\begin{eqnarray*}
O_3&=&\mathbb{E}\vert\sum_{k=1}^{M_1}\sum_{i=1}^{J^k_t}\nabla_x{c}({T}_{i}^{k},\eta^2_{T_{i}^{k}}(W_{i}^k),{Z}_{i}^{k},{x}^{M_1}_{{T}_{i}^{k}-},\rho_{{T}_{i}^{k}-})D^Z_{({\bullet},{\circ},{\diamond})}{x}^{M_1}_{{T}_{i}^{k}-}\\
&-&\sum_{k=1}^{M_2}\sum_{i=1}^{J^k_t}\nabla_x{c}({T}_{i}^{k},\eta^2_{T_{i}^{k}}(W_{i}^k),{Z}_{i}^{k},{x}^{M_2}_{{T}_{i}^{k}-},\rho_{{T}_{i}^{k}-})D^Z_{({\bullet},{\circ},{\diamond})}{x}^{M_2}_{{T}_{i}^{k}-}\vert_{l_2}^2.
\end{eqnarray*}

Firstly, using \textbf{Hypothesis 2.1},   we have
\begin{eqnarray*}
O_1&\leq& C[\mathbb{E} \int_{0}^{t}\vert\nabla_xb(r,{x}^{M_1}_{r},\rho_{r})-\nabla_xb(r,{x}^{M_2}_{r},\rho_{r})\vert^2\vert D^Z_{({\bullet},{\circ},{\diamond})}{x}^{M_2}_{r}\vert_{l_2}^2dr\nonumber\\
&+&\mathbb{E} \int_{0}^{t}\vert\nabla_xb(r,{x}^{M_1}_{r},\rho_{r})\vert^2\vert D^Z_{({\bullet},{\circ},{\diamond})}{x}^{M_1}_{r}-D^Z_{({\bullet},{\circ},{\diamond})}{x}^{M_2}_{r}\vert_{l_2}^2dr]\nonumber\\
&\leq& C[\mathbb{E}\int_0^t
\vert{x}^{M_1}_{r}-{x}^{M_2}_{r}\vert^{2}\vert D^Z_{({\bullet},{\circ},{\diamond})}{x}^{M_2}_{r}\vert_{l_2}^2dr+ \int_{0}^{t}\mathbb{E}\vert D^Z_{({\bullet},{\circ},{\diamond})}{x}^{M_1}_{r}-D^Z_{({\bullet},{\circ},{\diamond})}{x}^{M_2}_{r}\vert_{l_2}^2dr].
\end{eqnarray*}Then by \textbf{Lemma 3.7}, H$\ddot{o}$lder inequality with conjugates $1+\frac{\varepsilon_\ast}{2}$ and $\frac{2+\varepsilon_\ast}{\varepsilon_\ast}$, and \textbf{Lemma 2.6},   we obtain
\begin{eqnarray}
O_1&\leq& C[
\int_0^t(\mathbb{E}\vert{x}^{M_1}_{r}-{x}^{M_2}_{r}\vert^{2+\varepsilon_\ast})^{\frac{2}{2+\varepsilon_\ast}}dr+ \int_{0}^{t}\mathbb{E}\vert D^Z_{({\bullet},{\circ},{\diamond})}{x}^{M_1}_{r}-D^Z_{({\bullet},{\circ},{\diamond})}{x}^{M_2}_{r}\vert_{l_2}^2dr]\nonumber\\
&\leq& C[(\varepsilon_{M_1}+\varepsilon_{M_2})^{\frac{2}{2+\varepsilon_\ast}}+\int_{0}^{t}\mathbb{E}\vert D^Z_{({\bullet},{\circ},{\diamond})}{x}^{M_1}_{r}-D^Z_{({\bullet},{\circ},{\diamond})}{x}^{M_2}_{r}\vert_{l_2}^2dr].\label{O1}
\end{eqnarray}

Secondly, using \textbf{Hypothesis 2.1}, the isometry of the Poisson point measure $\mathcal{N}$,   we have
\begin{eqnarray*}
O_2&\leq&C[ \mathbb{E}\sum_{k=1}^{M_1}\sum_{i=1}^{J^k_t}\sum_{j=1}^d\vert\partial_{z_{j}}{c}({T}_{i}^{k},\eta^2_{T_{i}^{k}}(W_{i}^k),{Z}_{i}^{k},{x}^{M_1}_{{T}_{i}^{k}-},\rho_{{T}_{i}^{k}-})-\partial_{z_{j}}{c}({T}_{i}^{k},\eta^2_{T_{i}^{k}}(W_{i}^k),{Z}_{i}^{k},{x}^{M_2}_{{T}_{i}^{k}-},\rho_{{T}_{i}^{k}-})\vert^2\nonumber\\
&+&\mathbb{E}\sum_{k=M_1\wedge M_2}^{M_1\vee M_2}\sum_{i=1}^{J^k_t}\sum_{j=1}^d\vert\partial_{z_{j}}{c}({T}_{i}^{k},\eta^2_{T_{i}^{k}}(W_{i}^k),{Z}_{i}^{k},{x}^{M_2}_{{T}_{i}^{k}-},\rho_{{T}_{i}^{k}-})\vert^2]\nonumber\\
&\leq& C[\mathbb{E}\sum_{k=1}^M\sum_{i=1}^{J^k_t}\vert\bar{c}(Z^k_i)\vert^2\vert{x}^{M_1}_{{T}_{i}^{k}-}-{x}^{M_2}_{{T}_{i}^{k}-}\vert^2+\mathbb{E}\sum_{k=M_1\wedge M_2}^{M_1\vee M_2}\sum_{i=1}^{J^k_t}\vert\bar{c}(Z^k_i)\vert^2]\nonumber\\
&\leq& C[\mathbb{E}\int_0^t\int_{[0,1]\times\mathbb{R}^d}\vert\bar{c}(z)\vert^2\vert{x}^{M_1}_{r-}-{x}^{M_2}_{r-}\vert^2\mathcal{N}(dw,dz,dr)+\mathbb{E}\int_0^t\int_{[0,1]}\int_{\{\vert z\vert>M_1\wedge M_2\}}\vert\bar{c}(z)\vert^2\mathcal{N}(dw,dz,dr)]\nonumber\\
&=& C[\mathbb{E}\int_0^t\int_{[0,1]\times\mathbb{R}^d}\vert\bar{c}(z)\vert^2\vert{x}^{M_1}_{r-}-{x}^{M_2}_{r-}\vert^2dw\mu(dz)dr+\mathbb{E}\int_0^t\int_{[0,1]}\int_{\{\vert z\vert>M_1\wedge M_2\}}\vert\bar{c}(z)\vert^2dw\mu(dz)dr].
\end{eqnarray*}
Then by  H$\ddot{o}$lder inequality with conjugates $1+\frac{\varepsilon_\ast}{2}$ and $\frac{2+\varepsilon_\ast}{\varepsilon_\ast}$, \textbf{Hypothesis 2.1} and \textbf{Lemma 2.6}, \begin{eqnarray}
O_2&\leq& C[\int_0^t(\mathbb{E}\vert{x}^{M_1}_{r}-{x}^{M_2}_{r}\vert^{2+\varepsilon_\ast})^{\frac{2}{2+\varepsilon_\ast}}dr+\varepsilon_{M_1\wedge M_2}]\nonumber\\
&\leq& C(\varepsilon_{M_1}+\varepsilon_{M_2})^{\frac{2}{2+\varepsilon_\ast}}.\label{O2}
\end{eqnarray}

Thirdly, we write $O_3\leq C[O_{3,1}+O_{3,2}+O_{3,3}]$, where
\[O_{3,1}=\mathbb{E}(\sum_{k=M_1\wedge M_2}^{M_1\vee M_2}\sum_{i=1}^{J^k_t}\vert\nabla_x{c}({T}_{i}^{k},\eta^2_{T_{i}^{k}}(W_{i}^k),{Z}_{i}^{k},{x}^{M_1\vee M_2}_{{T}_{i}^{k}-},\rho_{{T}_{i}^{k}-})\vert\vert D^Z_{({\bullet},{\circ},{\diamond})}{x}^{M_1\vee M_2}_{{T}_{i}^{k}-}\vert_{l_2})^2,\]
\begin{eqnarray*}
O_{3,2}&=& \mathbb{E} (\sum_{k=1}^{M_1\wedge M_2}\sum_{i=1}^{J_t^k}\vert\nabla_x{c}({T}_{i}^{k},\eta^2_{T_{i}^{k}}(W_{i}^k),{Z}_{i}^{k},{x}^{M_1}_{{T}_{i}^{k}-},\rho_{{T}_{i}^{k}-})\\
&-&\nabla_x{c}({T}_{i}^{k},\eta^2_{T_{i}^{k}}(W_{i}^k),{Z}_{i}^{k},{x}^{M_2}_{{T}_{i}^{k}-},\rho_{{T}_{i}^{k}-})\vert\vert D^Z_{({\bullet},{\circ},{\diamond})}{x}^{M_2}_{{T}_{i}^{k}-}\vert_{l_2})^2,
\end{eqnarray*}
\[O_{3,3}=\mathbb{E} (\sum_{k=1}^{M_1\wedge M_2}\sum_{i=1}^{J^k_t}\vert\nabla_x{c}({T}_{i}^{k},\eta^2_{T_{i}^{k}}(W_{i}^k),{Z}_{i}^{k},{x}^{M_1}_{{T}_{i}^{k}-},\rho_{{T}_{i}^{k}-})\vert\vert D^Z_{({\bullet},{\circ},{\diamond})}{x}^{M_1}_{{T}_{i}^{k}-}-D^Z_{({\bullet},{\circ},{\diamond})}{x}^{M_2}_{{T}_{i}^{k}-}\vert_{l_2})^2.\]
Using \textbf{Hypothesis 2.1}, \textbf{Lemma 3.7}, (\ref{Burk*}) with \[\bar{\Phi}(r,w,z,\omega,\rho)=\vert\nabla_x{c}(r,\eta^2_{r}(w),z,{x}^{M_1\vee M_2}_{r-},\rho_{r-})\vert\vert D^Z_{({\bullet},{\circ},{\diamond})}{x}^{M_1\vee M_2}_{r-}\vert_{l_2},\]   we get
\begin{eqnarray*}
O_{3,1}&\leq& \mathbb{E}(\int_0^t\int_{[0,1]}\int_{\{\vert z\vert>M_1\wedge M_2\}}\vert\nabla_x{c}(r,\eta^2_{r}(w),z,{x}^{M_1\vee M_2}_{r-},\rho_{r-})\vert\vert D^Z_{({\bullet},{\circ},{\diamond})}{x}^{M_1\vee M_2}_{r-}\vert_{l_2}\mathcal{N}(dw,dz,dr))^2\\
&\leq& C[(\int_{\{\vert z\vert>M_1\wedge M_2\}}\bar{c}(z)\mu(dz))^2+\int_{\{\vert z\vert>M_1\wedge M_2\}}\vert\bar{c}(z)\vert^2\mu(dz)]\\
&=& C\varepsilon_{M_1\wedge M_2}.
\end{eqnarray*}Using \textbf{Hypothesis 2.1}, \textbf{Lemma 3.7},  (\ref{BurkM}) with \[\bar{\Phi}(r,w,z,\omega,\rho)=\vert\nabla_x{c}(r,\eta^2_{r}(w),z,{x}^{M_1}_{r-},\rho_{r-})-\nabla_x{c}(r,\eta^2_{r}(w),z,{x}^{M_2}_{r-},\rho_{r-})\vert\vert D^Z_{({\bullet},{\circ},{\diamond})}{x}^{M_2}_{r-}\vert_{l_2},\]   by \textbf{Lemma 2.6},   and H$\ddot{o}$lder inequality with conjugates $1+\frac{\varepsilon_\ast}{2}$ and $\frac{2+\varepsilon_\ast}{\varepsilon_\ast}$,  we have 
\begin{eqnarray*}
O_{3,2}&\leq& \mathbb{E}(\int_0^t\int_{[0,1]\times\mathbb{R}^d}\vert\nabla_x{c}(r,\eta^2_{r}(w),z,{x}^{M_1}_{r-},\rho_{r-})-\nabla_x{c}(r,\eta^2_{r}(w),z,{x}^{M_2}_{r-},\rho_{r-})\vert\vert D^Z_{({\bullet},{\circ},{\diamond})}{x}^{M_2}_{r-}\vert_{l_2}\mathcal{N}(dw,dz,dr))^2\\
&\leq& C\mathbb{E}\int_0^t\int_{[0,1]}\vert{x}^{M_1}_{r-}-{x}^{M_2}_{r-}\vert^2\vert D^Z_{({\bullet},{\circ},{\diamond})}{x}^{M_2}_{r-}\vert_{l_2}^2dwdr\\
&\leq& C\int_0^t(\mathbb{E}\vert{x}^{M_1}_{r-}-{x}^{M_2}_{r-}\vert^{2+\varepsilon_\ast})^{\frac{2}{2+\varepsilon_\ast}}dr\\
&\leq&C(\varepsilon_{M_1}+\varepsilon_{M_2})^{\frac{2}{2+\varepsilon_\ast}}.
\end{eqnarray*}Using \textbf{Hypothesis 2.1},  (\ref{BurkM}) with  \[\bar{\Phi}(r,w,z,\omega,\rho)=\vert\nabla_x{c}(r,\eta^2_{r}(w),z,{x}^{M_1}_{r-},\rho_{r-})\vert\vert D^Z_{({\bullet},{\circ},{\diamond})}{x}^{M_1}_{r-}-D^Z_{({\bullet},{\circ},{\diamond})}{x}^{M_2}_{r-}\vert_{l_2},\]  we have 
\begin{eqnarray*}
O_{3,3}&\leq& \mathbb{E} (\int_0^t\int_{[0,1]\times\mathbb{R}^d}\vert\nabla_x{c}(r,\eta^2_{r}(w),z,{x}^{M_1}_{r-},\rho_{r-})\vert\vert D^Z_{({\bullet},{\circ},{\diamond})}{x}^{M_1}_{r-}-D^Z_{({\bullet},{\circ},{\diamond})}{x}^{M_2}_{r-}\vert_{l_2}\mathcal{N}(dw,dz,dr))^2\\
&\leq& C\int_{0}^{t}\mathbb{E}\vert D^Z_{({\bullet},{\circ},{\diamond})}{x}^{M_1}_{r-}-D^Z_{({\bullet},{\circ},{\diamond})}{x}^{M_2}_{r-}\vert_{l_2}^2dr.\end{eqnarray*}
Therefore,
\begin{eqnarray}
O_3&\leq& C[O_{3,1}+O_{3,2}+O_{3,3}]\leq C[(\varepsilon_{M_1}+\varepsilon_{M_2})^{\frac{2}{2+\varepsilon_\ast}}+\int_{0}^{t}\mathbb{E}\vert D^Z_{({\bullet},{\circ},{\diamond})}{x}^{M_1}_{r}dr-D^Z_{({\bullet},{\circ},{\diamond})}{x}^{M_2}_{r}\vert_{l_2}^2dr].\quad\label{O3}
\end{eqnarray}

Combining (\ref{O1}), (\ref{O2}) and (\ref{O3}), 
\[\mathbb{E}\vert D^Z_{({\bullet},{\circ},{\diamond})}{x}^{M_1}_{t}-D^Z_{({\bullet},{\circ},{\diamond})}{x}^{M_2}_{t}\vert_{l_2}^2\leq C[(\varepsilon_{M_1}+\varepsilon_{M_2})^{\frac{2}{2+\varepsilon_\ast}}+\int_{0}^{t}\mathbb{E}\vert D^Z_{({\bullet},{\circ},{\diamond})}{x}^{M_1}_{r}-D^Z_{({\bullet},{\circ},{\diamond})}{x}^{M_2}_{r}\vert_{l_2}^2dr].\]
So we conclude by Gronwall lemma that $\mathbb{E}\vert D^Z_{({\bullet},{\circ},{\diamond})}{x}^{M_1}_{t}-D^Z_{({\bullet},{\circ},{\diamond})}{x}^{M_2}_{t}\vert_{l_2}^2\leq C(\varepsilon_{M_1}+\varepsilon_{M_2})^{\frac{2}{2+\varepsilon_\ast}}$.

Finally, we recall by (\ref{aMt}) that $a^M_{T}=\sqrt{T\int_{\{\vert z\vert>M\}}\underline{c}(z)  \mu(dz)}$ and by \textbf{Hypothesis 2.3} that $\underline{c}(z)\leq \vert\bar{c}(z)\vert^2$. We notice that \[\mathbb{E}\vert a^{M_1}_T\bm{e_{\diamond}}-a^{M_2}_T\bm{e_{\diamond}}\vert_{\mathbb{R}^d}^2\leq C\mathbb{E}\vert a^{M_1}_T-a^{M_2}_T\vert^2\leq C\int_{\{\vert z\vert>{M_1\wedge M_2}\}}\underline{c}(z)  \mu(dz)\leq \varepsilon_{M_1\wedge M_2}.\] Then by a similar argument as above, $\mathbb{E}\vert D^{\Delta}_{({\diamond})}{x}^{M_1}_{t}-D^{\Delta}_{({\diamond})}{x}^{M_2}_{t}\vert_{\mathbb{R}^d}^2\leq C(\varepsilon_{M_1}+\varepsilon_{M_2})^{\frac{2}{2+\varepsilon_\ast}}$, and we obtain (\ref{Cauchy}).

\bigskip\bigskip

\textbf{Proof of $iii)$:} $iii)$ is an immediate consequence of $i)$ and $ii)$.

\end{proof}

\bigskip

\textbf{Acknowledgment} The author thanks Prof. Vlad Bally for his kind and patient guidance and many useful suggestions.


\bigskip\bigskip





\begin{thebibliography}{99}  
\bibitem[1]{ref33}S. Albeverio, B. R$\ddot{u}$diger, and P. Sundar: The Enskog process. \textit{J. Stat. Phys.} \textbf{167}(1):90–122 (2017).
\bibitem[2]{ref27}R. Alexandre: A review of Boltzmann equation with singular kernels. \textit{Kinet. Relat. Models}, \textbf{2}(4), 551-646 (2009).
\bibitem[3]{ref1}A. Alfonsi, V. Bally
: Construction of Boltzmann and McKean Vlasov type flows (the sewing lemma approach). arXiv:2105.12677 [math.PR](2021).
\bibitem[4]{ref50}A. Alfonsi, B. Jourdain, A. Kohatsu-Higa: Optimal transport bounds between the time-marginals of a multidimensional diffusion and its Euler scheme. \textit{Electronic Journal of Probability} \textbf{20}, 1-31 06 (2015).
\bibitem[5]{ref44}F. Antonelli, A. Kohatsu-Higa: Rate of convergence of a particle method to the solution of the McKean-Vlasov equation. \textit{The Annals of Applied Probability}, \textbf{12}(2): 423-476 (2002).
\bibitem[6]{ref51}D. Applebaum: \textit{L\'{e}vy Processes and Stochastic Calculus} (2nd ed., Cambridge Studies in Advanced Mathematics). Cambridge: Cambridge University Press. (2009) doi:10.1017/CBO9780511809781.
\bibitem[7]{ref2}V. Bally, L. Caramellino, G. Poly: Regularization lemmas and convergence in total variation. \textit{Electron. J. Probab.} \textbf{25} 1-20. (2020). 
\bibitem[8]{ref3}V. Bally, E. Cl\'{e}ment: Integration by parts formula and applications to equations with jumps. \textit{Probab. Th. Rel. Fields}, \textbf{151}, 613-657 (2011).
\bibitem[9]{ref23}V. Bally, N. Fournier: Regularization properties of the 2D homogeneous Boltzmann equation without cutoff. \textit{Probab. Theory Related Fields}, \textbf{151}(3-4), 659-704 (2011).
\bibitem[10]{ref4}V. Bally, Y. Qin: Total variation distance between a jump-equation and its Gaussian approximation. \textit{Stoch PDE: Anal Comp}. (2022).
\bibitem[11]{ref52}V. Bally, D. Talay: The Law of the Euler scheme for stochastic differential equations: I. convergence rate of the distribution function. [Research Report] RR-2244, INRIA. (1994). ⟨inria-00074427⟩
\bibitem[12]{ref43}J. Bao, C. Reisinger, P. Ren, et al: First-order convergence of Milstein schemes for McKean–Vlasov equations and interacting particle systems. \textit{Proceedings of the Royal Society A}, \textbf{477}(2245): 20200258 (2021).
\bibitem[13]{ref5}K. Bichteler, J. B. Gravereaux, J. Jacod: \textit{Malliavin calculus for processes with jumps}. Gordon and Breach, (1987).
\bibitem[14]{ref24}R. Carmona, F. Delarue:  \textit{Probability theory of mean field games with applications}. Springer \textit{Probability Theory and Stochastic Modelling} Vol.\textbf{83}  (2018).
\bibitem[15]{ref25}C. Cercignani: \textit{The Boltzmann equation and its applications}. Springer-Verlag \textit{Applied Mathematical Sciences}, Vol.\textbf{67} (1988).
\bibitem[16]{ref6}R. Cont, P. Tankov: \textit{Finacial modelling with jump processes}. Chapman \& Hall/CRC (2004).
\bibitem[17]{ref26}L. Desvillettes, C. Graham, S. M\'{e}l\'{e}ard: Probabilistic interpretation and numerical approximation of a Kac equation without cutoff. \textit{Stochastic Process. Appl.}, \textbf{84}(1), 115-135 (1999).
\bibitem[18]{ref29}N. Fournier, A. Guillin: From a Kac-like particle system to the Landau equation for hard potentials and Maxwell molecules. \textit{Ann. Sci. \'{E}c. Norm. Sup\'{e}r.} (4), \textbf{50}(1), 157-199 (2017).
\bibitem[19]{ref30}N. Fournier, S. Mischler: Rate of convergence of the Nanbu particle system for hard potentials and Maxwell molecules. \textit{Ann. Proba.}, \textbf{44}(1), 589-627 (2016).
\bibitem[20]{ref34}M. Friesen, B. R$\ddot{u}$diger, and P. Sundar: The Enskog process for hard and soft potentials. \textit{NoDEA Nonlinear
Differential Equations Appl.} \textbf{26}(3):Paper No. 20, 42 (2019).
\bibitem[21]{ref35}M. Friesen, B. R$\ddot{u}$diger, and P. Sundar: On uniqueness and stability for the Enskog equation (2020).
\bibitem[22]{ref38}E. Gobet, G. Pag$\grave{e}$s, H. Pham and J. Printems: Discretization and simulation for a class of SPDEs with applications to Zakai and McKean-Vlasov equations. (2005). ⟨hal-00003917v1⟩
\bibitem[23]{ref7}C. Graham: Mckean-Vlasov Ito-Skorohod equations, and nonlinear diffusions with discrete jump sets. \textit{Stochastic Processes and their Applications} \textbf{40} 69-82 (1992).
\bibitem[24]{ref42}A. L. Haji-Ali, R. Tempone: Multilevel and Multi-index Monte Carlo methods for the McKean–Vlasov equation. \textit{Statistics and Computing}, \textbf{28}(4): 923-935 (2018).
\bibitem[25]{ref8}N. Ikeda, S. Watanabe: \textit{Stochastic differential equations and diffusion processes}. 2nd ed. Amsterdam, Netherlands, North Holland, (1989).
\bibitem[26]{ref9}Y. Ishikawa: \textit{Stochastic Calculus of Variations for Jump Processes}, Berlin, Boston: De Gruyter. (2013).
\bibitem[27]{ref48}B. Jourdain and A. Kohatsu-Higa: A review of recent results on approximation of solutions of stochastic differential equations. Proceedings of the Workshop on Stochastic Analysis with Financial Applications: Hong Kong (2009). Birkhauser (2011).
\bibitem[28]{ref45}A. Kohatsu-Higa: The Euler approximation for stochastic differential equations with boundary conditions. Proceedings of the Workshop on Turbulent Diffusion and Related Problems in Stochastic Numerics. The Institute of Statistical Mathematics, Tokyo (1996).
\bibitem[29]{ref47}A. Kohatsu-Higa, S. Ogawa: Monte Carlo methods weak rate of convergence for an Euler scheme of nonlinear SDE's. \textit{Monte Carlo Methods and Its Applications}, vol \textbf{3}, 327-345 (1997).
\bibitem[30]{ref46}A. Kohatsu-Higa, P. Protter: The Euler scheme for SDE's driven by semimartingales. In Stochastic analysis on infinite dimensional spaces. H. Kunita and H.Kuo (Eds.), 141-151, Pitman Research Notes in Mathematics Series ,vol. \textbf{310} (1994).
\bibitem[31]{ref49}A. Kohatsu-Higa and P. Tankov: Jump-adapted discretization schemes for L\'{e}vy-driven SDEs. \textit{Stochastic Processes and their Applications}, Vol. \textbf{120}, 2258-2285 (2010).
\bibitem[32]{ref10}A. M. Kulik: Malliavin calculus for L\'{e}vy processes with arbitrary L\'{e}vy measures. \textit{Theor. Probability ad Math. Statist.} No.\textbf{72}, 75-92 (2006).
\bibitem[33]{ref11}A. M. Kulik: Stochastic calculus of variations for general L\'{e}vy processes and its applications to jump-type SDE's with non-degenerated drift. arXiv:math/0606427 (2007).
\bibitem[34]{ref12}H. Kunita: Stochastic differential equations based on Lévy processes and stochastic flows of diffeomorphisms. In: Rao, MM,ed. \textit{Real and stochastic
analysis}. Boston, USA, Birkha\"{a}user, 305-373 (2004).
\bibitem[35]{ref13}H. Kunita: \textit{Stochastic flows and jump-diffusions}. Springer,(2019).
\bibitem[36]{ref14}B. Lapeyre, $\acute{E}$. Pardoux and R. Sentis: \textit{M$\acute{e}$thodes de Monte-Carlo pour les $\acute{e}$quations de transport et de diffusion}. Springer-Verlag, Berlin, (1998).
\bibitem[37]{ref19}E. Mariucci, M. Reiß: Wasserstein and total variation distance between marginals of L\'{e}vy processes. \textit{Electronic Journal of Statistics} \textbf{12}, 2482-2514 (2018).
\bibitem[38]{ref31}S. M\'{e}l\'{e}ard: Asymptotic behaviour of some interacting particle systems; McKean-Vlasov and Boltzmann models.
\textit{Probabilistic models for nonlinear partial differential equations} volume \textbf{1627} of
Lecture Notes in Math., pages 42–95. Springer, Berlin (1996).
\bibitem[39]{ref41}M. A. Mezerdi: On the convergence of carath\'{e}odory numerical scheme for Mckean-Vlasov equations. \textit{Stochastic Analysis and Applications}, \textbf{39}(5): 804-818 (2021).
\bibitem[40]{ref39}R. Naito, T. Yamada: A higher order weak approximation of McKean–Vlasov type SDEs. \textit{BIT Computer Science and Numerical Mathematics} United States (2022).
\bibitem[41]{ref15}D. Nualart: \textit{The Malliavin calculus and related topics}. Springer-Verlag, (2006).
\bibitem[42]{ref16}P. Protter and D. Talay: The Euler scheme for L\'{e}vy driven
stochastic differential equations. \textit{Ann. Probab.} Vol \textbf{25}, No.1, pg
393-423 (1997).
\bibitem[43]{ref36}K. Sato. L\'{e}vy processes and infinitely divisible distributions. \textit{Cambridge University press}, Cambridge (1999).
\bibitem[44]{ref37}K. Sato. Basic results on L\'{e}vy processes. In O.E. Barndorff-Nielsen, T. Mikosch, and S.I. Resnick,
editors, \textit{L\'{e}vy processes. Theory and applications}, pages 3–37. Birkh$\ddot{a}$user (2001).
\bibitem[45]{ref17}Y. Song and X. Zhang: Regularity of density for SDEs driven by degenerate L\'{e}vy noises. arXiv:1401.4624 (2014).
\bibitem[46]{ref32}A.-S. Sznitman: \'{E}quations de type de Boltzmann, spatialement homog$\grave{e}$nes. \textit{Z. Wahrsch. Verw. Gebiete},
\textbf{66}(4):559–592 (1984).
\bibitem[47]{ref40}L. Szpruch, S. Tan, A. Tse:
Iterative multilevel particle approximation for
Mckean–Vlasov SDEs. \textit{The Annals of Applied Probability}
Vol. \textbf{29}, No. 4, 2230–2265 (2019).
\bibitem[48]{ref20}H. Tanaka: Probabilistic treatment of the Boltzmann equation of Maxwellian molecules. \textit{Z. Wahrsch. Verw. Gebiete.}  \textbf{46}(1), 67-105 (1978).
\bibitem[49]{ref21}H. Tanaka: Stochastic differential equation corresponding to the spatially homogeneous Boltzmann equation of Maxwellian and noncutoff type. \textit{J. Fac. Sci. Univ. Tokyo Sect. IA Math.}  \textbf{34}(2), 351-369 (1987).
\bibitem[50]{ref28}C. Villani: On a new class of weak solutions to the spatially homogeneous Boltzmann and Landau equations. \textit{Arch. Rational Mech. Anal.}, \textbf{143}(3), 273-307 (1998).
\bibitem[51]{ref22}C. Villani:  \textit{Optimal Transport} Springer-Verlag, (2009).
\bibitem[52]{ref18}X. Zhang: Densities for SDEs driven by degenerate $\alpha$-stable processes. \textit{Ann. Probab.} Vol \textbf{42}, No.5, 1885-1910 (2014).
\end{thebibliography}
\end{document}